\documentclass[A4paper, 12pt]{article}
\usepackage[utf8]{inputenc}
\usepackage[OT2, T1]{fontenc}
\usepackage{cite}
\usepackage{amsmath}
\usepackage{amsfonts}
\usepackage{amssymb}
\usepackage{amsxtra}
\usepackage{mathrsfs}
\usepackage{tensor}
\usepackage[thmmarks, amsmath, thref, amsthm]{ntheorem}
\usepackage[nottoc]{tocbibind}
\usepackage[pagebackref=true]{hyperref}

\usepackage{graphicx}
\usepackage[all]{xy}
\usepackage[portrait, top=2cm, bottom=2cm, left=2cm, right=2cm]{geometry}
\usepackage{enumerate}

\numberwithin{equation}{section}

\theoremstyle{plain}
\newtheorem{thm}{Theorem}[section]
\newtheorem{lemm}[thm]{Lemma}
\newtheorem{prop}[thm]{Proposition}

\newtheorem{customthm}{Theorem}

\theoremstyle{definition}
\newtheorem{defn}[thm]{Definition}
\newtheorem{cons}[thm]{Construction}
\newtheorem{remk}[thm]{Remarks}

\newcommand{\ol}[1]{\overline{#1}}
\newcommand{\wh}[1]{\widehat{#1}}

\newcommand{\Aut}{\operatorname{Aut}}
\newcommand{\Br}{\operatorname{Br}}
\newcommand{\Coker}{\operatorname{Coker}}
\renewcommand{\div}{\operatorname{div}}
\newcommand{\Div}{\operatorname{Div}}
\newcommand{\dtype}{\mathbf{dtype}}
\newcommand{\DType}{\operatorname{DType}}
\newcommand{\et}{\textnormal{\'et}}
\newcommand{\ev}{\operatorname{ev}}
\newcommand{\Ext}{\operatorname{Ext}}
\newcommand{\Gal}{\operatorname{Gal}}
\newcommand{\gr}{\operatorname{gr}}
\newcommand{\Hom}{\operatorname{Hom}}
\newcommand{\hor}{\operatorname{hor}}
\newcommand{\id}{\operatorname{id}}
\newcommand{\Img}{\operatorname{Img}}
\newcommand{\inn}{\operatorname{inn}}
\newcommand{\Inn}{\operatorname{Inn}}
\newcommand{\Ker}{\operatorname{Ker}}
\newcommand{\lien}{\mathbf{lien}}
\newcommand{\Mor}{\operatorname{Mor}}
\newcommand{\Out}{\operatorname{Out}}
\newcommand{\Pic}{\operatorname{Pic}}
\newcommand{\pr}{\operatorname{pr}}
\newcommand{\SAut}{\operatorname{SAut}}
\newcommand{\Spec}{\operatorname{Spec}}
\newcommand{\SOut}{\operatorname{SOut}}
\newcommand{\trans}{\mathbf{trans}}
\newcommand{\triv}{\operatorname{triv}}
\newcommand{\type}{\mathbf{type}}
\newcommand{\Unit}{\operatorname{U}}
\newcommand{\UPic}{\operatorname{UPic}}
\newcommand{\ver}{\operatorname{ver}}
\newcommand{\xtype}{\mathbf{xtype}}

\newcommand{\iDiv}{\mathscr{D}\kern -.5pt iv}
\newcommand{\iExt}{\mathscr{E}\kern -.5pt xt}
\newcommand{\iHom}{\mathscr{H}\kern -.5pt om}

\newcommand{\Gbb}{\mathbb{G}}
\newcommand{\Rbb}{\mathbb{R}}

\newcommand{\Dcal}{\mathcal{D}}
\newcommand{\Gcal}{\mathcal{G}}
\newcommand{\Fcal}{\mathcal{F}}

\renewcommand{\H}{\mathrm{H}}
\newcommand{\Z}{\mathrm{Z}}

\newcommand{\Hscr}{\mathscr{H}}

\DeclareSymbolFont{cyrletters}{OT2}{wncyr}{m}{n}
\DeclareMathSymbol{\Be}{\mathalpha}{cyrletters}{"42}
\DeclareMathSymbol{\Sha}{\mathalpha}{cyrletters}{"58}

\newcommand{\sqdot}{\mathbin{\vcenter{\hbox{\rule{.3ex}{.3ex}}}}}

\title{Non-abelian descent types}
\author{Nguy$\tilde{\hat{\text{e}}}$n M$\d{\text{a}}$nh Linh}
\date{\today}

\begin{document}
\maketitle

\begin{abstract}
    We present the notion of non-abelian descent type, which classifies torsors up to twisting by a Galois cocycle. This relies on the previous construction of kernels and non-abelian Galois 2-cohomology due to Springer and Borovoi. The necessity of descent types arises in the context of the descent theory where no torsors are given {\em a priori}, for example, when we wish to study the arithmetic properties such as the Brauer--Manin obstruction to the Hasse principle on homogeneous spaces without rational points. This new definition also unifies the types by Colliot-Th\'el\`ene--Sansuc, the extended types by Harari--Skorobogatov, and the finite descent type by Harpaz--Wittenberg.
\end{abstract}


\tableofcontents

\section{Introduction} \label{sec:Introduction}

\subsection{Context of the problem} \label{subsec:Context}

Let $X$ be a (smooth and geometrically integral) variety defined over a number field $k$. A fundamental question in arithmetic geometry is to ask whether $X(k) \neq \varnothing$. An obvious necessary condition is that the set $X(k_\Omega):=\prod_{v \in \Omega}X(k_v)$ is non-empty, where $\Omega$ denotes the set of places of $k$. If this condition is also sufficient (that is, if either $X(k_\Omega) = \varnothing$ or $X(k) \neq \varnothing$), we say that $X$ satisfies the {\em Hasse principle}. This is the case, for example, when $X$ is a projective quadric hypersurface (Hasse--Minkowski) or a principal homogeneous space of a simply connected semisimple linear algebraic group (Kneser--Harder--Chernousov). In general, the Hasse principle does not always hold; counterexamples exist already within the class of projective cubic surface or that of principal homogeneous spaces of a torus.

In 1971, Manin \cite{Manin1971Brauer} introduced a pairing between $X(k_\Omega)$ and the Brauer--Grothendieck group $\Br(X):=\H^2_{\et}(X,\Gbb_m)$. By the global reciprocity law (Albert--Brauer--Hasse--Noether), this pairing cuts out a subset $X(k_\Omega)^{\Br}$ of $X(k_\Omega)$ containing (the diagonal image) of $X(k)$. The emptiness of $X(k_\Omega)^{\Br}$ is often referred as the {\em Brauer--Manin obstruction to the Hasse principle}. It is natural to ask if this obstruction is the only one ({\em i.e.}, if the non-vacuity of the Brauer--Manin set $X(k_\Omega)^{\Br}$ is sufficient for the existence of $k$-rational points on $X$). This is still too much to expect. Indeed, Skorobogatov constructed a bielliptic surface which is a counterexample to the Hasse principle not explained by the Brauer--Manin obstruction \cite{Skorobogatov1999Beyond}. 

A conjecture of Colliot-Th\'el\`ene predicts that the Brauer--Manin obstruction is the only for smooth proper {\em rationally connected} (a notion due to Campana, Koll\'ar, Miyaoka, and Mori) varieties (see \cite{CT2003Rational}). Examples of rationally connected varieties are (geometrically) unirational varieties such as homogeneous spaces of connected linear algebraic groups. A method to attack Colliot-Th\'el\`ene conjecture is the {\em descent theory}. Its idea is to reduce an arithmetic problem (such as the Brauer--Manin obstruction) on $X$ to the same problem on its {\em descent varieties}. More precisely, let $f: Y \to X$ be a {\em torsor} under a linear algebraic group $G$, that is, $Y$ is equipped with a right action $Y \times_k G \to Y$ of $G$ such that the morphism $f$ is fppf and $G$ acts simply transitively on its fibres. For each Galois cocycle $\sigma: \Gal(\bar{k}/k) \to G(\bar{k})$, we may perform an operation called {\em twisting} to obtain a torsor $f^{\sigma}: Y^{\sigma} \to X$ under an inner $k$-form $G^{\sigma}$ of $G$ (see paragraph \ref{subsec:Notation} below). The set of $k$-rational points of $X$ has the following partition (see {\em e.g.} \cite[p. 22]{Skorobogatov2001Torsors}):
    \begin{equation*}
        X(k) = \bigsqcup_{[\sigma] \in \H^1(k,G)} f^{\sigma}(Y^{\sigma}(k)),
    \end{equation*}
where the notation $[\sigma] \in \H^1(k,G)$ means $[\sigma]$ is a cohomology class represented by a Galois cocycle $\sigma: \Gal(\bar{k}/k) \to G(\bar{k})$. If we could show a similar identity for the Brauer--Manin sets, namely
    \begin{equation*}
        X(k_\Omega)^{\Br} = \ol{\bigcup_{[\sigma] \in \H^1(k,G)} f^{\sigma}(Y^{\sigma}(k_\Omega)^{\Br})},
    \end{equation*}
in the topological space $X(k_\Omega)$ (equipped with the product of $v$-adic topologies), then the Brauer--Manin obstruction to the Hasse principle is the only obstruction for $X$, as soon as it is the case for the twisted torsors $Y^{\sigma}$. The problem with this approach is that sometimes we need descent theory when no torsors are given {\em a priori}. In other words, the {\em existence} of torsors satisfying certain conditions has to be taken into account. For example, if $X$ is a homogeneous space of a connected linear algebraic group $G'$, then we do not have a torsor $G' \to X$ unless $X$ already has a $k$-rational point! Worse, the stabilizer $\bar{G}$ of a geometric point $x \in X(\bar{k})$ needs not be defined over $k$.

The above difficulty has motivated the author to define a notion weaker than that of torsor, the {\em descent type}, which is the main subject of study in this article. {\em Grosso modo}, a descent type $\lambda$ on $X$ consists of a Galois equivariant torsor $f: \bar{Y} \to X \times_k \bar{k}$ under an algebraic $\bar{k}$-group $\bar{G}$ and a class $\eta$ in a non-abelian 2-cohomology set. We want our definition of descent type to entail the following features.
\begin{itemize}
    \item The class $\eta$ is the obstruction to lifting $f$ to a torsor $Y \to X$ under a $k$-form $G$ of $\bar{G}$.

    \item Two $X$-torsors of the same descent type $\lambda$ differ by twisting by a Galois cocycle.

    \item If $X(k) \neq \varnothing$, then any descent type can be lifted to a torsor.
\end{itemize}

In the original descent theory of Colliot-Th\'el\`ene--Sansuc \cite{CTS1987Descente}, the object having the above properties is the {\em type}. There, they consider torsors under a torus $T$, with the assumption $\bar{k}[X]^\times = \bar{k}^\times$ on $X$. A type is then a Galois equivariant homomorphism from the character module of $T$ to the geometric Picard group $X$, or equivalently, an element of $\H^0(k,\H^1(X \times_k \bar{k}, T  \times_k \bar{k}))$. Later in \cite{HS2013Descent}, Harari and Skorobogatov developped an ``open descent theory'' for torsors under a group of multiplicative type, without any assumption on the invertible functions on $X$. There, they introduced the notion of {\em extended type}, which is a morphism in the derived category of Galois modules. Very recently in \cite{HW2022Supersolvable}, Harpaz and Wittenberg deployed the notion of {\em finite descent type} for their work on supersolvable descent, which is the first non-abelian definition. The descent type constructed in this paper unifies these three definitions.

\subsection{Notations and conventions} \label{subsec:Notation}

The followings conventions shall be deployed throughout the article.

{\em Fields.} For most of the cases, we shall consider a perfect field $k$ and a fixed algebraic closure $\bar{k}$ of it. We denote by $\Gamma_k$ the absolute Galois group $\Gal(\bar{k}/k)$. By convention, a {\em $k$-variety} a separated $k$-scheme of finite type. When $U$ is a $k$-scheme, we denote by $X_U$ the base change $X \times_k U$. We usually write $\bar{X}:=X \times_k \bar{k}$ (but the notation $\bar{X}$ does not always mean $X$ is the base change to $\bar{k}$ of a $k$-variety). For each $s \in \Gamma_k$, we have an $X$-automorphism $\id_X \times_k \Spec(s^{-1}): \bar{X} \to \bar{X}$, which is usually denoted by $s_{\#}$. It induces a bijection on the set $\bar{X}(S)$ for any $\bar{k}$-scheme $S$, denoted by $x \mapsto \tensor[^s]{x}{}$.

{\em Cohomology.} Unless stated otherwise, every (hyper-)cohomology groups (and sets) shall be {\em \'etale} or {\em Galois}. By abuse of notation, any object of an abelian category is identified to a 1-term complex concentrated in degree $0$. When $k$ is a field, every sheaf of abelian groups $\Fcal$ on the small \'etale site $\Spec(k)_{\et}$ shall be abusively identified to the $\Gamma_k$-module $\Fcal(\bar{k}):=\varinjlim_{K} \Fcal(K)$, where $K$ runs through the finite (or equivalently, finite Galois) extensions of $k$. We also use $\Dcal^+(k)$ to denote the derived category of bounded below $\Gamma_k$-modules.

{\em Torsors.} Let $G$ be a {\em smooth} algebraic group over a perfect field $k$. If $X$ is a non-empty $k$-variety, a (right) $X$-torsor under $G$ is a $k$-variety $Y$ equipped with an action $Y \times_k G \to Y$ as well as an fppf and $G$-invariant morphism $f: Y \to X$, such that $G(\bar{k})$ acts simply transitively on the fibres $f^{-1}(x)$ for all $x \in X(\bar{k})$. Such a torsor gives rise to a class $[Y] \in \H^1(X,G)$ in the pointed $\check{\textrm{C}}$ech cohomology set for the \'etale topology (see \cite[Chapter III, \S 4]{Milne1980Etale} or \cite[p. 18]{Skorobogatov2001Torsors}), and $[Y] = [Y']$ if and only if $Y$ and $Y'$ are ($G$-equivariantly) isomorphic. For $g \in G(\bar{k})$, we usually denote by $\rho_g: \bar{Y} \to \bar{Y}$ the $\bar{X}$-automorphism $y \mapsto y \cdot g$. Note that this automorphism is {\em not} necessarily $\bar{G}$-equivariant unless $G$ is commutative. For $X = \Spec(k)$, the set $\H^1(k,G)$ can be described in terms of cocycles as follows. Equip $G(\bar{k})$ with the discrete topology. A cocycle $\sigma \in \Z^1(k,G)$ with coefficients in $G$ is a continuous ({\em i.e.}, locally constant) map $\sigma: \Gamma_k \to G(\bar{k})$ such that $\sigma_{st} = \sigma_s \tensor[^s]{\sigma}{_t}$ for all $s,t \in \Gamma_k$. Two such cocycles $\sigma$ and $\tau$ are cohomologous ({\em i.e.}, define the same element of $\H^1(k,G)$) if and only if there exists $g \in G(\bar{k})$ such that $\tau_s = g^{-1} \sigma_s \tensor[^s]{g}{}$ for all $s \in \Gamma_k$.

{\em Twisted forms.} Let $k$ and $G$ be as above. For each cocycle $\sigma \in \Z^1(k,G)$, we may {\em twist} $G$ to obtain an inner $k$-form $G^{\sigma}$ as follows. As an algebraic $\bar{k}$-group, $\bar{G}^{\sigma}:=\bar{G}$, but the twisted action $\cdot_\sigma$ of $\Gamma_k$ on $\bar{G}^{\sigma}$ is given by $s \mapsto \inn(\sigma_s) \circ s_{\#}$ (here, $\inn(g)$ denotes the conjugation by $g$, for each $g \in G(\bar{k})$), that is
    \begin{equation*}
        \forall s \in \Gamma_k,\forall g \in \bar{G}, \quad s \cdot_\sigma g := \sigma_s \tensor[^s]{g}{} \sigma_s^{-1}.
    \end{equation*} 
If $f: Y \to X$ is a torsor under $G$, we define the twisted torsor $f^{\sigma}: Y^{\sigma} \to X$ as follows. As a $\bar{k}$-variety, $\bar{Y}^{\sigma}:=\bar{Y}$, and as a $\bar{k}$-morphism, $f^{\sigma} \times_k \bar{k} = f \times_k \bar{k}$. The Galois action $\cdot_\sigma$ on $\bar{Y}^{\sigma}$ is given by $s \mapsto \rho_{\sigma_s}^{-1} \circ s_{\#}$, that is,
    \begin{equation*}
        \forall s \in \Gamma_k,\forall y \in \bar{Y}, \quad s \cdot_\sigma y := \tensor[^s]{y}{} \cdot \sigma_s^{-1}.
    \end{equation*}
If $G$ is commutative, then $G^{\sigma} = G$, and $[Y^\sigma] = [Y] - p^\ast [\sigma]$ in $\H^1(X,G)$, where $p: X \to \Spec(k)$ denotes the structure morphism.

{\em Groups of multiplicative type.} Let $k$ be a field. A {\em group of multiplicative type} $M$ over $k$ is a $k$-form of $\Gbb_{m}^n \times \prod_{i=1}^m \mu_{n_i}$, where $n \ge 0$ and $n_i \ge 1$. A $k$-{\em torus} is a connected group of multiplicative type, that is, a $k$-form of $\Gbb_{m}^n$ for some $n \ge 0$. The {\em character module} of $M$ is $\wh{M}:=\iHom_k(M,\Gbb_m)$, that is, $\wh{M} = \Hom_{\bar{k}}(\bar{M},\Gbb_m)$ as a (finitely generated) abelian group, equipped with the action of $\Gamma_k$ given by the formula
    \begin{equation*}
        \forall s \in \Gamma_k,\forall \chi \in \wh{M},\forall m \in \bar{M}, \quad (\tensor[^s]{\chi}{})(m):=\tensor[^s]{\chi}{}(\tensor[^{s^{-1}}]{m}{}).
    \end{equation*}
The group $M$ is smooth precisely when the integers $n_i$ are not divisible by the characteristic of $k$. If we denote by $\Delta: \bar{k}[M] \to \bar{k}[M] \otimes_{\bar{k}} \bar{k}[M]$ the comultiplication of the Hopf algebra $\bar{k}[M]$, then the $\bar{k}$-vector space $\bar{k}[M]$ has a basis consisting of the invertible functions $f \in \bar{k}[M]^{\times}$ satisfying $\Delta(f) = f \otimes f$ (the {\em group-like} elements). The $\Gamma_k$-module of group-like elements of $\bar{k}[M]$ is isomorphic to $\wh{M}$. Finally, we recall that the construction $M \mapsto \wh{M}$ establishes an exact anti-equivalence of categories, called {\em Cartier duality}, between the $k$-groups of multiplicative type and the discrete $\Gamma_k$-modules which are finitely generated as abelian groups.


\subsection{Statements of the main results} 

The article is organized as follows. 

In \S \ref{sec:Definition}, we give the definition of descent type. Paragraph \ref{subsec:H2} recalls the notion of kernel and non-abelian Galois $2$-cohomology, which is crucial to our construction. The definition of descent type and its three fundamental properties shall be presented in paragraph \ref{subsec:Type}. In paragraph \ref{subsec:Functoriality}, we show that the set of descent types on a base variety $X$ bound by a kernel $L$ is contravariantly functorial in $X$ and partially convariantly functoriality in $L$. Unlike torsors, descent types also behave nicely with extension of algebraic groups; this property shall be mentioned in paragraph \ref{subsec:Devissage}.

The \S \ref{sec:Abelian} of this article is devoted to the study of descent types bound by a commutative algebraic group. Our main result here is

\begin{customthm} [Theorem \ref{thm:Abelian}] \label{customthm:Abelian}
    Let $k$ be a perfect field, $X$ a non-empty $k$-variety, and $G$ a smooth commutative algebraic $k$-group. There exists a structure of abelian group on the set $\DType(X,G)$ of descent types on $X$ bound by $\lien(G)$, which is functorial in $X$ and $G$ (see Definition \ref{defn:DescentType} \ref{defn:DescentType3}). Furthermore, we have a commutative diagram 
	\begin{equation} \label{eq:DescentAbelianExactSequence}
		\xymatrix@C-1pc{
			& \H^1(k,G) \ar[r] \ar[d] & \H^1(X,G) \ar[r] \ar@{=}[d] & \DType(X,G) \ar[r] \ar[d] & \H^2(k,G) \ar[r] \ar[d] & \H^2(X,G) \ar@{=}[d] \\
			0 \ar[r] & \H^1(k,G(\bar{X})) \ar[r] & \H^1(X,G) \ar[r] & \H^0(k,\H^1(\bar{X},\bar{G})) \ar[r] &\H^2(k,G(\bar{X})) \ar[r] & \H^2(X,G),
		}
	\end{equation}
    where the bottom row is the exact sequence issued from the Hochschild--Serre spectral sequence 
	\begin{equation*}
		\H^p(k,\H^q(\bar{X},\bar{G})) \Rightarrow \H^{p+q}(X,G).
	\end{equation*}
    The top row of \eqref{eq:DescentAbelianExactSequence} is a complex, which is exact whenever $X$ is quasi-projective, reduced, and geometrically connected.
\end{customthm}

The top row of \eqref{eq:DescentAbelianExactSequence} is similar to the fundamental exact sequence of Colliot-Th\'el\`ene--Sansuc (see \cite[Theorem 2.3.6, Corollary 2.3.9]{Skorobogatov2001Torsors}), or more generally, the exact sequence for the extended types {\em \`a la} Harari--Skorobogatov \cite[Proposition 8.1]{HS2013Descent}. In fact, we shall show in \S \ref{sec:Comparison} that they are equivalent. In doing so, we also give an explicit and accessible description of the extended type in paragraph \ref{subsec:ExtendedType}.

\begin{customthm} [Theorem \ref{thm:Comparison}] \label{customthm:Comparison}
    For torsors under groups of multiplicative types, the notions of descent type and extended type \cite[Definition 8.2]{HS2013Descent} are equivalent.
\end{customthm}

{\em Acknowledgements.} This work is a part of the author's Ph. D. thesis. I would like to thank the Laboratoire de Math\'ematiques d'Orsay, Universit\'e Paris-Saclay for the excellent working condition, as well as the ``Contrat doctorant sp\'ecifique normalien'' from \'Ecole normale sup\'erieure which funded my research. I am especially grateful to my Ph. D. advisor, David Harari, for his support during the work and the valuable comments on the manuscript. The existence of this article was inspired by a personal communication with Olivier Wittenberg, whom I warmly thank. I also appreciate Giancarlo Lucchini Arteche for the discussions on kernels and non-abelian $2$-cohomology.
\section{Definition and basic properties}  \label{sec:Definition}

\subsection{Kernels and non-abelian 2-cohomology} \label{subsec:H2}

In the last 30 years, the notion of kernel\footnote{or band, or lien.} and non-abelian Galois $2$-cohomology has been systematically being used for the study of homogeneous spaces without rational points. We recall their definitions and essential properties in this \S. The typical references for this subject are \cite[\S 1 and \S 2]{Borovoi1993H2}, \cite[\S 1]{FSS1998H2}, or \cite[\S 2.2]{DLA2019Reduction}.

Let $k$ be a {\em perfect} field. We shall deploy the notion of $k$-kernel in the sense of Borovoi--Springer. Let $p: \bar{X} \to \Spec(\bar{k})$ be a non-empty $\bar{k}$-variety. For $s \in \Gamma_k$, we shall say that a $k$-automorphism $\phi: \bar{X} \to \bar{X}$ is {\em $s$-semilinear} if $p \circ \phi = \Spec(s^{-1}) \circ p$. A $1$-semilinear $k$-automorphism is simply a $\bar{k}$-automorphism. In general, to give an $s$-semilinear automorphism of $\bar{X}$ is to give a $\bar{k}$-isomorphism $(\Spec(s))^{\ast} \bar{X} \to \bar{X}$. For example, if $X$ is a $k$-form of $\bar{X}$, then the action of $s$ on $\bar{X} = X \times_k \bar{k}$, {\em i.e.}, $s_{\#} = \id_X \times_k \Spec(s^{-1})$, is $s$-semilinear. We say that a $k$-automorphism of $\bar{X}$ is {\em semilinear} if it is $s$-semilinear for some $s \in \Gamma_k$; such an $s$ is necessarily unique. Thus, if $\SAut(\bar{X}/k)$ denotes the group of semilinear $k$-automorphisms of $\bar{X}$, then we have a homomorphism $\SAut(\bar{X}/k) \to \Gamma_k$ with kernel $\Aut(\bar{X}/\bar{k})$.

Let $\bar{G}$ be a {\em smooth} algebraic group over $\bar{k}$. Let $\SAut^{\gr}(\bar{G}/k)$ denote the subgroup of $\SAut(\bar{G}/k)$ consisting of semilinear automorphisms $\phi: \bar{G} \to \bar{G}$ that correspond to an isomorphism 
    \begin{equation*}
        (\Spec(s))^\ast \bar{G} \to \bar{G}
    \end{equation*}
of algebraic $\bar{k}$-groups for some (unique) $s \in \Gamma_k$. Equivalently, an element $\phi \in \SAut(\bar{G}/k)$ belongs to $\SAut^{\gr}(\bar{G}/k)$ if and only if it induces an isomorphism on the abstract group $\bar{G}(\bar{k})$. It is clear that $\SAut^{\gr}(\bar{G}/k) \cap \Aut(\bar{G}/\bar{k}) = \Aut^{\gr}(\bar{G}/\bar{k})$, the group of automorphism of the algebraic $\bar{k}$-group $\bar{G}$. Thus, we have an exact sequence
    \begin{equation} \label{eq:H2KernelAut}
        1 \to \Aut^{\gr}(\bar{G}/\bar{k}) \to \SAut^{\gr}(\bar{G}/k) \to \Gamma_k.
    \end{equation}
Let $\Inn(\bar{G}) := \{\inn(g): g \in \bar{G}(\bar{k})\}$, and put $\Out(\bar{G}) := \Aut^{\gr}(\bar{G}/\bar{k})/\Inn(\bar{G})$ and $\SOut(\bar{G}/k) := \SAut^{\gr}(\bar{G}/k)/\Inn(\bar{G})$, then the sequence
    \begin{equation} \label{eq:H2KernelOut}
        1 \to \Out(\bar{G}) \to \SOut(\bar{G}/k) \to \Gamma_k
    \end{equation}
is exact. Equip the set $\SAut^{\gr}(\bar{G}/k)$ with the initial topology defined by the evaluation maps
    \begin{equation*}
        \ev_g: \SAut^{\gr}(\bar{G}/k) \to G(\bar{k}), \quad \phi \mapsto \phi(g)
    \end{equation*}
for $g \in \bar{G}(\bar{k})$, where $\bar{G}(\bar{k})$ is given the discrete topology. A {\em $k$-kernel on $\bar{G}$} is a homomorphic section $\kappa: \Gamma_k \to \SOut(\bar{G}/k)$ of \eqref{eq:H2KernelOut} ({\em i.e.}, a semilinear outer action of $\Gamma_k$ on $\bar{G}$), which lifts to a continuous (set-theoretic) section $\Gamma_k \to \SAut^{\gr}(\bar{G}/k)$ of \eqref{eq:H2KernelAut}. 

We also call the pair $L = (\bar{G},\kappa)$ a $k$-kernel. A $k$-form $G$ of $\bar{G}$ defines a $k$-kernel, denoted by $\lien(G)$. We shall say $G$ is a {\em $k$-form of $L$} if $L = \lien(G)$. By Galois descent, to give a $k$-form of $G$ of $L$ is to give a continuous homomorphic section $\Gamma_k \to \SAut^{\gr}(\bar{G}/k)$ of \eqref{eq:H2KernelAut} ({\em i.e.}, a semilinear action of $\Gamma_k$ on $\bar{G}$) lifting $\kappa$. If $G'$ is another $k$-form of $\bar{G}$, then $G'$ is a $k$-form of $\lien(G)$ if and only if $G'$ is an inner form of $G$. When $\bar{G}$ is commutative, any $k$-kernel on $\bar{G}$ has a {\em unique} $k$-form.

There exist at least three equivalent definitions of the set $\H^2(k,L)$ of non-abelian Galois $2$-cohomology with coefficients in $L$. We shall mainly use the one in terms of group extensions. A topological extension of $\Gamma_k$ by $\bar{G}(\bar{k})$ is an exact sequence of topological groups
    \begin{equation*}
        1 \to \bar{G}(\bar{k}) \to E \xrightarrow{q} \Gamma_k \to 1,
    \end{equation*}
where $\bar{G}(\bar{k})$ is discrete in $E$ and where the map $q$ is open (and continuous). It naturally induces an outer action $\Gamma_k \to \Out(\bar{G}(\bar{k})) = \Aut(\bar{G}(\bar{k})) / \Inn(\bar{G}(\bar{k}))$ on the abstract group $\bar{G}(\bar{k})$. We say that such an extension is {\em bound by $L$} if this outer action coincides with the natural outer action induced by $\kappa$. Two such extensions $E$ and $E'$ are said to be {\em equivalent} if there is an isomorphism $E \xrightarrow{\cong} E'$ of topological groups fitting in a commutative diagram
    \begin{equation*}
        \xymatrix{
            1 \ar[r] & \bar{G}(\bar{k}) \ar@{=}[d] \ar[r] & E \ar[d]^{\cong} \ar[r] & \Gamma_k  \ar[r] \ar@{=}[d] & 1 \\
            1 \ar[r] & \bar{G}(\bar{k}) \ar[r] & E' \ar[r] & \Gamma_k  \ar[r] & 1.
        }
    \end{equation*}
    
Let $\H^2(k,L)$ be the set of equivalent classes of topological extensions of $\Gamma_k$ by $\bar{G}(\bar{k})$ which are bound by $L$. It might happen that the set $\H^2(k,L)$ is empty \cite[\S 1.20 and \S 1.21]{FSS1998H2}. The class $[E] \in \H^2(k,L)$ of an extension $E$ is said to be {\em neutral} if it is split, {\em i.e.}, $E \to \Gamma_k$ has a continuous homomorphic section. Even if $\H^2(k,L)$ is non-empty, it needs not have a neutral element; and if a neutral element exists, it needs not be unique. When $G$ is a commutative algebraic $k$-group, $\H^2(k,G):=\H^2(k,\lien(G))$ is the usual Galois $2$-cohomology group with coefficients in $G$, and $0$ is its unique neutral element. \label{NeutralClassInH2}

If $L = (\bar{G},\kappa)$ and $L' = (\bar{G}',\kappa')$ are $k$-kernels, a {\em morphism of $k$-kernels} $L \to L'$ is a morphism $\varphi: \bar{G} \to \bar{G}'$ of algebraic $\bar{k}$-groups such that there exist respective continuous set-theoretic sections $\phi: \Gamma_k \to \SAut^{\gr}(\bar{G}/k)$ and $\phi': \Gamma_k \to \SAut^{\gr}(\bar{G}'/k)$ lifting $\kappa$ and $\kappa'$, such that $\varphi \circ \phi_s = \phi_s' \circ \varphi$ for all $s \in \Gamma_k$. The morphism $\varphi$ does not necessarily induce a map $\H^2(k,L) \to \H^2(k,L')$. Nevertheless, there is always a {\em relation}
    \begin{equation*}
        \varphi_\ast: \H^2(k,L) \multimap \H^2(k,L'),
    \end{equation*}
defined as follows (see \cite[\S 2.2.3]{DLA2019Reduction}). A class $\eta' \in \H^2(k,L')$ is in relation $\varphi_\ast$ with another class $\eta \in \H^2(k,L)$ if and only if there exist topological extensions $E$ and $E'$ representing $\eta$ and $\eta'$ respectively, and a continuous homomorphism $E \to E'$ fitting in a commutative diagram
        \begin{equation*}
            \xymatrix{
                1 \ar[r] & \bar{G}(\bar{k}) \ar[d]^{\varphi} \ar[r] & E \ar[d] \ar[r] & \Gamma_k \ar[r] \ar@{=}[d] & 1 \\
                1 \ar[r] & \bar{G}'(\bar{k}) \ar[r] & E' \ar[r] & \Gamma_k \ar[r] & 1.
            }
        \end{equation*}
If either the underlying morphism $\varphi: \bar{G} \to \bar{G}'$ of algebraic $\bar{k}$-groups is surjective or $\bar{G}'$ is commutative, then the relation $\varphi_\ast$ is in fact a {\em map}. We also note that, if $K$ is any field extension of $k$ which is perfect, then $L$ restricts to a $K$-kernel $L_K$, and we have a restriction map 
    \begin{equation*}
        \H^2(k,L) \to \H^2(K,L) := \H^2(K, L_K).
    \end{equation*}

\subsection{Non-abelian descent types} \label{subsec:Type}

We consider a perfect field $k$ and a non-empty variety $p: X \to \Spec(k)$. For most of the cases, $X$ will be quasi-projective (a condition which allows Galois descent). 

If $\bar{Y}$ is a $\bar{k}$ variety equipped with a morphism $f: \bar{Y} \to \bar{X} := X_{\bar{k}}$, we set $\SAut(\bar{Y}/X) := \Aut(\bar{Y}/X) \cap \SAut(\bar{Y}/k)$. This group is equipped with the {\em weak topology}, which, by definition, is the initial topology defined by the maps
	\begin{equation*}
		\ev_y: \SAut(\bar{Y}/X) \to \bar{Y}(\bar{k}), \quad \alpha \mapsto \alpha(y),
	\end{equation*}
for $y \in \bar{Y}(\bar{k})$. One can then show that the restriction $q: \SAut(\bar{Y}/X) \to \Gamma_k$ of the natural homomorphism  $\SAut(\bar{Y}/k) \to \Gamma_k$ is open. Nevertheless, $q$ needs not be continuous. For technical reasons, we shall work with a finer topology on $\SAut(\bar{Y}/X)$, called the {\em modified weak topology}. It is, by definition, the initial topology defined by the maps $\ev_y$ and $q$. This makes $q$ continuous and open without changing the subspace topology on $\Aut(\bar{Y}/\bar{X})$ (which equals $\Ker(q)$). Indeed, to show that $q$ is open, consider a basic open subset $U \subseteq \SAut(\bar{Y}/X)$, which has the form 
	\begin{equation*}
            U = \left(\bigcap_{i=1}^n \ev_{y_i}^{-1}(\Sigma_i)\right)\cap q^{-1}(s_0\Gamma_K) = \{\alpha \in q^{-1}(s_0\Gamma_K): \forall i \in \{1,\ldots,n\}, \alpha(y_i) \in \Sigma_i\},
	\end{equation*}
where $n \ge 0$, $y_1,\ldots,y_n \in \bar{Y}(\bar{k})$, $\Sigma_1,\ldots,\Sigma_n \subseteq \bar{Y}(\bar{k})$, $s_0 \in \Gamma_k$, and $K/k$ is a finite extension. Take a sufficiently large finite extension $K'/K$ such that $\bar{Y}$ is the base change to $\bar{k}$ of a torsor $Y \to X_{K'}$ under a $K'$-form $G'$ of $\bar{G}$, and $y_1,\ldots,y_n \in Y(K')$. Let $\alpha \in U$. For each $t \in \Gamma_{K'}$, let $t_{\#}: \bar{Y} \to \bar{Y}$ denote the action of $t$ on $\bar{Y}$ defined {\em via} its $K'$-form $Y$. Since $t_{\#} \in \Gamma_{K'} \subseteq \Gamma_K$ stabilizes $y_1,\ldots,y_n$, one has $\alpha \circ t_{\#} \in U$. It follows that $q(\alpha) t = q(\alpha \circ t_{\#}) \in q(U)$, hence $q(U) \subseteq \Gamma_k$ contains the open neighborhood $q(\alpha)\Gamma_{K'}$ of $q(\alpha)$. Since $\alpha$ is arbitrary, we conclude that $q(U)$ is open. Furthermore, a set-theoretic section $\Gamma_k \to \SAut(\bar{Y}/X)$ is continuous with respect to this modified topology if and only if it is continuous with respect to the weak topology.

To sum up, we have an exact sequence of topological groups
	\begin{equation} \label{eq:DescentTypeExactSequenceSAut}
		1 \to \Aut(\bar{Y}/\bar{X}) \to \SAut(\bar{Y}/X) \xrightarrow{q} \Gamma_k.
	\end{equation}
If $f: \bar{Y} \to \bar{X}$ is a torsor under a smooth algebraic group $\bar{G}$ over $\bar{k}$, then we have an inclusion
    \begin{equation*}
        \bar{G}(\bar{X}) \hookrightarrow \Aut(\bar{Y}/\bar{X}), \quad \sigma \mapsto \rho_{\sigma}^{-1},
    \end{equation*}
where, for each $\bar{k}$-morphism $\sigma: \bar{X} \to \bar{G}$, the $\bar{X}$-isomorphism $\rho_\sigma: \bar{Y} \to \bar{Y}$ is defined by $y \mapsto y \cdot \sigma(f(y))$. In particular, we have an inclusion 
    \begin{equation*}
        \bar{G}(\bar{k}) \hookrightarrow \Aut(\bar{Y}/\bar{X}), \quad g \mapsto \rho_{g}^{-1},
    \end{equation*}
Note that the topology on $\SAut(\bar{Y}/X)$ induces the discrete topology on $\bar{G}(\bar{k})$. Indeed, fix a point $y \in \bar{Y}(\bar{k})$, then its stabilizer $U \subseteq \SAut(\bar{Y}/X)$ is an open subset and $U \cap \bar{G}(\bar{k}) = \{1\}$ by the very definition of a torsor.

\begin{defn} \label{defn:DescentType}
    Let $L = (\bar{G},\kappa)$ be a $k$-kernel.
	\begin{enumerate}
	    \item \label{defn:DescentType1} A {\em descent datum on $X$ bound by $L$} is a pair $(\bar{Y},E)$, where $\bar{Y} \to \bar{X}$ is a torsor under $\bar{G}$ and $E \subseteq \SAut(\bar{Y}/X)$ is a subgroup fitting in an exact sequence
		\begin{equation*}
			\xymatrix{
				1 \ar[r] & \bar{G}(\bar{k}) \ar[r] \ar@{_{(}->}[d]^{g \mapsto \rho_g^{-1}} & E \ar[r] \ar@{_{(}->}[d] & \Gamma_k \ar[r] \ar@{=}[d] & 1 \\
				1 \ar[r] & \Aut(\bar{Y}/\bar{X}) \ar[r] & \SAut(\bar{Y}/X) \ar[r]^-{q} & \Gamma_k
			}
		\end{equation*} 
		of topological groups, where the top row is bound by $L$ and the bottom row is \eqref{eq:DescentTypeExactSequenceSAut}.
		
	   \item \label{defn:DescentType2} Let $L' = (\bar{G}',\kappa')$ be a second $k$-kernel, $\varphi: L \to L'$ a morphism, and $(\bar{Y}', E')$ a descent datum on $X$ bound by $L'$. A {\em morphism of descent data $(\bar{Y}, E) \to (\bar{Y}', E')$ compatible with $\varphi$} is a pair $(\iota,\varpi)$, where $\iota: \bar{Y} \to \bar{Y}'$ is an $\varphi$-equivariant $\bar{X}$-morphism and $\varpi: E \to E'$ is a continuous homomorphism fitting in a commutative diagram
		\begin{equation*}
			\xymatrix{
				1 \ar[r] & \bar{G}(\bar{k}) \ar[r] \ar[d]^{\varphi} & E \ar[r] \ar[d]^{\varpi} & \Gamma_k \ar@{=}[d] \ar[r] & 1 \\
				1 \ar[r] & \bar{G}'(\bar{k}) \ar[r] & E' \ar[r] & \Gamma_k \ar[r] & 1
			}
		\end{equation*} 
	such that the diagram
		\begin{equation*}
			\xymatrix{
				\bar{Y} \ar[r]^{\alpha} \ar[d]^{\iota} & \bar{Y} \ar[d]^{\iota} \\
				\bar{Y}' \ar[r]^{\varpi_\alpha} & \bar{Y}'
			}
		\end{equation*}
	commutes for all $\alpha \in E$.
		
	\item \label{defn:DescentType3} An isomorphism of descent data on $X$ bound by $L$ is an isomorphism compatible with $\id_L$. The isomorphism class of such a descent datum $(\bar{Y},E)$ is denoted by $[\bar{Y},E]$ and is called a {\em descent type on $X$ bound by $L$}. Denote by $\DType(X,L)$ the set of such isomorphism classes.
	\end{enumerate}
\end{defn} 

It is easy to see that an isomorphism $(\bar{Y},E) \xrightarrow{\cong} (\bar{Y}',E')$ of descent data on $X$ bound by $L$ is given by a $\bar{G}$-equivariant $\bar{X}$-isomorphism $\iota: \bar{Y} \to \bar{Y}'$ such that $\iota \circ \alpha \circ \iota^{-1} \in E'$ for all $\alpha \in E$.

\begin{remk} \label{remk:DescentType}
    Let $L = (\bar{G},\kappa)$ be a $k$-kernel.
    \begin{enumerate}
        \item \label{remk:DescentType1} For each $s$-semilinear automorphism $\phi \in \SAut^{\gr}(\bar{G}/k)$, where $s \in \Gamma_k$, one defines the {\em conjugate torsor} $\phi\bar{Y}$ over $\bar{X}$ as follows. Let $\phi\bar{Y}$ be the $\bar{X}$-scheme $\bar{Y}$ equipped with the ``twisted'' structure morphism $\bar{Y} \to \bar{X} \xrightarrow{s_{\#}} \bar{X}$. Then the action
        \begin{equation*}
            \bar{Y} \times \bar{G} \to \bar{Y}, \quad (y,g) \mapsto y \cdot \phi^{-1}(g)
        \end{equation*}
        makes $\phi \bar{Y}$ an $\bar{X}$-torsor under $\bar{G}$. This construction yields action of $\SAut^{\gr}(\bar{G}/k)$ on the set $\H^1(\bar{X},\bar{G})$. If $\phi = \inn(g)$ for some $g \in \bar{G}(\bar{k})$, then the map $\rho_g$ defines a $\bar{G}$-equivariant $\bar{X}$-isomorphism $\bar{Y} \xrightarrow{\cong} \phi\bar{Y}$, that is, $[\bar{Y}] = [\phi\bar{Y}]$ in $\H^1(\bar{X},\bar{G})$. Thus, we obtain an action of $\SOut(\bar{G}/k)$ on $\H^1(\bar{X},\bar{G})$. In particular, the map $\kappa$ naturally induces an action of $\Gamma_k$ on $\H^1(\bar{X},\bar{G})$. If $(\bar{Y},E)$ is descent datum on $X$ bound by $L$, then the surjectivity of $E \to \Gamma_k$ implies $[\bar{Y}] \in \H^0(k,\H^1(\bar{X},\bar{G}))$. Thus, we obtain two maps
		\begin{equation*}
			\DType(X,L) \to  \H^0(k,\H^1(\bar{X},\bar{G})), \quad [\bar{Y}, E] \mapsto [\bar{Y}] 
		\end{equation*}
	and
		\begin{equation*}
			\DType(X,L) \to \H^2(k,L), \quad [\bar{Y}, E] \mapsto [E]. 
		\end{equation*}

    \item \label{remk:DescentType2} Suppose that $(\bar{Y},E)$ is a descent datum on $X$ bound by $L$ with $\bar{Y}$ connected and $\bar{G}(\bar{k})$ finite. Then $\Aut(\bar{Y}/\bar{X}) = \bar{G}(\bar{k})$, hence $E = \Aut(\bar{Y}/X)$. Thus, the descent type $[\bar{Y},E]$ is completely determined by the $\bar{X}$-torsor $\bar{Y}$ under $\bar{G}$, which is a finite \'etale cover of $\bar{X}$ such that the composite $\bar{Y} \to \bar{X} \to X$ is Galois (that is, the field extension $\bar{k}(\bar{Y})/k(X)$ is Galois). We recover the notion of {\em finite descent type} by Harpaz--Wittenberg \cite[Definition 2.1]{HW2022Supersolvable}.
    \end{enumerate}
\end{remk}

A descent type $[\bar{Y},E] \in \DType(X,L)$ gives rise to not only an $\bar{X}$-torsor $\bar{Y}$, but also a cohomology class $[E] \in \H^2(k,L)$. Let us show that the non-neutrality this class is precisely the obstruction to lifting $\bar{Y}$ to an $X$-torsor $Y$ and that two such liftings differ by twisting by a Galois cocycle.

Let $Y \to X$ be a torsor under an algebraic $k$-group $G$. Then we have a semilinear action of $\Gamma_k$ of $Y$, that is, a continuous homomorphic section
	\begin{equation*}
		\varsigma_Y: \Gamma_k \to \SAut(\bar{Y}/X), \quad s \mapsto s_{\#}.
	\end{equation*} 
It is easily checked that 
	\begin{equation} \label{eq:DescentType1}
		\forall s \in \Gamma_k,\forall g \in G(\bar{k}), \quad \rho_{\tensor[^s]{g}{}} = s_{\#} \circ \rho_g \circ s_{\#}^{-1}.
	\end{equation}
Let $E_Y \subseteq \SAut(\bar{Y}/X)$ denote the semidirect product of $G(\bar{k})$ and $\varsigma_Y(\Gamma_k)$. Then the pair $(\bar{Y},E_Y)$ is a descent datum on $X$ bound by $\lien(G)$ and the cohomology class $[E_Y] \in \H^2(k,\lien(G))$ is neutral. The converse also holds under mild conditions.

\begin{defn} \label{def:DescentTypeOfTorsor}
    The {\em descent type} of the torsor $Y \to X$ is
        \begin{equation*}
            \dtype(Y):=[\bar{Y},E_Y] \in \DType(X,\lien(G)).
        \end{equation*}
\end{defn}

\begin{prop} \label{prop:DescentType}
    Suppose that $X$ is quasi-projective, reduced, and geometrically connected. Let $L = (\bar{G},\kappa)$ be a $k$-kernel and $\lambda=[\bar{Y},E] \in \DType(X,L)$ a descent type on $X$ bound by $L$ (in particular, $\bar{Y}$ is quasi-projective).
	\begin{enumerate}
		\item \label{prop:DescentType1} There is a canonical bijection between the continuous homomorphic sections $\Gamma_k \to E$ and the equivalence classes of pair $(Y,G)$, where $G$ is a $k$-form of $L$ (that is, $\lien(G) = L$) and $Y \to X$ is a torsor under $G$ of descent type $\dtype(Y) = \lambda$. Here, we impose that two pairs $(Y,G)$ and $(Y',G')$ to be equivalent if $G' = G$ and the $X$-torsors $Y$ and $Y'$ under $G$ are isomorphic. In particular, $X$-torsors of descent type $\lambda$ exist if and only if the class $[E] \in \H^2(k,L)$ is neutral.
		
		\item \label{prop:DescentType2} If $X(k) \neq \varnothing$, then the class $[E] \in \H^2(k,L)$ is neutral.
		
		\item \label{prop:DescentType3} Let $Y \to X$ be a torsor under a $k$-form $G$ of $L$, of descent type $\dtype(Y) = \lambda$. If $Y' \to X$ is a torsor under a $k$-form $G'$ of $L$, then $\dtype(Y') = \lambda$ if and only if there exists a Galois cocycle $\sigma \in \Z^1(k,G)$ such that $G' = G^{\sigma}$ and the $X$-torsors $Y'$ and $Y^{\sigma}$ under $G^{\sigma}$ are isomorphic.
	\end{enumerate}
\end{prop}
\begin{proof}
\begin{enumerate}
    \item As we have seen above, a torsor $Y \to X$ under a $k$-form $G$ of $L$ defines a continuous homomorphic section $\varsigma_Y: \Gamma_k \to E$ (here, we abusively identity $\bar{Y}$ and $\bar{G}$ to the respective base changes $Y_{\bar{k}}$ and $G_{\bar{k}}$, hence $E_Y = E$ whenever $\dtype(Y) = \lambda$). Conversely, as in the proof of \cite[Proposition 2.2 (3)]{HS2002Nonabelian}, a continuous homomorphic section $\Gamma_k \to E$ gives rise to a torsor $Y \to X$ under a $k$-form $G$ of $L$ (in {\em loc. cit.}, the authors use the weak topology on $\SAut(\bar{Y}/X)$ instead of our modified weak topology, but we have seen that this does not matter). One checks that these two constructions are converse to each other.

    \item See \cite[Theorem 2.5]{HS2002Nonabelian}.

    \item We make the following identifications: $\bar{Y} = Y_{\bar{k}}$, $\bar{G} = G_{\bar{k}}$, and $E = E_Y$, the semidirect product of $\bar{G}(\bar{k})$ and $\varsigma_Y(\Gamma_k) = \{s_{\#}: s \in \Gamma_k\}$ inside $\SAut(\bar{Y}/X)$ (as always, $s \mapsto s_{\#}$ denotes the semilinear action of $\Gamma_k$ on $\bar{Y}$ defined {\em via} the $k$-form $Y$). If $\sigma \in \Z^1(k,G)$ is any Galois cocycle, then clearly $\bar{Y}^\sigma = \bar{Y}$ and $\lien(G^\sigma) = \lien(G) = L$. Furthermore, the twisted action of $\Gamma_k$ on $\bar{Y}^\sigma$ is $s \mapsto \rho_{\sigma_s}^{-1} \circ s_{\#}$. Thus, $E_{Y^\sigma} \subseteq \SAut(\bar{Y}/X)$ is the semidirect product of $\bar{G}(\bar{k})$ and $\varsigma_{Y^{\sigma}}(\Gamma_k) = \{\rho_{\sigma_s}^{-1} \circ s_{\#}: s \in \Gamma_k\}$, which is visibly equal to $E_Y$. Hence, $\dtype(Y^{\sigma}) = \lambda$.

    Conversely, let $Y' \to X$ be a torsor under a $k$-form $G'$ of $L$, with $\dtype(Y') = \lambda$. By definition, there exists a $\bar{G}$-equivariant $\bar{X}$-isomorphism $\iota: \bar{Y} \to \bar{Y}'$ such that the diagram
	\begin{equation} \label{eq:DiagramDescentType1}
		\xymatrix{
			1 \ar[r] & G(\bar{k}) \ar[r] \ar@{=}[d] & E \ar[rr] \ar[d]^-{\alpha \mapsto \iota \circ \alpha \circ \iota^{-1}} && \Gamma_k \ar[r] \ar@{=}[d] & 1 \\
			1 \ar[r] & G'(\bar{k}) \ar[r] & E' \ar[rr] && \Gamma_k \ar[r] & 1
		}
	\end{equation}
	commutes. Write $(s,g) \mapsto \tensor[^s]{g}{}$
	for the action of $\Gamma_k$ on $\bar{G}(\bar{k})$ induced by the $k$-form $G$. 
	Unwinding the commutativity of the left square of \eqref{eq:DiagramDescentType1} yields
	\begin{equation} \label{eq:DescentType2}
		\forall g \in G(\bar{k}), \quad \rho'_g = \iota \circ \rho_g \circ \iota^{-1}.
	\end{equation}
	where $\rho'_g: \bar{Y}' \to \bar{Y}'$ is the $\bar{X}$-isomorphism $y' \mapsto y' \cdot g$. Write $s \mapsto s_{\#}'$ for the action of $\Gamma_k$ on $\bar{Y}'$. That the right square of \eqref{eq:DiagramDescentType1} commutes implies that $s_{\#}$ and $\iota^{-1} \circ s'_{\#} \circ \iota$ are both mapped to $s \in \Gamma_k$, so that we may write
	\begin{equation} \label{eq:DescentType3}
		\iota^{-1} \circ s'_{\#} \circ \iota = \rho_{\sigma_s}^{-1} \circ s_{\#}
	\end{equation}
	for some (unique) $\sigma_s \in G(\bar{k})$. The map $s \mapsto \sigma_s$ is continuous. Let us verify that it is a cocycle. Indeed, let $s,t \in \Gamma_k$. Since $(st)'_{\#} = s'_{\#} \circ t'_{\#}$, it follows from \eqref{eq:DescentType3} that 
	\begin{equation*}
		\rho_{\sigma_{st}}^{-1} \circ s_{\#} \circ t_{\#} = \rho_{\sigma_{st}}^{-1} \circ (st)_{\#} = \rho_{\sigma_{s}}^{-1} \circ s_{\#} \circ \rho_{\sigma_{t}}^{-1} \circ t_{\#},
	\end{equation*}
	or $\rho_{\sigma_{st}}^{-1} = \rho_{\sigma_{s}}^{-1} \circ s_{\#} \circ \rho_{\sigma_{t}}^{-1} \circ s_{\#}^{-1} = \rho_{\sigma_{s}}^{-1} \circ \rho_{\tensor[^s]{\sigma}{_t}}^{-1} = \rho_{\sigma_s \tensor[^s]{\sigma}{_\tau}}^{-1}$, in view of \eqref{eq:DescentType1}. Hence, $\sigma_{st} = \sigma_s \tensor[^s]{\sigma}{_t}$, or $\sigma \in \Z^1(k,G)$. Now, let us show that $G' = G^\sigma$. Indeed, let $s \in \Gamma_k$ and $g \in G(\bar{k})$, then
	\begin{align*}
		\rho'_{s \cdot_\sigma g} & = \iota \circ \rho_{s \cdot_\sigma g} \circ \iota^{-1}, & \text{by \eqref{eq:DescentType2}},\\ 
		& = \iota \circ \rho_{\sigma_s \tensor[^s]{g}{} \sigma_s^{-1}} \circ \iota^{-1}, & \text{by definition of $\cdot_\sigma$}, \\
		& = (\iota \circ \rho_{\sigma_s}^{-1}) \circ \rho_{\tensor[^s]{g}{}} \circ (\rho_{\sigma_s} \circ \iota^{-1}) \\
		& = (s'_{\#} \circ \iota \circ s_{\#}^{-1}) \circ (s_{\#} \circ \rho_g \circ s_{\#}^{-1}) \circ (s_{\#} \circ \iota^{-1} \circ (s'_{\#})^{-1}), & \text{by \eqref{eq:DescentType1} and \eqref{eq:DescentType3}}, \\
		& = s'_{\#} \circ (\iota \circ \rho_g \circ \iota^{-1}) \circ (s'_{\#})^{-1} \\
		& = s'_{\#} \circ \rho'_g \circ (s'_{\#})^{-1}, & \text{by \eqref{eq:DescentType2}}.
	\end{align*}
	This equality means precisely that the action of $\Gamma_k$ on $\bar{G}(\bar{k})$ induced by the $k$-form $G'$ is given by $(s,g) \mapsto s \cdot_\sigma g$, or $G' = G^\sigma$. Finally, for each $s \in \Gamma_k$, the diagram
	\begin{equation*}
		\xymatrix{
			\bar{Y} \ar[rr]^{\iota} \ar[d]^{\rho_{\sigma_s}^{-1} \circ s_{\#}} && \bar{Y}' \ar[d]^{s'_{\#}} \\
			\bar{Y} \ar[rr]^{\iota} && \bar{Y}'
		}
	\end{equation*}
    commutes, in view of \eqref{eq:DescentType3}. By the very definition of twisting, $s \mapsto \rho_{\sigma_s}^{-1} \circ s_{\#}$ is just the action of $\Gamma_k$ on $\bar{Y}$ defined {\em via} the $k$-form $Y^{\sigma}$. Thus, $\iota: \bar{Y} \to \bar{Y}'$ descends into a $G^{\sigma}$-equivariant $X$-isomorphism $Y^{\sigma} \to Y'$.
\end{enumerate}
\end{proof}

\subsection{Functoriality} \label{subsec:Functoriality}

Let $k$ be a perfect field. Naturally, we expect the set $\DType(X,L)$ to be functorial in $X$ and in $L$.

\begin{cons} [Pullback of descent types] \label{cons:PullbackOfDescentType}
    Let $f: X' \to X$ be a morphism between non-empty $k$-varieties and let $L = (\bar{G},\kappa)$ be a $k$-kernel. Then $f$ induces a pullback map $f^\ast: \DType(X,L) \to \DType(X',L)$ as follows. Let $\lambda = [\bar{Y},E] \in \DType(X,L)$. For each $s$-semilinear automorphism $\alpha \in \SAut(\bar{Y}/X)$ (where $s \in \Gamma_k$), define $f^\ast \alpha:=\alpha \times_{\bar{X}} s_{\#}'$ (where $s_{\#}'$ denotes the natural semilinear Galois action on $\bar{X}'$), then $f^\ast \alpha \in \SAut(f^\ast \bar{Y}/X)$ (where $f^\ast \bar{Y} = \bar{Y} \times_{\bar{X}} \bar{X}'$). We obtain a continuous homomorphism $f^\ast: \SAut(\bar{Y}/X) \to \SAut(f^\ast \bar{Y}/X')$. Let $f^\ast E$ be the image of $E$ by $f^\ast$, then $f^\ast \lambda:=[f^\ast \bar{Y}, f^\ast E]$ is a descent type on $X'$ bound by $L$. It is clear that $[E] = [f^\ast E] \in \H^2(k,L)$. If $Y \to X$ is a torsor under a $k$-form $G$ of $L$, then $\dtype(f^\ast Y) = f^\ast \dtype(Y)$. Finally, $\id_X^\ast = \id_{\DType(X,L)}$ and $(f \circ f')^\ast = (f')^\ast \circ f^\ast$ for any morphism $f': X'' \to X'$ between non-empty $k$-varieties. We also note that, for any field extension $K/k$ (with $K$ perfect), a similar construction yields a restriction map $\DType(X,L) \to \DType(X_K,L_K)$.
\end{cons}

As we have seen in the previous paragraph, ``functoriality'' in $L$ for the non-abelian cohomology set $\H^2(k,L)$ is more delicate and can only defined partially in terms of relations. This is also the case for descent types.

\begin{cons} [Pushforward of descent types]  \label{cons:PushforwardOfDescentType}
	Let $L = (\bar{G},\kappa)$ and $L' = (\bar{G}',\kappa')$ be $k$-kernels, and let $\varphi: L \to L'$ be a morphism. We define a relation
		\begin{equation*}
			\varphi_\ast: \DType(X,L) \multimap \DType(X,L'),
		\end{equation*}
    where a descent type $[\bar{Y}',E'] \in \DType(X,L')$ is in relation with another descent type $[\bar{Y},E] \in \DType(X,L)$ if there exists a morphism of descent data $(\bar{Y},E) \to (\bar{Y}',E')$ compatible with $\varphi$ in the sense of Definition \ref{defn:DescentType} \ref{defn:DescentType2}. In this case, the class $[E']$ is in relation $\varphi_\ast: \H^2(k,L) \multimap \H^2(k,L')$ with $[E]$. Furthermore, $\id_{L\ast}$ is the relation ``$=$'' on $\DType(X,L)$. If $\varphi': L' \to L''$ is another morphism of $k$-kernels, then $(\varphi' \circ \varphi)_\ast$ is the composite relation of $\varphi'_\ast$ and $\varphi_\ast$. If $G$ and $G'$ are respective $k$-forms of $L$ and $L'$ such that $\varphi$ descends into a morphism $\varphi: G \to G'$ of algebraic $k$-groups, and if $Y \to X$ is a torsor under $G$, then $\dtype(\varphi_\ast Y)$ is in relation $\varphi_\ast$ with $\dtype(Y)$.
\end{cons}

Let us show that, in the case where $\varphi: 
\bar{G} \to \bar{G}'$ is surjective or the underlying $\bar{k}$-group $\bar{G'}$ of $L'$ is commutative, the relation $\varphi_\ast$ in Construction \ref{cons:PushforwardOfDescentType} is in fact a map.

\begin{lemm} \label{lemm:PushforwardOfDescentType1}
    Let $L = (\bar{G},\kappa)$ and $L' = (\bar{G}',\kappa')$ be $k$-kernels, and let $\varphi: L \to L'$ be a surjective morphism. Let $\lambda = [\bar{Y},E] \in \DType(X,L)$, $\bar{H}:=\Ker(\bar{G} \xrightarrow{\varphi} \bar{G}')$, and $\bar{Z}:=\varphi_\ast \bar{Y} = \bar{Y} / \bar{H}$. For each $s \in \Gamma_k$, any $s$-semilinear automorphism $\alpha \in E$ induces a unique $s$-semilinear automorphism $\varphi_\ast \alpha \in \SAut(\bar{Z}/X)$ fitting in a commutative diagram
        \begin{equation*}
            \xymatrix{
                \bar{Y} \ar[r]^{\alpha} \ar[d]^{\pi} & \bar{Y} \ar[d]^{\pi} \\
                \bar{Z} \ar[r]^{\varphi_\ast \alpha} & \bar{Z},
            }
        \end{equation*} 
    where $\pi: \bar{Y} \to \bar{Z}$ is the canonical $\varphi$-equivariant $\bar{X}$-morphism. Let $\varphi_\ast E:=\{\varphi_\ast \alpha : \alpha \in E\}$, then $\varphi_\ast \lambda:=[\bar{Z},\varphi_\ast E]$ is the unique element of $\DType(X,L')$ which is in relation $\varphi_\ast$ with $\lambda$.
\end{lemm}
\begin{proof}
    Since $E$ is an extension of $\Gamma_k$ by $\bar{G}(\bar{k})$, we have a semilinear action $\psi: E \to \SAut^{\gr}(\bar{G}/k)$ by conjugation in $E$, that is, $\alpha \circ \rho_g \circ \alpha^{-1} = \rho_{\psi_\alpha(g)}$ for all $g \in \bar{G}(\bar{k})$ and $\alpha \in E$. That $\varphi$ is a morphism of $k$-kernels means there are respective continuous (set-theoretic) liftings $\phi: \Gamma_k \to \SAut^{\gr}(\bar{G}/k)$ and $\phi': \Gamma_k \to \SAut^{\gr}(\bar{G}'/k)$ of $\kappa$ and $\kappa'$ such that $\varphi \circ \phi_s = \phi'_s \circ \varphi$ for all $s \in \Gamma_k$. Let $\alpha \in E$ and let $s$ be its image in $\Gamma_k$, then $\psi_\alpha = \inn(g_0) \circ \phi_s$ for some $g_0 \in \bar{G}(\bar{k})$. It follows that $\psi_\alpha$ stabilizes $\bar{H}$. For $y \in \bar{Y}$ and $h \in \bar{H}$, we have $\alpha(y \cdot h) = \alpha(y) \cdot \psi_{\alpha}(h)$, hence $\pi(\alpha(y \cdot h)) = \pi(\alpha(y))$. It follows that $\alpha$ induces an $s$-semilinear automorphism $\varphi_\ast \alpha \in \SAut(\bar{Z}/X)$ as claimed. Uniqueness follows from surjectivity of $\pi$. In particular, the formation 
        \begin{equation*}
            E \to \SAut(\bar{Z}/X), \quad \alpha \mapsto \varphi_\ast \alpha
        \end{equation*}
    is a continuous homomorphism. Furthermore, for each $g \in \bar{G}(\bar{k})$, the automorphism $\varphi_\ast \rho_g: \bar{Z} \to \bar{Z}$ coincides with the $\bar{X}$-isomorphism $\rho_{\varphi(g)}$ taking $z$ to $z \cdot \varphi(g)$. Since $\varphi$ is surjective, it follows that $\varphi_\ast E$ contains $\{\rho_{g'} : g' \in \bar{G}'(\bar{k}) \}$ and that the sequence
        \begin{equation*}
            1 \to \bar{G}'(\bar{k}) \to \varphi_\ast E \to \Gamma_k \to 1
        \end{equation*}
    is exact. In addition, it is immediate from Construction \ref{cons:PushforwardOfDescentType} that $[\bar{Z},\varphi_\ast E]$ is in relation $\varphi_\ast$ with $\lambda$.

    Suppose that $[\bar{Y}',E'] \in \DType(X,L')$ is a descent type in relation $\varphi_\ast$ with $\lambda$. By definition, there is an $\varphi$-equivariant $\bar{X}$-morphism $\iota: \bar{Y} \to \bar{Y}'$ and a continuous homomorphism $\varpi: E \to E'$ fitting in a commutative diagram
		\begin{equation*}
			\xymatrix{
				1 \ar[r] & \bar{G}(\bar{k}) \ar[r] \ar[d]^{\varphi} & E \ar[r] \ar[d]^{\varpi} & \Gamma_k \ar@{=}[d] \ar[r] & 1 \\
				1 \ar[r] & \bar{G}'(\bar{k}) \ar[r] & E' \ar[r] & \Gamma_k \ar[r] & 1,
			}
		\end{equation*} 
    such that $\iota \circ \alpha = \varpi_{\alpha} \circ \iota$ for all $\alpha \in E$. Then $\iota$ induces a $\bar{G}'$-equivariant $\bar{X}$-isomorphism $i: \bar{Z} \xrightarrow{\cong} \bar{Y}'$, that is, $\iota = i \circ \pi$. For any $\alpha \in E$, we have
        \begin{equation*}
		i \circ \varphi_\ast \alpha \circ i^{-1} \circ \iota = i\circ \varphi_\ast \alpha \circ \pi = i\circ \pi \circ \alpha = \iota \circ \alpha = \varpi_\alpha \circ \iota.
	\end{equation*}
    Since $\iota$ is surjective, it follows that $i \circ \varphi_\ast \alpha \circ i^{-1} = \varpi_\alpha \in E'$, hence $[\bar{Y}',E'] = [\bar{Z},\varphi_\ast E] = \varphi_\ast \lambda$.
\end{proof}

\begin{lemm} \label{lemm:PushforwardOfDescentType2}
    Let $L = (\bar{G},\kappa)$ be a $k$-kernel, $G'$ a smooth commutative algebraic $k$-group, and $\varphi: L \to \lien(G')$ a morphism. Let $\lambda = [\bar{Y},E] \in \DType(X,L)$, and $\bar{Z}:=\varphi_\ast \bar{Y} = \bar{Y} \times_{\bar{k}}^{\bar{G}} \bar{G}'$, that is, the quotient of $\bar{Y} \times_{\bar{k}} \bar{G}'$ by the left action of $\bar{G}$ defined by $g \cdot (y,g') := (y \cdot g^{-1}, \varphi(g) g')$. For each $s \in \Gamma_k$ and each $s$-semilinear automorphism $\alpha \in E$, the automorphism $\alpha \times_{\bar{k}} s_{\#} \in \SAut(\bar{Y} \times_{\bar{k}} \bar{G}'/X)$ induces a unique $s$-semilinear automorphism $\varphi_\ast \alpha \in \SAut(\bar{Z}/X)$ fitting in a commutative diagram
        \begin{equation*}
			\xymatrix{
				\bar{Y} \ar[rr]^{\alpha} \ar[d]^{y \mapsto (y,1)} && \bar{Y} \ar[d]^{y \mapsto (y,1)} \\
				\bar{Y} \times_{\bar{k}} \bar{G}' \ar[rr]^{\alpha \times_{\bar{k}} s_{\#}} \ar[d]^{\pi} && \bar{Y} \times_{\bar{k}} \bar{G}' \ar[d]^{\pi} \\
				\bar{Z} \ar[rr]^{\varphi_\ast \alpha} &&  \bar{Z}
			}
		\end{equation*}
    where $\pi: \bar{Y} \to \bar{Z}$ is the canonical projection. Let $\varphi_\ast E$ be subgroup of $\SAut(\bar{Z}/X)$ generated by $\{\varphi_\ast \alpha : \alpha \in E\}$ and $\{\rho_{g'}: g' \in G'(\bar{k})\}$, then $\varphi_\ast \lambda:=[\bar{Z},\varphi_\ast E]$ is the unique element of $\DType(X,\lien(G'))$ which is in relation $\varphi_\ast$ with $\lambda$.
\end{lemm}
\begin{proof}
    Let $\psi_\alpha \in \SAut^{\gr}(\bar{G}/k)$ be as in the proof of Lemma \ref{lemm:PushforwardOfDescentType1}, that is, $\alpha \circ \rho_g \circ \alpha^{-1} = \rho_{\psi_\alpha(g)}$ for all $g \in \bar{G}(\bar{k})$. Using this identity and \eqref{eq:DescentType1}, we have, for all $g' \times \bar{G}'(\bar{k})$, that
        \begin{equation*}
            (\alpha \times_{\bar{k}} s_{\#}) \circ (\rho_{g}^{-1} \times_{\bar{k}} \rho_{\varphi(g)}) = (\rho_{\psi_\alpha(g)}^{-1} \times_{\bar{k}} \rho_{\tensor[^s]{\varphi(g)}{}})) \circ (\alpha \times_{\bar{k}} s_{\#}).
        \end{equation*}
    Since $\varphi$ is a morphism of $k$-kernels, there exists a continuous (set-theoretic) lifting $\phi: \Gamma_k \to \SAut^{\gr}(\bar{G}/k)$ such that $\varphi \circ \phi_t = t_{\#} \circ \varphi$ for all $t \in \Gamma_k$. Then $\psi_\alpha = \inn(g_0) \circ \phi_s$ for some $g_0 \in \bar{G}(\bar{k})$. It follows that $\tensor[^s]{\varphi(g)}{} = \varphi(\phi_s(g)) = \varphi(g_0^{-1} \psi_\alpha(g) g_0) = \varphi(\psi_\alpha(g))$ since $G'$ is commutative. Thus,
        \begin{equation*}
            \pi \circ (\alpha \times_{\bar{k}} s_{\#}) \circ (\rho_{g}^{-1} \times_{\bar{k}} \rho_{\varphi(g)}) = \pi \circ (\rho_{\psi_\alpha(g)}^{-1} \times_{\bar{k}} \rho_{\psi_\alpha(g)})) \circ (\alpha \times_{\bar{k}} s_{\#}) = \pi \circ (\alpha \times_{\bar{k}} s_{\#}).
        \end{equation*}
    From this, we conclude that $\alpha \times_{\bar{k}} s_{\#}$ induces an $s$-semilinear automorphism $\varphi_\ast \alpha \in \SAut(\bar{Z}/X)$ as stated. Uniqueness follows from surjectivity of $\pi$. In particular, the formation
        \begin{equation*}
            E \to \SAut(\bar{Z}/X), \quad \alpha \mapsto \varphi_\ast \alpha
        \end{equation*}
    is a continuous homomorphism. Furthermore, for each $g \in \bar{G}(\bar{k})$, the automorphism $\varphi_\ast \rho_g: \bar{Z} \to \bar{Z}$ coincides with the $\bar{X}$-isomorphism $\rho_{\varphi(g)}$ taking $z$ to $z \cdot \varphi(g)$. Let us show that the sequence
        \begin{equation*}
            1 \to G'(\bar{k}) \to \varphi_\ast E \to \Gamma_k \to 1
        \end{equation*}
    is exact. Clearly, only exactness at the term $\varphi_\ast E$ needs to be checked. First, we note that for any $g' \in G'(\bar{k})$, the diagram
        \begin{equation*}
		\xymatrix{
				\bar{Y} \times_{\bar{k}} \bar{G}' \ar[rr]^{\id_{\bar{Y}} \times_{\bar{k}} \rho_{g'}} \ar[d]^{\pi} && \bar{Y} \times_{\bar{k}} \bar{G}' \ar[d]^{\pi} \\
				\bar{Z} \ar[rr]^{\rho_{g'}} &&  \bar{Z}
			}
	\end{equation*}
    commutes. It follows that, for any $s \in \Gamma_k$ and any
    and any $s$-semilinear automorphism $\alpha \in E$,
        \begin{align*}
            \varphi_\ast \alpha \circ \rho_{g'} \circ \pi & = \varphi_\ast \alpha \circ \pi \circ (\id_{\bar{Y}} \times_{\bar{k}} \rho_{g'}) \\
            & = \pi \circ (\alpha \times_{\bar{k}} s_{\#}) \circ (\id_{\bar{Y}} \times_{\bar{k}} \rho_{g'}) \\
            & = \pi \circ (\id_{\bar{Y}} \times_{\bar{k}} \rho_{\tensor[^s]{g}{}'}) \circ (\alpha \times_{\bar{k}} s_{\#}), & \text{by \eqref{eq:DescentType1}}, \\
            & = \rho_{\tensor[^s]{g}{}'} \circ \pi \circ (\alpha \times_{\bar{k}} s_{\#}) \\
            & = \rho_{\tensor[^s]{g}{}'} \circ \varphi_\ast \alpha \circ \pi,
        \end{align*}
    hence $\varphi_\ast \alpha \circ \rho_{g'} = \rho_{\tensor[^s]{g}{}'} \circ \varphi_\ast \alpha$ because $\pi$ is surjective. Thus, $\varphi_\ast E$ normalizes $G'(\bar{k})$, that is, any element $\beta \in \varphi_\ast E$ can be written as $\beta = \varphi_\ast \alpha \circ \rho_{g'}$ for some $\alpha \in E$ and $g' \in G'(\bar{k})$. If the image of $\beta$ in $\Gamma_k$ is $1$, then $\alpha = \rho_g$ for some $g \in G(\bar{k})$, hence $\beta = \varphi_\ast \rho_g \circ \rho_{g'} = \rho_{g'\varphi(g)}$, which is what needed to be proved. It is now immediate from Construction \ref{cons:PushforwardOfDescentType} that the descent type $[\bar{Z},\varphi_\ast E]$ is in relation $\varphi_\ast$ with $\lambda$.

    Suppose that $[\bar{Y}',E'] \in \DType(X,\lien(G'))$ is a descent type in relation $\varphi_\ast$ with $\lambda$. By definition, there is an $\varphi$-equivariant $\bar{X}$-morphism $\iota: \bar{Y} \to \bar{Y}'$ and a continuous homomorphism $\varpi: E \to E'$ fitting in a commutative diagram
		\begin{equation*}
			\xymatrix{
				1 \ar[r] & \bar{G}(\bar{k}) \ar[r] \ar[d]^{\varphi} & E \ar[r] \ar[d]^{\varpi} & \Gamma_k \ar@{=}[d] \ar[r] & 1 \\
				1 \ar[r] & G'(\bar{k}) \ar[r] & E' \ar[r] & \Gamma_k \ar[r] & 1,
			}
		\end{equation*} 
    such that $\iota \circ \alpha = \varpi_{\alpha} \circ \iota$ for all $\alpha \in E$. Then $\iota$ induces a $\bar{G}'$-equivariant $\bar{X}$-isomorphism $i: \bar{Z} \xrightarrow{\cong} \bar{Y}'$. Denote by $\gamma: \bar{Y}' \times_{\bar{k}} \bar{G}' \to \bar{Y}'$ the action of $\bar{G}'$, then $i \circ \pi = \gamma \circ (\iota \times_{\bar{k}} \id_{\bar{G}'})$. As we have seen above, any element $\beta \in \varphi_\ast E$ can be written as $\beta = \varphi_\ast \alpha \circ \rho_{h} = \rho_{\tensor[^s]{h}{}} \circ \varphi_\ast \alpha$ for some $\alpha \in E$ and $h \in G'(\bar{k})$ (where $s \in \Gamma_k$ is the image of $\alpha$). Since the extension $E'$ is bound by $\lien(G')$, its conjugation action on $\bar{G'}$ is compatible with the natural Galois action, hence $\varpi_\alpha \circ \rho_h = \rho_{\tensor[^s]{h}{}} \circ \varpi_\alpha$ and $\gamma \circ (\varpi_\alpha \times_{\bar{k{}}} s_{\#}) = \varpi_\alpha \circ \gamma$. Thus
        \begin{align*}
            i \circ \beta \circ \pi & = (i \circ \rho_{\tensor[^s]{h}{}}) \circ (\varphi_\ast \alpha \circ \pi) \\
            & = \rho_{\tensor[^s]{h}{}} \circ i \circ \pi \circ (\alpha \times_{\bar{k}} s_{\#}), & \text{since $i$ is $\bar{G}'$-equivariant}, \\
            & = \rho_{\tensor[^s]{h}{}} \circ \gamma \circ (\iota \times_{\bar{k}} \id_{\bar{G}'}) \circ (\alpha \times_{\bar{k}} s_{\#}) \\
            & = \gamma \circ (\id_{\bar{Y}'} \times_{\bar{k}} \rho_{\tensor[^s]{h}{}}) \circ (\iota \times_{\bar{k}} \id_{\bar{G}'}) \circ (\alpha \times_{\bar{k}} s_{\#}), & \text{since } G' \text{ is commutative},\\
            & = \gamma \circ (\varpi_\alpha \times_{\bar{k{}}} s_{\#})\circ (\iota \times_{\bar{k}} \rho_h), & \text{by the equalities $\iota \circ \alpha = \varpi_\alpha \circ \iota$ and \eqref{eq:DescentType1}}, \\
            & = \omega_{\alpha} \circ \gamma \circ (\id_{\bar{Y}'} \times_{\bar{k}} \rho_h)  \circ (\iota \times_{\bar{k}} \id_{\bar{G}'}), \\
            & = \varpi_\alpha \circ \rho_h \circ \gamma \circ (\iota \times_{\bar{k}} \id_{\bar{G}'}), & \text{since } G' \text{ is commutative},\\
            & = \varpi_\alpha \circ \rho_h \circ i \circ \pi.
        \end{align*}
    Since $\pi$ is surjective, it follows that $i \circ \beta = \varpi_\alpha \circ \rho_h \circ i$, hence $i \circ \beta \circ i^{-1} = \varpi_\alpha \circ \rho_h \in E'$. This shows that $[\bar{Y}',E'] = [\bar{Z},\varphi_\ast E] = \varphi_\ast \lambda$.
\end{proof}

\begin{lemm} \label{lemm:DescentTypePullbackVsPushforward}
    Constructions \ref{cons:PullbackOfDescentType} and \ref{cons:PushforwardOfDescentType} commute. More precisely, let $f: X' \to X$ be a morphism of non-empty $k$-varieties, let $L = (\bar{G},\kappa)$ and $L' = (\bar{G}',\kappa')$ be $k$-kernels, and let $\varphi: L \to L'$ be a morphism. If a descent type $\lambda' = [\bar{Y}',E'] \in \DType(X,L')$ is in relation $\varphi_\ast$ with a descent type $\lambda = [\bar{Y},E] \in \DType(X,L)$, then $f^\ast \lambda' \in \DType(X',L')$ is in relation $\varphi_\ast$ with $f^\ast \lambda \in \DType(X',L)$. In particular, if either $\bar{G}'$ is commutative or $\varphi$ is surjective, then we have a commutative diagram
		\begin{equation*}
		\xymatrix{
			\DType(X,L) \ar[r]^{\varphi_\ast} \ar[d]^{f^\ast}& \DType(X,L') \ar[d]^{f^\ast}\\
			\DType(X',L) \ar[r]^{\varphi_\ast} & \DType(X',L').
		}
		\end{equation*}
\end{lemm}
\begin{proof}
    That $\lambda'$ is in relation $\varphi_\ast$ with $\lambda$ means there exist an $\varphi$-equivariant morphism $\iota: \bar{Y} \to \bar{Y}'$ and a continuous homomorphism $\varpi: E \to E'$ satisfying the conditions in Definition \ref{defn:DescentType} \ref{defn:DescentType2}. Recall that there are isomorphisms $f^\ast: E \xrightarrow{\cong} f^\ast E$ and $f^\ast: E' \xrightarrow{\cong} f^\ast E'$. We define the continuous homomorphism $\varpi': f^\ast E \to f^\ast E'$ by the formula 
		\begin{equation*}
			\forall \alpha \in E, \quad  \varpi'_{f^\ast \alpha}:= f^\ast \varpi_{\alpha} = \varpi_{\alpha} \times_{\bar{X}} s_{\#}'
		\end{equation*}
    (where $s \in \Gamma_k$ is the image of $\alpha$). Then $f^\ast \iota \circ f^\ast \alpha = (\iota \times_{\bar{X}} \id_{\bar{X}'}) \circ (\alpha \times_{\bar{X}} s_{\#}') = (\varpi_\alpha \times_{\bar{X}}  s_{\#}') \circ (\iota \times_{\bar{X}} \id_{\bar{X}'}) = \varpi'_{f^{\ast}\alpha} \circ f^\ast \iota$. This means the pair $(f^\ast \iota, \varpi'): (f^\ast \bar{Y}, f^\ast E) \to (f^\ast \bar{Y}', f^\ast E')$ is a morphism of descent data compatible with $\varphi$ in the sense of Definition \ref{defn:DescentType} \ref{defn:DescentType2}, {\em i.e.}, $f^\ast \lambda'$ is in relation $\varphi_\ast$ with $f^\ast \lambda$.
\end{proof}

\subsection{D\'evissage} \label{subsec:Devissage}

Let $X$ be a variety over a perfect field $k$. Given a torsor $Y \to X$ under an algebraic $k$-group $G$ and a normal Zariski closed subgroup $H \subseteq G$, we may put $Z:=Y/H$ to obtain a torsor $Y \to Z$ under $H$ and a torsor $Z \to X$ under $G/H$. A d\'evissage argument is something that allows one to deduce a property for the ``total'' torsor $Y \to X$ from the same property for the above ``intermediate'' torsors. A difficulty in deploying such an argument is that when we have a $(G/H)$-torsor $Z' \to X$ of the same type as $Z$ (by virtue of Proposition \ref{prop:DescentType} \ref{prop:DescentType3}, this means $Z'$ is isomorphic to the twist $Z^\tau$ of $Z$ by a cocycle $\tau \in \Z^1(k,G/H)$), we do not necessarily have a torsor $Y' \to Z'$ (under some $k$-form of $H$) such that the composite $Y' \to Z' \to X$ is a torsor of the same type as $Y$. In fact, this is equivalent to the existence of a Galois cocycle $\sigma \in \Z^1(k,G)$ lifting $\tau$ (and in the case where such a cocycle $\sigma$ exists, $Y'$ is isomorphic to the twisted $X$-torsor $Y^{\sigma}$).

When torsors are replaced by descent types, the formalism for a d\'evissage argument turns out to be somewhat simpler. We shall study this problem in this paragraph by establishing some sort of ``long exact sequence'' associated with an exact sequence $1 \to H \to G \to G/H \to 1$ for the non-abelian $\H^2$ set as well as for the set $\DType$.

Let $L = (\bar{G},\kappa)$ and $L' = (\bar{G}',\kappa')$ be $k$-kernels, $\varphi: L \to L'$ a surjective morphism, and $\bar{H}:=\Ker(\bar{G} \xrightarrow{\varphi} \bar{G}')$. Then, the relation $\varphi_\ast: \H^2(k,L) \multimap \H^2(k,L')$ is a map. Let $E$ (resp. $E'$) be a (topological) extension of $\Gamma_k$ by $\bar{G}(\bar{k})$ (resp. $\bar{G}'(\bar{k})$) such that $[E'] = \varphi_\ast[E]$ in $\H^2(k,L')$, that is, there is a continuous surjection $\phi: E \to E'$ fitting in a commutative diagram
    \begin{equation*}
		\xymatrix{
                1 \ar[r] & \bar{G}(\bar{k}) \ar[r] \ar@{->>}[d]^{\varphi} & E \ar@{->>}[d]^{\phi} \ar[r]^{q} & \Gamma_k \ar@{=}[d] \ar[r] & 1 \\
                1 \ar[r] & \bar{G}'(\bar{k}) \ar[r] & E' \ar[r]^{q'} & \Gamma_k \ar[r] & 1.	
		}
	\end{equation*}

\begin{lemm} \label{lemm:Devissage}
    With the above notations, suppose that $E'$ has a continuous homomorphic section $\varsigma$ (that is, the class $[E'] \in \H^2(k,L')$ is neutral). Then, we have a commutative diagram
	\begin{equation*}
		\xymatrix{
			1 \ar[r] & \bar{H}(\bar{k}) \ar[r] \ar@{_{(}->}[d] & E_\varsigma \ar[r] \ar@{_{(}->}[d] & \Gamma_k \ar[r] \ar@{=}[d] & 1 \\
			1 \ar[r] & \bar{G}(\bar{k}) \ar[r] & E \ar[r]^q & \Gamma_k \ar[r] & 1	
		}
	\end{equation*}
    with exact rows, where $E_\varsigma:=\phi^{-1}(\varsigma(\Gamma_k)) \subseteq E$. In particular, there exist a $k$-kernel $L_\varsigma = (\bar{H},\kappa_\varsigma)$ and a class $[E_\varsigma] \in \H^2(k,L_\varsigma)$ such that the inclusion $\bar{H} \hookrightarrow \bar{G}$ induces a morphism $i: L_\varsigma \to L$ of $k$-kernels, and $[E]$ is in relation $i_\ast: \H^2(k,L_\varsigma) \multimap \H^2(k,L)$ with $[E_\varsigma]$.
\end{lemm}
\begin{proof}
    Let us show that the restriction $q|_{E_\varsigma}: E_\varsigma \to \Gamma_k$ is surjective. Indeed, let $s \in \Gamma_k$ and lift it to an element $\alpha \in E$. Since both $\phi(\alpha) \in E'$ and $\varsigma_s \in E'$ lift $s \in \Gamma_k$, there exists $g' \in \bar{G}'(\bar{k})$ such that $\phi(\alpha) g' = \varsigma_s$. Since $\varphi$ is surjective, $g' = \varphi(g)$ for some $g \in \bar{G}(\bar{k})$. Then $\phi(\alpha g) = \varsigma_s \in \varsigma(\Gamma_k)$, or $\alpha g \in E_\varsigma$, which is an element satisfying $q(\alpha g) = s$.
	
    It remains to check exactness at $E_\varsigma$, that is, to show that $\bar{H}(\bar{k}) = \bar{G}(\bar{k}) \cap E_\varsigma$ inside $E$. Indeed, let $g \in \bar{G}(\bar{k})$ such that $\varphi(g) \in \varsigma(\Gamma_k) \subseteq E$. That $\varphi(g) \in E$ is a lifting of $1 \in \Gamma_k$ forces $\varphi(g) = \varsigma_1 = 1$, or $g \in \bar{H}(\bar{k})$.
\end{proof}

Let $X$ be a quasi-projective, reduced, and geometrically connected variety over $k$. We recall from Lemma \ref{lemm:PushforwardOfDescentType1} that for each descent type $\lambda = [\bar{Y},E] \in \DType(X,L)$, the pushforward type $\varphi_\ast \lambda \in \DType(X,L')$ from Construction \ref{cons:PushforwardOfDescentType} is represented by a pair $(\bar{Z},E')$, where $\bar{Z} := \varphi_\ast \bar{Y} = \bar{Y}/\bar{H}$ and $E':=\varphi_\ast E$. Indeed, $[E'] = \varphi_\ast [E]$ in $\H^2(k,L')$.

\begin{prop} \label{prop:Devissage}
    With the above notations, suppose that $Z \to X$ is a torsor under a $k$-form $G'$ of $L'$, of descent type $\dtype(Z) = \varphi_\ast \lambda$, which defines a continuous homomorphic section $\varsigma$ of $E'$ ({\em cf.} Proposition \ref{prop:DescentType} \ref{prop:DescentType1}). Let $E_Z =: E_{\varsigma} \subseteq E$, $L_Z:=L_\varsigma = (\bar{H},\kappa_\varsigma)$, and $i: L_Z \to L$ be as in Lemma \ref{lemm:Devissage}.
    \begin{enumerate}
	\item \label{prop:Devissage1} We have $E_Z = E \cap \SAut(\bar{Y}/Z)$ inside $\SAut(\bar{Y}/X)$. In particular, $\lambda_Z:= [\bar{Y},E_Z]$ is a descent type on $Z$ bound by $L_Z$.
		
    \item \label{prop:Devissage2} For every torsor $Y \to Z$ under a $k$-form $H$ of $L_Z$, of descent type $\lambda_Z$, the composite $Y \to Z \to X$ is a torsor under a $k$-form $G$ of $L$, of descent type $\lambda$. In particular, the inclusion $\bar{H} \hookrightarrow \bar{G}$ is defined over $k$, and $G' = G/H$.
\end{enumerate}
\end{prop}
\begin{proof}
    We show \ref{prop:Devissage1}. For each $\alpha \in E_Z$, one has $\varphi_\ast \alpha = \varsigma_s$ for some $s \in \Gamma_k$. Hence, the diagram
		\begin{equation*}
			\xymatrix{
				\bar{Y} \ar[r]^{\alpha} \ar[d] & \bar{Y} \ar[d] \\
				\bar{Z} \ar[r]^{\varsigma_s} & \bar{Z}
			}
		\end{equation*}
    commutes. Since $\varsigma_s \in \SAut(\bar{Z}/Z)$, we have $\alpha \in \SAut(\bar{Y}/Z)$. Conversely, if $\alpha \in E \cap \SAut(\bar{Y}/Z)$, then $\varphi_\ast \alpha \in \SAut(\bar{Z}/Z) = \varsigma(\Gamma_k)$ by virtue of the exact sequence \eqref{eq:DescentTypeExactSequenceSAut} (take $\bar{Y} = \bar{Z}$ and $X = Z$). It follows that $\alpha \in E_\varsigma = E_Z$.
	
    Let us now show \ref{prop:Devissage2}. The torsor $Z \to X$ defines a homomorphic section of $E_Z$, {\em a fortiori} a homomorphic section of $E$. In view of Proposition \ref{prop:DescentType} \ref{prop:DescentType1}, this section yields a $k$-form $G$ of $L$, such that $\bar{H} \hookrightarrow \bar{G}$ descends into a $k$-morphism $H \subseteq G$, and $G' = G/H$. Furthermore, it gives rise to a torsor $Y \to X$ under $G$ such that $Y/H = Z$.
\end{proof}
\section{The abelian case} \label{sec:Abelian}

Let $X$ be a non-empty variety over a perfect field $k$. In this \S, we study the set $\DType(X,G):=\DType(X,\lien(G))$ of descent type on $X$ bound by $\lien(G)$, where $G$ is a smooth commutative algebraic group over $k$. Our aim is to prove Theorem \ref{customthm:Abelian}.

\subsection{External product} \label{subsec:ExternalProduct}

To define a group structure on abelian descent types in a conceptual way, we introduce the {\em external product} of descent types. First, we define this on the level of kernels. Let $L_1 = (\bar{G}_1,\kappa_1)$ and $L_2 = (\bar{G}_2,\kappa_2)$ be $k$-kernels. If $s \in \Gamma_k$ and $\phi_1 \in \SAut^{\gr}(\bar{G}_1/k)$, $\phi_2 \in \SAut^{\gr}(\bar{G}_2/k)$ are $s$-semilinear automorphisms, then $\phi_1 \times_{\bar{k}} \phi_2 \in \SAut^{\gr}((\bar{G}_1 \times_{\bar{k}} \bar{G}_2)/k)$. Thus, we have a natural inclusion
    \begin{equation*}
        \SAut^{\gr}(\bar{G}_1/k) \times_{\Gamma_k} \SAut^{\gr}(\bar{G}_2/k) \hookrightarrow \SAut^{\gr}((\bar{G}_1 \times_{\bar{k}} \bar{G}_2)/k),
    \end{equation*}
which induces a homomorphism
    \begin{equation*}
        \SOut(\bar{G}_1/k) \times_{\Gamma_k} \SOut(\bar{G}_2/k) \to \SOut((\bar{G}_1 \times_{\bar{k}} \bar{G}_2)/k).
    \end{equation*}
Denote by $\kappa_1 \boxtimes \kappa_2$ the composite of this homomorphism with
    \begin{equation*}
        \Gamma_k \xrightarrow{(\kappa_1,\kappa_2)} \SOut(\bar{G}_1/k) \times_{\Gamma_k} \SOut(\bar{G}_2/k).
    \end{equation*}
Then $L_1 \times L_2 := (\bar{G}_1 \times_{\bar{k}} \bar{G}_2, \kappa_1 \boxtimes \kappa_2)$ is a $k$-kernel. If $G_1$ and $G_2$ are respective $k$-forms of $L_1$ and $L_2$ (that is, $L_1 = \lien(G_1)$ and $L_2 = \lien(G_2)$), then $G_1 \times_k G_2$ is a $k$-form of $L_1 \times L_2$.

It is immediate to see that the operation ``$\times$'' on $k$-kernels is associative. Namely, if $L_3 = (\bar{G}_3, \kappa_3)$ is a third $k$-kernel, then the canonical isomorphism $(\bar{G}_1 \times_{\bar{k}} \bar{G}_2)  \times_{\bar{k}} \bar{G}_3 \cong \bar{G}_1 \times_{\bar{k}} (\bar{G}_2 \times_{\bar{k}} \bar{G}_3)$ of algebraic $\bar{k}$-groups induces an isomorphism $(L_1 \times L_2) \times L_3 \cong L_1 \times (L_2 \times L_3)$ of $k$-kernels. In fact, they are both isomorphic to the $k$-kernel $L_1 \times L_2 \times L_3 = (\bar{G}_1 \times_{\bar{k}} \bar{G}_2 \times_{\bar{k}} \bar{G}_3, \kappa_1 \boxtimes \kappa_2 \boxtimes \kappa_3)$, where $\kappa_1 \boxtimes \kappa_2 \boxtimes \kappa_3$ is the composite
    \begin{equation*}
        \Gamma_k \xrightarrow{(\kappa_1,\kappa_2,\kappa_3)} \SOut(\bar{G}_1/k) \times_{\Gamma_k} \SOut(\bar{G}_2/k) \times_{\Gamma_k} \SOut(\bar{G}_3/k) \to \SOut((\bar{G}_1 \times_{\bar{k}} \bar{G}_2 \times_{\bar{k}} \bar{G}_3)/k),
    \end{equation*}
the second arrow being induced by the natural inclusion
    \begin{align*}
        \SAut^{\gr}(\bar{G}_1/k) \times_{\Gamma_k} \SAut^{\gr}(\bar{G}_2/k) \times_{\Gamma_k} \SAut^{\gr}(\bar{G}_3/k) & \hookrightarrow \SAut^{\gr}((\bar{G}_1 \times_{\bar{k}} \bar{G}_2 \times_{\bar{k}} \bar{G}_3)/k), \\
        (\phi_1,\phi_2,\phi_3) & \mapsto \phi_1 \times_{\bar{k}} \phi_2 \times_{\bar{k}} \phi_3.
    \end{align*}
In what follows, we abusively identify $(L_1 \times L_2) \times L_3$ and $L_1 \times (L_2 \times L_3)$ to $L_1 \times L_2 \times L_3$. We also note that the neutral element for the operation ``$\times$'' is the {\em trivial} $k$-kernel $L_{\triv}:=\lien(\{1\})$. 

Let $X_1$ and $X_2$ be non-empty $k$-varieties. If $\bar{Y}_1$ and $\bar{Y}_2$ are respective non-empty $\bar{k}$-varieties equipped with morphisms $\bar{Y}_1 \to \bar{X}_1$ and $\bar{Y}_2 \to \bar{X}_2$, we have a natural inclusion
    \begin{equation} \label{eq:ExternalProduct}
        \SAut(\bar{Y}_1/X_1) \times_{\Gamma_k} \SAut(\bar{Y}_2/X_2) \hookrightarrow \SAut((\bar{Y}_1 \times_{\bar{k}} \bar{Y}_2)/ (X_1 \times_k X_2)), \quad (\alpha_1,\alpha_2) \mapsto \alpha_1 \times_{\bar{k}} \alpha_2.
    \end{equation}
Indeed, if $s \in \Gamma_k$ and if $\alpha_1 \in \Aut(\bar{Y}_1/X_1)$ and $\alpha_2 \in \Aut(\bar{Y}_2/X_2)$ are $s$-semilinear, then the automorphism $\alpha_1 \times_{\bar{k}} \alpha_2 \in \Aut((\bar{Y}_1 \times_{\bar{k}} \bar{Y}_2)/ (X_1 \times_k X_2))$ is $s$-semilinear. Now, let $\lambda_1 = [\bar{Y}_1,E_1]$ (resp. $\lambda_2 = [\bar{Y}_2,E_2]$) be a descent type on $X_1$ (resp. $X_2$) bound by a $k$-kernel $L_1 = (\bar{G}_1,\kappa_1)$ (resp. $L_2 = (\bar{G}_2,\kappa_2)$). This means in particular that $E_1$ (resp. $E_2$) is a topological extension of $\Gamma_k$ by $\bar{G}_1(\bar{k})$ (resp. $\bar{G}_2(\bar{k})$) bound by the $k$-kernel $L_1$ (resp. $L_2$). One checks with ease that the extension
    \begin{equation*}
        1 \to \bar{G}_1(\bar{k}) \times \bar{G}_2(\bar{k}) \to E_1 \times_{\Gamma_k} E_2 \to \Gamma_k \to 1
    \end{equation*}
is bound by the product kernel $L_1 \times L_2$. Let us denote by $E_1 \boxtimes E_2$ the image of $E_1 \times_{\Gamma_k} E_2$ by the inclusion \eqref{eq:ExternalProduct}. Then $[\bar{Y}_1 \times_{\bar{k}} \bar{Y}_2, E_1 \boxtimes E_2]$ is a descent type on $X_1 \times_k X_2$ bound by $L_1 \times L_2$.

\begin{defn} \label{defn:ExternalProduct}
    We call $\lambda_1 \times \lambda_2:=[\bar{Y}_1 \times_{\bar{k}} \bar{Y}_2, E_1 \boxtimes E_2] \in \DType(X_1 \times_k X_2, L_1 \times L_2)$ the {\em external product} of $\lambda_1$ and $\lambda_2$.
\end{defn}
\begin{lemm} \label{lemm:ExternalProduct}
    \begin{enumerate}
        \item \label{lemm:ExternalProduct1} For $i \in \{1,2,3\}$, let $X_i$ be a non-empty $k$-variety, $L_i = (\bar{G}_i,\kappa_i)$ a $k$-kernel, and $\lambda_i = [\bar{Y}_i,E_i] \in \DType(X_i,L_i)$. Under the identification $(L_1 \times L_2) \times L_3 = L_1 \times (L_2 \times L_3) = L_1 \times L_2 \times L_3$, we have $(\lambda_1 \times \lambda_2) \times \lambda_3 = \lambda_1 \times (\lambda_2 \times \lambda_3)$ in $\DType(X_1 \times_k X_2 \times_k X_3, L_1 \times L_2 \times L_3)$.

        \item \label{lemm:ExternalProduct2} The set $\DType(\Spec(k), L_{\triv})$ has a unique element $\lambda_{\triv}$. Furthermore, let $X$ be a non-empty $k$-variety, $L = (\bar{G},\kappa)$ a $k$-kernel, and $\lambda = [\bar{Y},E] \in \DType(X,L)$. Under the identifications $L_{\triv} \times L = L \times L_{\triv} = L$, one has $\lambda_{\triv} \times \lambda = \lambda \times \lambda_{\triv} = \lambda$.
    \end{enumerate}
\end{lemm}
\begin{proof}
    \begin{enumerate}
        \item We consider the $(\bar{X}_1 \times_{\bar{k}} \bar{X}_2 \times_{\bar{k}} \bar{X}_3)$-isomorphism 
            \begin{equation*}
                \iota: (\bar{Y}_1 \times_{\bar{k}} \bar{Y}_2) \times_{\bar{k}} \bar{Y}_3 \xrightarrow{\cong} \bar{Y}_1 \times_{\bar{k}} (\bar{Y}_2 \times_{\bar{k}} \bar{Y}_3), \quad ((y_1,y_2),y_3) \mapsto (y_1,(y_2,y_3))
            \end{equation*}
        of torsors under $\bar{G}_1 \times_{\bar{k}} \bar{G}_2 \times_{\bar{k}} \bar{G}_3$. An element of $(E_1 \boxtimes E_2) \boxtimes E_3$ has the form $(\alpha_1 \times_{\bar{k}} \alpha_2) \times_{\bar{k}} \alpha_3$, where $(\alpha_1,\alpha_2,\alpha_3) \in E_1 \times_{\Gamma_k} E_2 \times_{\Gamma_k} E_3$. Then, we have
            \begin{equation*}
                \iota \circ ((\alpha_1 \times_{\bar{k}} \alpha_2) \times_{\bar{k}} \alpha_3) \circ \iota^{-1} = \alpha_1 \times_{\bar{k}} (\alpha_2 \times_{\bar{k}} \alpha_3) \in E_1 \boxtimes (E_2 \boxtimes E_3).
            \end{equation*}
        It follows from definition that $(\lambda_1 \times \lambda_2) \times \lambda_3 = \lambda_1 \times (\lambda_2 \times \lambda_3)$.

        \item Since $\Aut(\Spec(\bar{k})/\Spec(k)) = \{\Spec(s^{-1}): s \in \Gamma_k\} := E_{\triv}$, it is immediate that the only element of $\DType(\Spec(k),L_{\triv})$ is $\lambda_{\triv} = (\Spec(\bar{k}),E_{\triv})$. For $\lambda = [\bar{Y},E] \in \DType(X,L)$, the first projection $\pi: \bar{Y} \times_{\bar{k}} \bar{k} \to \bar{Y}$ is an $\bar{X}$-isomorphism of torsors under $\bar{G}$. Furthermore, for any $s \in \Gamma_k$ and any $s$-semilinear automorphism $\alpha \in E$, one has $\pi \circ (\alpha \times_{\bar{k}} \Spec(s^{-1})) \circ \pi^{-1} = \alpha \in E$. It follows that $\lambda \times \lambda_{\triv} = \lambda$. Similarly, $\lambda_{\triv} \times \lambda = \lambda$.
    \end{enumerate}
\end{proof}

Let us now show that the external product behaves in a functorial way.

\begin{prop} \label{prop:ExternalProduct}
    For $i \in \{1,2\}$, let $f_i: X_i' \to X_i$ be a morphism between non-empty $k$-varieties, $\varphi_i: L_i \to L_i'$ a morphism of $k$-kernels, and $\lambda_i = [\bar{Y}_i,E_i] \in \DType(X_i,L_i)$.
    \begin{enumerate}
        \item \label{prop:ExternalProduct1} We have $(f_1 \times_k f_2)^\ast(\lambda_1 \times \lambda_2) = f_1^\ast \lambda_1 \times f_2^\ast \lambda_2$ in $\DType(X_1' \times_k X_2',L_1 \times L_2)$.

        \item \label{prop:ExternalProduct2} The underlying morphism $\varphi_1 \times_{\bar{k}} \varphi_2$ of algebraic $\bar{k}$-groups is a morphism $L_1 \times L_2 \to L_1' \times L_2'$  of $k$-kernels. For $i \in \{1,2\}$, let $\lambda_i' = [\bar{Y}_i',E_i'] \in \DType(X_i,L_i')$ be a descent type which is in relation $\varphi_{i\ast}$ with $\lambda_i$. Then, $\lambda_1' \times \lambda_2'$ is in relation $(\varphi_1 \times_{\bar{k}} \varphi_2)_\ast$ with $\lambda_1 \times \lambda_2$. In particular, if either the morphisms $\varphi_1$ and $\varphi_2$ are surjective or the underlying $\bar{k}$-groups of $L_1$ and $L_2$ are commutative, then 
		\begin{equation*}
                (\varphi_1 \times_{\bar{k}} \varphi_2)_\ast(\lambda_1 \times \lambda_2) = \varphi_{1\ast} \lambda_1 \times \varphi_{2\ast} \lambda_2 \in \DType(X_1 \times_k X_2,L_1' \times L_2').
		\end{equation*}
  
        \item \label{prop:ExternalProduct3} If $\lambda_i = \dtype(Y_i)$ for some torsor $Y_i \to X_i$ under a $k$-form $G_i$ of $L_i$, $i=1,2$, then $\lambda_1 \times \lambda_2 = \dtype(Y_1 \times_k Y_2)$ in $\DType(L_1 \times L_2)$.
    \end{enumerate}
\end{prop}
\begin{proof}
\begin{enumerate}
    \item Let $X:=X_1 \times_k X_2$, $X':=X_1' \times_k X_2'$, $\bar{Y}:=\bar{Y}_1 \times_{\bar{k}} \bar{Y}_2$, $\bar{Y}_1':=\bar{Y}_1 \times_{\bar{X}_1} \bar{X}'_1$, and $\bar{Y}_2':=\bar{Y}_2 \times_{\bar{X}_2} \bar{X}'_2$. By definition, we have $(f_1 \times_k f_2)^\ast (\lambda_1 \times \lambda_2) = [\bar{Y} \times_{\bar{X}} \bar{X}', (f_1 \times_k f_2)^\ast(E_1 \boxtimes E_2)]$ and $f_1^\ast \lambda_1 \times f_2^\ast \lambda_2 = [\bar{Y}'_1 \times_{\bar{k}} \bar{Y}'_2, f_1^\ast E_1 \boxtimes f_2^\ast E_2]$. Consider the isomorphism of $\bar{X}'$-torsors
	\begin{equation*}
		\iota: \bar{Y} \times_{\bar{X}} \bar{X}' \xrightarrow{\cong} \bar{Y}'_1 \times_{\bar{k}} \bar{Y}'_2, \quad ((y_1,y_2),(x_1',x_2')) \mapsto ((y_1,x_1'),(y_2,x_2')).
	\end{equation*}
    An element $\alpha \in (f_1 \times_k f_2)^\ast(E_1 \boxtimes E_2)$ has the form $\alpha = (f_1 \times_k f_2)^\ast(\alpha_1 \times_{\bar{k}} \alpha_2)$, where $\alpha_i \in E_i$ for $i \in \{1,2\}$, and they have the same image in $\Gamma_k$ (let us denote it by $s$). For all $y_1' \in \bar{Y}_1'$ and $y_2' \in \bar{Y}_2'$, write them as $y_1' = (y_1,x_1')$ and $y_2' = (y_2,x_2')$ (where $y_1 \in \bar{Y}_1$, $x_1' \in \bar{X}_1'$, $y_2 \in \bar{Y}_2$, and $x_2' \in \bar{X}_2'$). Then, we have
		\begin{align*}
			(\iota \circ \alpha \circ \iota^{-1})(y_1',y_2') & = (\iota \circ \alpha)((y_1,y_2),(x_1',x_2')) \\
			& = \iota((\alpha_1(y_1),\alpha_2(y_2)),(\tensor[^s]{x}{}{}_1',\tensor[^s]{x}{}{}_2')) \\
			& = ((\alpha_1(y_1), \tensor[^s]{x}{}{}_1'),(\alpha_2(y_2), \tensor[^s]{x}{}{}_2')) \\
			& = ((f_1^\ast \alpha_1)(y_1'),(f_2^\ast \alpha_2)(y_2')),
		\end{align*} 
    or $\iota \circ \alpha \circ \iota^{-1} = f_1^\ast \alpha_1 \times_{\bar{k}} f_2^\ast \alpha_2 \in f_1^\ast E_1 \boxtimes f_2^\ast E_2$. It follows that the isomorphism $\iota$ identifies the descent types $(f_1 \times_k f_2)^\ast (\lambda_1 \times \lambda_2)$ and $f_1^\ast \lambda_1 \times f_2^\ast \lambda_2$.

    \item  For $i \in \{1,2\}$, let $L_i = (\bar{G}_i,\kappa_i)$ and $L_i' = (\bar{G}'_i,\kappa_i')$. That $\varphi_i$ is a morphism of $k$-kernels means there are respective continuous set-theoretic liftings $\phi_i: \Gamma_k \to \SAut^{\gr}(\bar{G}_i/k)$ and $\phi_i': \Gamma_k \to \SAut^{\gr}(\bar{G}'_i/k)$ of $\kappa_i$ and $\kappa_i'$ such that $\varphi_i \circ \phi_i(s) = \phi'_i(s) \circ \varphi_i$ for all $s \in \Gamma_k$. Let $\phi$ and $\phi'$ denote the respective composites
        \begin{equation*}
            \Gamma_k \xrightarrow{(\phi_1,\phi_2)} \SAut^{\gr}(\bar{G}_1/k) \times_{\Gamma_k} \SAut^{\gr}(\bar{G}_2/k) \xrightarrow{(\psi_1,\psi_2) \mapsto \psi_1 \times_{\bar{k}} \psi_2} \SAut^{\gr}((\bar{G}_1 \times_{\bar{k}} \bar{G}_2)/k)
        \end{equation*}
    and
        \begin{equation*}
            \Gamma_k \xrightarrow{(\phi_1',\phi_2')} \SAut^{\gr}(\bar{G}_1'/k) \times_{\Gamma_k} \SAut^{\gr}(\bar{G}_2'/k) \xrightarrow{(\psi_1,\psi_2) \mapsto \psi_1 \times_{\bar{k}} \psi_2} \SAut^{\gr}((\bar{G}'_1 \times_{\bar{k}} \bar{G}'_2)/k)
        \end{equation*}
    Then $\phi$ and $\phi'$ are respective continuous set-theoretic liftings of $\kappa_1 \boxtimes \kappa_2$ and $\kappa_1' \boxtimes \kappa_2'$. Furthermore, $(\varphi_1 \times_{\bar{k}} \varphi_2) \circ \phi(s) = \phi'(s) \circ (\varphi_1 \times_{\bar{k}} \varphi_2)$ for all $s \in \Gamma_k$. Hence, $\varphi_1 \times_{\bar{k}} \varphi_2$ is indeed a morphism $L_1 \times L_2 \to L_1' \times L_2'$ of $k$-kernels.

    Now, since $\lambda_i'$ is in relation $\varphi_{i\ast}$ with $\lambda_i$, there exist an $\varphi_i$-equivariant $\bar{X}$-morphism $\iota_i: \bar{Y}_i \to \bar{Y}_i'$ and a continuous homomorphism $\varpi_i: E_i \to E_i'$ satisfying the conditions in Definition \ref{defn:DescentType} \ref{defn:DescentType2}. Let $\iota:=\iota_1 \times_{\bar{k}} \iota_2$, and let $\varpi: E_1 \boxtimes E_2 \to E_1' \boxtimes E_2'$ be the continuous homomorphism $\alpha_1 \times_{\bar{k}} \alpha_2 \mapsto \varpi_1(\alpha_1) \times_{\bar{k}} \varpi_2(\alpha_2)$. Then the pair $(\iota,\varpi)$ defines a morphism of descent data $(\bar{Y}_1 \times_{\bar{k}} \bar{Y}_2, E_1 \boxtimes E_2) \to (\bar{Y}'_1 \times_{\bar{k}} \bar{Y}'_2, E'_1 \boxtimes E'_2)$ compatible with $\varphi_1 \times_{\bar{k}} \varphi_2$, {\em i.e.}, $\lambda_1' \times \lambda_2'$ is in relation $(\varphi_1 \times_{\bar{k}} \varphi_2)_\ast$ with $\lambda_1 \times \lambda_2$.

    \item For $i \in \{1,2\}$, the $k$-form $G_i$ defines a continuous homomorphic lifting $\phi_i: \Gamma_k \to \SAut^{\gr}(\bar{G}_i/k)$ of $\kappa_i$, and we have $\dtype(Y_i) = [\bar{Y}_i,E_i]$, where $E_i := \{\phi_i(s) \circ \rho_{g_i}: s \in \Gamma_k, g_i \in G_i(\bar{k})\}$. The semilinear action $\phi$ of $\Gamma_k$ on $\bar{G}_1 \times_{\bar{k}} \bar{G}_2$ {\em via} its $k$-form $G_1 \times_k G_2$ is given by the composite
        \begin{equation*}
            \Gamma_k \xrightarrow{(\phi_1,\phi_2)} \SAut^{\gr}(\bar{G}_1/k) \times_{\Gamma_k} \SAut^{\gr}(\bar{G}_2/k) \xrightarrow{(\psi_1,\psi_2) \mapsto \psi_1 \times_{\bar{k}} \psi_2} \SAut^{\gr}((\bar{G}_1 \times \bar{G}_2)/k),
        \end{equation*}
    and we have $\dtype(Y_1 \times_k Y_2) = [\bar{Y}_1 \times_{\bar{k}} \bar{Y}_2, E]$, where $E := \{\phi(s) \circ \rho_{g}: s \in \Gamma_k, g \in G_1(\bar{k}) \times G_2(\bar{k})\}$. For all $s \in \Gamma_k$, $g_1 \in G_1(\bar{k})$, and $g_2 \in G_2(\bar{k})$, we have
        \begin{equation*}
            (\phi_1(s) \circ \rho_{g_1}) \times_{\bar{k}} (\phi_2(s) \circ \rho_{g_2}) = \phi_s \circ \rho_{(g_1,g_2)}.
        \end{equation*}
    It follows that $E_1 \boxtimes E_2 = E$, or $\dtype(Y_1) \times \dtype(Y_2) = \dtype(Y_1 \times_k Y_2)$.
\end{enumerate}
\end{proof}

\subsection{The group law} \label{subsec:GroupLaw}

Let $X$ be a non-empty variety over a perfect field $k$, and $G$ a smooth commutative algebraic $k$-group. By commutativity, the multiplication morphism $\nabla: G \times_k G \to G$ is a morphism of algebraic $k$-groups. Let $\Delta: X \to X \times_k X$ denote the diagonal morphism.

\begin{defn} \label{defn:SumOfDescentType}
    Let $\lambda$ and $\lambda'$ be descent types on $X$ bound by $\lien(G)$. We define their {\em sum} $\lambda + \lambda' \in \DType(X,G)$ to be the descent type $\Delta^\ast \nabla_\ast (\lambda \times \lambda') = \nabla_\ast \Delta^\ast  (\lambda \times \lambda')$ ({\em cf.} Definition \ref{defn:ExternalProduct} and Lemma \ref{lemm:DescentTypePullbackVsPushforward}).
\end{defn} 
\begin{prop} \label{prop:SumOfDescentType}
    The operation ``$+$'' from Definition \ref{defn:SumOfDescentType} makes $\DType(X,G)$ an abelian group, functorial in $X$ and $G$. More precisely, we have the following claims.
	\begin{enumerate}
		\item \label{prop:SumOfDescentType1} The neutral element is the descent type of the trivial torsor $G_X = G \times_k X \to X$.
		
	   \item \label{prop:SumOfDescentType2} The inverse of a descent type $\lambda$ is $-\lambda = i_\ast \lambda$, where $i: G \to G$ is the inversion morphism.
	\end{enumerate}
\end{prop}
\begin{proof}
    First, we show associativity. Indeed, we have the standard properties
        \begin{equation} \label{eq:SumOfDescentType1}
            (\Delta \times_k \id_X) \circ \Delta = (\id_X \times_k \Delta) \circ \Delta \quad \text{and} \quad \nabla \circ (\nabla \times_k \id_G) = \nabla \circ (\id_G \times_k \nabla).
        \end{equation}
    Let $\lambda$, $\lambda'$, $\lambda'' \in \DType(X,G)$, then
        \begin{align*}
            (\lambda + \lambda') + \lambda'' & = \Delta^\ast \nabla_\ast (\Delta^\ast \nabla_\ast(\lambda \times \lambda') \times \lambda''), & \text{by Definition \ref{defn:SumOfDescentType}}, \\ 
            & = \Delta^\ast \nabla_\ast (\Delta \times_k \id_X)^\ast (\nabla \times_{k} \id_G)_\ast((\lambda \times \lambda') \times \lambda''), & \text{by Proposition \ref{prop:ExternalProduct}\ref{prop:ExternalProduct1} and \ref{prop:ExternalProduct2}},\\
            & = \Delta^\ast (\Delta \times_k \id_X)^\ast \nabla_\ast (\nabla \times_{k} \id_G)_\ast(\lambda \times (\lambda' \times \lambda'')), & \text{by Lemmata \ref{lemm:DescentTypePullbackVsPushforward} and \ref{lemm:ExternalProduct} \ref{lemm:ExternalProduct1}}, \\
            & = \Delta^\ast (\id_X \times_k \Delta)^\ast \nabla_\ast (\id_G \times_k \nabla)_\ast(\lambda \times (\lambda' \times \lambda'')), & \text{by \eqref{eq:SumOfDescentType1}}, \\
            & = \Delta^\ast \nabla_\ast (\id_X \times_k \Delta)^\ast (\id_G \times_k \nabla)_\ast(\lambda \times (\lambda' \times \lambda'')), & \text{by Lemma \ref{lemm:DescentTypePullbackVsPushforward}}, \\
            & = \Delta^\ast \nabla_\ast(\lambda \times \Delta^\ast \nabla_\ast(\lambda' \times \lambda'')), & \text{by Proposition \ref{prop:ExternalProduct}\ref{prop:ExternalProduct1} and \ref{prop:ExternalProduct2}},\\
            & = \lambda + (\lambda' + \lambda''), & \text{by Definition \ref{defn:SumOfDescentType}}.
        \end{align*}
    Next, we show commutativity. Let $\lambda = [\bar{Y},E]$ and $\lambda' = [\bar{Y}',E']$. Let $\iota: \bar{Y }\times \bar{Y}' \to \bar{Y}' \times \bar{Y}$ be the isomorphism $(y,y') \mapsto (y',y)$. An element of $E \boxtimes E'$ has the form $\alpha \times_{\bar{k}} \alpha'$, where $(\alpha,\alpha') \in E \times_{\Gamma_k} E'$. It is clear that $\iota \circ (\alpha \times_{\bar{k}} \alpha') \circ \iota^{-1} = \alpha' \times_{\bar{k}} \alpha \in E' \boxtimes E$, thus $\iota$ identifies $\lambda \times \lambda' = [\bar{Y} \times_{\bar{k}} \bar{Y}', E \boxtimes E']$ to $\lambda' \times \lambda = [\bar{Y}' \times_{\bar{k}} \bar{Y}, E' \boxtimes E]$. It follows that
        \begin{equation*}
            \lambda+\lambda' = \Delta^\ast \nabla_\ast(\lambda \times \lambda') = \Delta^\ast \nabla_\ast(\lambda' \times \lambda) = \lambda' + \lambda.
        \end{equation*}
    \begin{enumerate}
        \item We construct the neutral element for ``$+$''. Let $p: X \to \Spec(k)$ (resp. $e: \{1\} \to G$) denote the structure morphism (resp. the inclusion of the neutral element. We have the standard properties
            \begin{equation} \label{eq:SumOfDescentType2}
                (\id_X \times_k p) \circ \Delta = \id_X \quad \text{and} \quad \nabla \circ (\id_G \times_k e) = \id_G.
            \end{equation}
        Viewing $G$ as the trivial $k$-torsor under itself, we have $G_X = p^\ast G$ and $G = e_\ast \{1\}$. Hence
		\begin{equation} \label{eq:SumOfDescentType3}
			\dtype(G_X) = p^\ast \dtype(G) = p^\ast e_\ast \lambda_{\triv}.
		\end{equation}
        where $\lambda_{\triv} = \dtype(\{1\})$ is the unique element of $\DType(\Spec(k),\{1\})$ ({\em cf.} Lemma \ref{lemm:ExternalProduct} \ref{lemm:ExternalProduct2}). Now, for any $\lambda \in \DType(X,G)$, we have
		\begin{align*}
		      \lambda + \dtype(G_X) & = \Delta^\ast \nabla_\ast(\lambda \times G_X), & \text{by Definition \ref{defn:SumOfDescentType}},\\
			& = \Delta^\ast \nabla_\ast(\lambda \times p^\ast e_\ast \lambda_{\triv}), & \text{by \eqref{eq:SumOfDescentType3}}, \\
			& = \Delta^\ast \nabla_\ast (\id_X \times_k p)^\ast (\id_X \times_k e)_\ast (\lambda \times \lambda_{\triv}), & \text{by Proposition \ref{prop:ExternalProduct} \ref{prop:ExternalProduct1} and \ref{prop:ExternalProduct2}}, \\
			& = \Delta^\ast (\id_X \times_k p)^\ast \nabla_\ast (\id_X \times_k e)_\ast \lambda, & \text{by Lemmata \ref{lemm:DescentTypePullbackVsPushforward} and \ref{lemm:ExternalProduct} \ref{lemm:ExternalProduct2}}, \\
			& = \id_X^\ast (\id_G)_\ast \lambda, & \text{by \eqref{eq:SumOfDescentType2}},  \\
			& = \lambda,
		\end{align*}
        showing that $\dtype(G_X)$ is the neutral element for the operation ``$+$'' on $\DType(X,G)$.

        \item Finally, we construct the inverse (for ``$+$'') of a given element $\lambda = [\bar{Y},E] \in \DType(X,G)$. Since $G$ is commutative, the inversion $i: G \to G$ is a morphism of algebraic $k$-groups. Let $\pi: G \to \Spec(k)$ denote the structure morphism. Then, we have the standard properties
        \begin{equation} \label{eq:SumOfDescentType4}
            \nabla \circ (\id_X \times_k i) = e \circ (\pi \times_k \pi) \quad \text{and} \quad (p \times_k p) \circ \Delta = p.
        \end{equation}

        We observe that $p^\ast \lambda_{\triv} = [\bar{X},\{s_{\#}: s \in \Gamma_k\}]$ by Construction \ref{cons:PullbackOfDescentType}. The structure morphism $f: \bar{Y} \to \bar{X}$ can be seen as a $\pi$-equivariant $\bar{X}$-morphism. We have a continuous homomorphism 
            \begin{equation*}
                \varpi: E \to \{s_{\#}: s \in \Gamma_k\}, \quad \alpha \mapsto q(\alpha)_{\#},
            \end{equation*}
        where $q: E \to \Gamma_k$ is the restriction of the natural homomorphism $\SAut(\bar{Y}/X) \to \Gamma_k$. For any $\alpha \in E$, one has $f \circ \alpha = q(\alpha)_{\#} \circ f$. In other words, the pair $(f,\varpi): \lambda \to p^\ast \lambda_{\triv}$ is a morphism of descent data $(\bar{Y},E) \to (\bar{X},\{s_{\#}: s \in \Gamma_k\})$ compatible with $\pi$ in the sense of Definition \ref{defn:DescentType} \ref{defn:DescentType2}. It follows that 
            \begin{equation} \label{eq:SumOfDescentType5}
                \pi_\ast \lambda = p^\ast \lambda_{\triv}.
            \end{equation}
        Thus, we have
		\begin{align*}
			\lambda + i_\ast \lambda & = \Delta^\ast \nabla_\ast (\lambda \times i_\ast \lambda), & \text{by Definition \ref{defn:SumOfDescentType}}, \\
			& = \Delta^\ast \nabla_\ast (\id_X \times_k i)_\ast (\lambda \times \lambda), & \text{by Proposition \ref{prop:ExternalProduct} \ref{prop:ExternalProduct2}},\\
			& = \Delta^\ast e_\ast (\pi \times_k \pi)_\ast (\lambda \times \lambda), & \text{by \eqref{eq:SumOfDescentType4}}, \\
			& = \Delta^\ast e_\ast (\pi_\ast \lambda \times \pi_\ast \lambda), & \text{by Proposition \ref{prop:ExternalProduct} \ref{prop:ExternalProduct2}}, \\
			& = e_\ast \Delta^\ast (p^\ast \lambda_{\triv} \circ p^\ast \lambda_{\triv}), & \text{by Lemma \ref{lemm:DescentTypePullbackVsPushforward} and \eqref{eq:SumOfDescentType5}}, \\
			& = e_\ast \Delta^\ast (p \times_k p)^\ast \lambda_{\triv}, & \text{by Proposition \ref{prop:ExternalProduct} \ref{prop:ExternalProduct1}}, \\
			& = e_\ast p^\ast \lambda_{\triv} , & \text{by \eqref{eq:SumOfDescentType4}}, \\
			& = p^\ast e_\ast \lambda_{\triv}, & \text{by Lemma \ref{lemm:DescentTypePullbackVsPushforward}}, \\
			& = \dtype(G_X), & \text{by \eqref{eq:SumOfDescentType3}}, 
		\end{align*}	
        showing that $i_\ast \lambda$ is the inverse of $\lambda$ for the operation ``$+$''.
    \end{enumerate}
    Functoriality and $X$ and $G$ follows immediately from Lemma \ref{lemm:DescentTypePullbackVsPushforward} and Definition \ref{defn:SumOfDescentType}.
\end{proof}

Recall that a descent type $\lambda = [\bar{Y},E] \in \DType(X,G)$ yields a class $[\bar{Y}] \in \H^0(k,\H^1(\bar{X},\bar{G}))$ as well as a class $[E] \in \H^2(k,G)$ ({\em cf.} Remarks \ref{remk:DescentType} \ref{remk:DescentType1}).

\begin{lemm} \label{lemm:SumOfDescentType}
	The maps
		\begin{equation*}
			\dtype: \H^1(X,G) \to \DType(X,G),
		\end{equation*}
		\begin{equation*}
			\DType(X,G) \to \H^0(k,\H^1(\bar{X},\bar{G})), \quad [\bar{Y},E] \mapsto [\bar{Y}],
		\end{equation*}
		and
		\begin{equation*}
			\DType(X,G) \to \H^2(k,G), \quad [\bar{Y},E] \mapsto [E]
		\end{equation*}
	are group homomorphisms.
\end{lemm}
\begin{proof}
    If $Y \to X$ and $Y' \to X$ are torsors under $G$, then $Y \times_k Y' \to X \times_k X$ is a torsor under $G \times_k G$, and $[Y] + [Y'] = [\Delta^\ast \nabla_\ast (Y \times_k Y')]$ in $\H^1(X,G)$. On the other hand, one has
		\begin{align*}
			\dtype(\Delta^\ast \nabla_\ast (Y \times_k Y')) & = \Delta^\ast \nabla_\ast \dtype (Y \times_k Y'), & \text{by Constructions \ref{cons:PullbackOfDescentType} and \ref{cons:PushforwardOfDescentType}}, \\
			& =  \Delta^\ast \nabla_\ast(\dtype(Y) \times \dtype(Y')), & \text{by Proposition \ref{prop:ExternalProduct} \ref{prop:ExternalProduct3}}, \\
			& = \dtype(Y) + \dtype(Y'), & \text{by Definition \ref{defn:SumOfDescentType}}.
		\end{align*}
	This shows that $\dtype: \H^1(X,G) \to \DType(X,G)$ is a group homomorphism.
	
	Let $\lambda = [\bar{Y},E]$ and $\lambda' = [\bar{Y}',E']$ be elements of $\DType(X,G)$. Then
		\begin{equation*}
			\lambda + \lambda' = [\Delta^\ast \nabla_\ast(\bar{Y} \times_{\bar{k}} \bar{Y}')], \Delta^\ast \nabla_\ast(E \boxtimes E')]
		\end{equation*} 
	by Constructions \ref{cons:PullbackOfDescentType} and \ref{cons:PushforwardOfDescentType}, Lemma \ref{lemm:PushforwardOfDescentType2}, and Definition \ref{defn:SumOfDescentType}. First, we have
		\begin{equation*}
			\Delta^\ast \nabla_\ast(\bar{Y} \times_{\bar{k}} \bar{Y}')] = [\bar{Y}] + [\bar{Y}']
		\end{equation*}
	 in $\H^1(\bar{X},\bar{G})$, which means the map
	 	\begin{equation*}
	 		\DType(X,G) \to \H^0(k,\H^1(\bar{X},\bar{G})), \quad [\bar{Y},E] \mapsto [\bar{Y}],
	 	\end{equation*}
	 is a group homomorphism. Second, by the description of the addition in $\H^2(k,G)$ in terms of group extensions (the ``Baer sum''), we have
	 	\begin{equation*}
	 	 [E] + [E'] = [\nabla_\ast(E \times_{\Gamma_k} E')] = [\nabla_\ast(E \boxtimes E')] = [\Delta^\ast \nabla_\ast(E \boxtimes E')].
	 	\end{equation*}
    This shows that the map
	 \begin{equation*}
	   \DType(X,G) \to \H^2(k,G), \quad [\bar{Y},E] \mapsto [E],
	 \end{equation*}
	 is also a group homomorphism.
\end{proof}

\subsection{The fundamental exact sequence} \label{subsec:Fundamental}

In this paragraph, we establish Theorem \ref{thm:Abelian}. Keep the notations $k$, $X$, and $G$ from the previous paragraph, and let $p: X \to \Spec(k)$ denote the structure morphism.

\begin{thm} [Theorem \ref{customthm:Abelian}] \label{thm:Abelian}
    We have a commutative diagram 
	\begin{equation} \label{eq:DescentAbelianExactSequence0}
		\xymatrix@C-1pc{
			& \H^1(k,G) \ar[r]^-{p^\ast} \ar[d] & \H^1(X,G) \ar[rr]^-{\dtype} \ar@{=}[d] && \DType(X,G) \ar[rr]^-{[\bar{Y},E] \mapsto [E]} \ar[d]^-{[\bar{Y},E] \mapsto [\bar{Y}]} && \H^2(k,G) \ar[r]^-{p^\ast} \ar[d] & \H^2(X,G) \ar@{=}[d] \\
			0 \ar[r] & \H^1(k,G(\bar{X})) \ar[r] & \H^1(X,G) \ar[rr] && \H^0(k,\H^1(\bar{X},\bar{G})) \ar[rr]^-{\trans} && \H^2(k,G(\bar{X})) \ar[r] & \H^2(X,G)
		}
	\end{equation}
    where the bottom row is the exact sequence issued from the Hochschild--Serre spectral sequence 
	\begin{equation*}
		\H^p(k,\H^q(\bar{X},\bar{G})) \Rightarrow \H^{p+q}(X,G)
	\end{equation*}
    (in particular, $\trans$ stands for the ``transgression'' map). The top row of \eqref{eq:DescentAbelianExactSequence0} is a complex, which is exact whenever $X$ is quasi-projective, reduced, and geometrically connected.
\end{thm}

We consider the four squares of \eqref{eq:DescentAbelianExactSequence0} from left to right. The first and the last squares obviously commute. The second square commutes by the very definition of the map $\dtype$. As for the third square, one requires a further investigation of the transgression map $\H^0(k,\H^1(\bar{X},\bar{G})) \to \H^2(k,G(\bar{X}))$. We recall that the action of $\Gamma_k$ on $G(\bar{X})$ is defined as follows. For any $\bar{k}$-morphism $\sigma: \bar{X} \to \bar{G}$ and any $s \in \Gamma_k$, the $\bar{k}$-morphism $\tensor[^s]{\sigma}{}: \bar{X} \to \bar{G}$ is given by the formula 
    \begin{equation} \label{eq:GaloisActionOnSections}
        \forall x \in \bar{X}, \quad (\tensor[^s]{\sigma}{})(x) := \tensor[^s]{\sigma}{}(\tensor[^{s^{-1}}]{x}{}).
    \end{equation}
The stabilizer of any $\sigma \in G(\bar{X})$ is open, since such a $\bar{k}$-morphism is defined over some finite extension of $k$. Again, let $f: \bar{Y} \to \bar{X}$ be a torsor under $\bar{G}$. Denote by $\SAut_G(\bar{Y}/X)$ the subgroup of $\SAut(\bar{Y}/X)$ consisting of semilinear $X$-automorphisms $\alpha$ such that $\alpha \circ \rho_g = \rho_{\tensor[^{q(\alpha)}]{g}{}} \circ \alpha$ for all $g \in G(\bar{k})$ (where $q: \SAut(\bar{Y}/X) \to \Gamma_k$ is the natural continuous and open map from \eqref{eq:DescentTypeExactSequenceSAut}). This is equivalent to saying that $\alpha$ induces a $\bar{G}$-equivariant $\bar{X}$-isomorphism $\phi_{q(\alpha)} \bar{Y} \to \bar{Y}$, where $\phi: \Gamma_k \to \SAut^{\gr}(\bar{G}/k)$ is the semilinear Galois action on $\bar{G}$ associated with its $k$-form $G$ ({\em cf.} Remarks \ref{remk:DescentType} \ref{remk:DescentType1}). We have $\Ker(q) \cap \SAut_G(\bar{Y}/X) = \Aut_{\bar{G}}(\bar{Y}/\bar{X})$, the subgroup of $\bar{G}$-equivariant $\bar{X}$-automorphism of $\bar{Y}$, which, by \cite[Chapitre III, \S 1.4.8 and \S 1.5.7]{Giraud1971Cohomology}, is $G(\bar{X})$. 
We obtain a sequence of topological groups
    \begin{equation} \label{eq:H2Transgression}
        1 \to G(\bar{X}) \xrightarrow{\sigma \mapsto \rho_\sigma^{-1}} \SAut_G(\bar{Y}/X) \xrightarrow{q} \Gamma_k \to 1,
    \end{equation}
which is exact except possibly at the last non-trivial term.

\begin{prop} \label{prop:H2Transgression}
    With the above notations, the following claims hold.
    \begin{enumerate}
    	\item \label{prop:H2Transgression1} The continuous homomorphism $q$ from \eqref{eq:H2Transgression} is open.
    	
        \item \label{prop:H2Transgression2} We have $[\bar{Y}] \in \H^0(k,\H^1(\bar{X},\bar{G}))$ if and only if \eqref{eq:H2Transgression} is exact, {\em i.e.}, if $q$ is surjective.
        
        \item \label{prop:H2Transgression3} When \eqref{eq:H2Transgression} is exact, the action of $\Gamma_k$ on $G(\bar{X})$ induced by conjugation in $\SAut_G(\bar{Y}/X)$ is precisely the one defined in \eqref{eq:GaloisActionOnSections}.

        \item \label{prop:H2Transgression4} Suppose that $[\bar{Y}] \in \H^0(k,\H^1(\bar{X},\bar{G}))$. Then the transgression map 
            \begin{equation*}
                \trans: \H^0(k,\H^1(\bar{X},\bar{G}))\to \H^2(k,G(\bar{X}))
            \end{equation*}
        (issued from the Hochschild--Serre spectral sequence $\H^p(k,\H^q(\bar{X},\bar{G})) \Rightarrow \H^{p+q}(X,G)$) takes $[\bar{Y}]$ to the class of the extension \eqref{eq:H2Transgression} ({\em cf.} \cite[Chapitre VIII, \S 8]{Giraud1971Cohomology}).
    \end{enumerate}
\end{prop}
\begin{proof}
    \begin{enumerate}
        \item It suffices to repeat the argument at the beginning of paragraph \ref{subsec:Type}, with the following addition: Suppose that $K/k$ is a finite Galois extension such that $\bar{Y}$ is the base change to $\bar{k}$ of a torsor $Y \to X_K$ under a $G_K$ (a ``$K$-form of $\bar{Y}$''). Then, for each $t \in \Gamma_K$, the action $t_{\#}: \bar{Y} \to \bar{Y}$ (induced by the $K$-form $Y$) is an element of $\SAut_{G_K}(\bar{Y}/X_K) \subseteq \SAut_G(\bar{Y}/X)$. Indeed, one has $\tensor[^t]{(y \cdot g)}{} = \tensor[^t]{y}{} \cdot \tensor[^t]{g}{}$ for all $y \in \bar{Y}$ and $g \in \bar{G}$.
    	  	
        \item That $[\bar{Y}] \in \H^0(k,\H^1(\bar{X},\bar{G}))$ means for every $s \in \Gamma_k$, there exists a $\bar{G}$-equivariant $\bar{X}$-isomorphism $\phi_s\bar{Y} \to \bar{Y}$ (where $\phi: \Gamma_k \to \SAut^{\gr}(\bar{G}/k)$ is the semilinear Galois action on $\bar{G}$ associated with its $k$-form $G$), that is, $s$ lifts to an element of $\SAut_G(\bar{Y}/X)$.

        \item Let $\alpha \in \SAut_G(\bar{Y}/X)$ and $s := q(\alpha) \in \Gamma_k$. By definition, we have
            \begin{equation*}
                \forall y \in \bar{Y}, \forall g \in \bar{G}, \quad \alpha(y \cdot g) = \alpha(y) \cdot \tensor[^s]{g}{}.
            \end{equation*}
        Let $\sigma \in G(\bar{X})$. For all $y \in \bar{Y}$, let $x := f(y) \in \bar{X}$. One has $f(\alpha^{-1}(y)) = \tensor[^{s^{-1}}]{x}{}$, hence
            \begin{equation*}
                \alpha(\rho_\sigma(\alpha^{-1}(y))) = \alpha(\alpha^{-1}(y) \cdot \sigma(\tensor[^{s^{-1}}]{x}{})) = \alpha(\alpha^{-1}(y)) \cdot \tensor[^s]{\sigma}{}(\tensor[^{s^{-1}}]{x}{}) = y \cdot (\tensor[^s]{\sigma}{})(x) = y \cdot (\tensor[^s]{\sigma}{})(f(y)).
            \end{equation*}
        It follows that $\alpha \circ \rho_\sigma \circ \alpha^{-1} = \rho_{\tensor[^s]{\sigma}{}}$, which is exactly what need to be proved.
        
        \item First, we note that the extension \eqref{eq:H2Transgression} has a continuous set-theoretic section. Indeed, let $K/k$ be a finite extension such that there exists a torsor $Y \to X_K$ under $G_K$ whose base change to $\bar{k}$ is isomorphic to $\bar{Y}$. The action of $\Gamma_K$ on $Y$ yields a continuous homomorphism $\varsigma: \Gamma_K \to \SAut_G(\bar{Y}/X)$ such that $q \circ \varsigma = \id_{\Gamma_K}$. For any $\gamma \in \Gal(K/k) = \Gamma_k/\Gamma_K$, choose an arbitrary lifting $s_\gamma \in \Gamma_k$, and lift it to an arbitrary element $\alpha_\gamma \in \SAut_G(\bar{Y}/X)$ (which is possible by virtue of \ref{prop:H2Transgression2}). Then, each element $s \in \Gamma_k$ can be uniquely written as $s = s_\gamma t$ for some $\gamma \in \Gal(K/k)$ and $t \in \Gamma_K$, and the map
        \begin{equation*}
        	\Gamma_K \to \SAut_G(\bar{Y}/X), \quad s_\gamma t \mapsto \alpha_\gamma \varsigma_t
        \end{equation*}
        is a continuous set-theoretic section of $q$.
        
         Recall the description of $\trans([\bar{Y}])$ in terms of $k$-{\em gerbes} (see {\em e.g.} \cite[pp. 10--11]{DLA2019Reduction}). Let $\Gcal_{\bar{Y}}$ be the ``$k$-gerbe of forms of $\bar{Y}$'', that is, for every finite extension $K/k$, the fibre category $\Gcal_{\bar{Y}}(K)$ is the groupoid of $K$-form of $\bar{Y} \to \bar{X}$. This is not necessarily an {\em algebraic} $k$-gerbe; it is bound by the sheaf $p_\ast p^\ast G$ (where $p: X \to \Spec(k)$ is the structure morphism), or equivalently, the $\Gamma_k$-module $G(\bar{X})$. Nevertheless, one can speak of its class $[\Gcal_{\bar{Y}}] \in \H^2(k,G(\bar{X}))$, which turns out to be exactly $\trans([\bar{Y}])$ (see \cite[Chapitre V, \S 3.1.6 and \S 3.2.1]{Giraud1971Cohomology}). To translate the description of $[\Gcal_{\bar{Y}}]$ into one in terms of cocycles, we refer to the general result in \cite[\S 2]{Breen1994Classification} (if $p_\ast p^\ast G$ is representable by a smooth algebraic group over $k$, {\em e.g.}, when $p_\ast p^\ast G = G$, this can be found in \cite[p. 13]{DLA2019Reduction}). The result is as follows. Take any continuous set-theoretic section $\alpha: \Gamma_k \to \SAut_G(\bar{Y}/X)$, and put $\sigma_{s,t}:=\alpha_{st} \circ \alpha_t^{-1} \circ \alpha_s^{-1} \in G(\bar{X})$ for all $s,t \in \Gamma_k$. Then, the continuous $2$-cocycle $\sigma \in \Z^2(k, G(\bar{X}))$ represents $[\Gcal_{\bar{Y}}]$. But, by definition, it also represents the class of the extension \eqref{eq:H2Transgression}.
    \end{enumerate}
\end{proof}

To give a descent type $\lambda = [\bar{Y},E] \in \DType(X,G)$ is to give a torsor $f: \bar{Y} \to \bar{X}$ under $\bar{G}$ such that $[\bar{Y}] \in \H^0(k,\H^1(\bar{X},\bar{G}))$ and a subgroup $E \subseteq \SAut_G(\bar{Y}/X)$ extension of $\Gamma_k$ of $G(\bar{k})$. In particular, by virtue of Proposition \ref{prop:H2Transgression} \ref{prop:H2Transgression4}, the natural map $\H^2(k,G) \to \H^2(k,G(\bar{X}))$ (induced by the inclusion $G(\bar{k}) \hookrightarrow G(\bar{X})$) takes $[E]$ to $\trans([\bar{Y}])$. This show that the third square in \eqref{eq:DescentAbelianExactSequence0} commutes. It remains to investigate its top row, in order to finish the proof of Theorem \ref{thm:Abelian}.

\begin{proof} [End of proof of Theorem \ref{thm:Abelian}]
    If $Z$ is any $k$-torsor under $G$, then Proposition \ref{prop:DescentType} \ref{prop:DescentType3} (applied to $X = \Spec(k)$) tells us that $\dtype(Z) = \dtype(G) = 0$ in $\DType(\Spec(k),G)$. By functoriality, we have $\dtype(p^\ast Z) = p^\ast \dtype(Z) = 0$ in $\DType(X,G)$, that is, the sequence
        \begin{equation*}
            \H^1(k,G) \xrightarrow{p^\ast} \H^1(X,G) \xrightarrow{\dtype} \DType(X,G)
        \end{equation*}
    is a complex. Next, if $Y \to X$ is a torsor under $G$, then $\dtype(Y) = [\bar{Y},E_Y]$ and $[E_Y] = 0$ in $\H^2(k,G)$ by definition of the map $\dtype$. Thus, the sequence
        \begin{equation*}
            \H^1(X,G) \xrightarrow{\dtype} \DType(X,G) \xrightarrow{[\bar{Y},E] \mapsto [E]} \H^2(k,G)
        \end{equation*}
    is a complex. Finally, the sequence 
        \begin{equation*}
            \DType(X,G) \xrightarrow{[\bar{Y},E] \mapsto [E]} \H^2(k,G) \xrightarrow{p^\ast} \H^2(X,G)
        \end{equation*}
    is a complex because the bottom row of \eqref{eq:DescentAbelianExactSequence0} is a complex.

    Now, we assume that $X$ is quasi-projective, reduced, and geometrically connected. We shall show that this implies the exactness of the top row of \eqref{eq:DescentAbelianExactSequence0}. Exactness at the term $\H^1(X,G)$ follows from Proposition \ref{prop:DescentType} \ref{prop:DescentType3}, while exactness at $\DType(X,G)$ is a consequence of Proposition \ref{prop:DescentType} \ref{prop:DescentType1}. Now, let $\eta \in \H^2(k,G)$ be a class represented by a topological extension $E$ of $\Gamma_k$ by $G(\bar{k})$, such that $p^\ast \eta = 0$ in $\H^2(X,G)$. Since the bottom row of \eqref{eq:DescentAbelianExactSequence0} is exact, the image of $\eta$ in $\H^2(k,G(\bar{X}))$ is $\trans([\bar{Y}])$ for some torsor $\bar{Y} \to \bar{X}$ under $\bar{G}$ such that $[\bar{Y}] \in \H^0(k,\H^1(\bar{X},\bar{G}))$. In view of Proposition \ref{prop:H2Transgression}, there is a continuous homomorphism $i: E \to \SAut_G(\bar{Y}/X)$ fitting in a commutative diagram
        \begin{equation*}
			\xymatrix{
				1 \ar[r] & \bar{G}(\bar{k}) \ar[r] \ar@{_{(}->}[d] & E \ar[r] \ar[d]^{i} & \Gamma_k \ar[r] \ar@{=}[d] & 1 \\
				1 \ar[r] & G(\bar{X}) \ar[r]^-{\sigma \mapsto \rho_{\sigma}^{-1}} & \SAut_G(\bar{Y}/X) \ar[r]^-{q} & \Gamma_k \ar[r]
			& 1}.
		\end{equation*} 
    In particular, $i$ is necessarily injective. Thus, we may regard $E$ as a subgroup of $\SAut_G(\bar{Y}/X)$, {\em i.e.}, the pair $(\bar{Y},E)$ defines a descent type $\lambda = [\bar{Y},E] \in \DType(X,G)$, whose image in $\H^2(k,G)$ is precisely $\eta$. This proves exactness of the top row of \eqref{eq:DescentAbelianExactSequence0} at the term $\H^2(k,G)$.

    Finally, it is immediate from our construction that \eqref{eq:DescentAbelianExactSequence0} is functorial in $X$ and $G$.
\end{proof}

We conclude this \S{} with the following arithmetic application, which shall be used in a subsequent work. Let $X$ be a smooth quasi-projective geometrically integral variety over a number field $k$. We recall that the Brauer--Grothendieck group of $X$ is $\Br(X) := \H^2(X,\Gbb_m)$. Let $\Br_1(X):=\Ker(\Br(X) \to \Br(\bar{X}))$ be its ``algebraic'' subgroup. Let $\Omega$ denote the set of places of $k$, and let
    \begin{equation*}
        \Be(X):=\Ker\left(\Br_1(X) \to \prod_{v \in \Omega}\frac{\Br_1(X_{k_v})}{\Img(\Br(k_v) \to \Br(X_{k_v}))}\right)
    \end{equation*}
be the subgroup of $\Br_1(X)$ consisting of ``everywhere locally constant'' elements. We recall that $X(k_\Omega) := \prod_{v \in \Omega}X(k_v)$. Let us define $X(k_\Omega)^{\Be(X)}$ to be the subset of $X(k_\Omega)$ consisting of families of local points which are orthogonal to $\Be(X)$ relative to the Brauer--Manin pairing (see {\em e.g.} \cite[(5.2)]{Skorobogatov2001Torsors}, noting that the elements of $\Be(X)$ are ``unramified'' thanks to \cite[Th\'eor\`eme 2.1.1]{Harari1994Fibration}).

\begin{prop} \label{prop:Connected}
    Keep the above notations and assumptions, and let $L = (\bar{G},\kappa)$ be a $k$-kernel, where $\bar{G}$ is a connected linear algebraic $\bar{k}$-group. If $X(k_\Omega)^{\Be(X)} \neq \varnothing$, then every descent type $\lambda \in \DType(X,L)$ comes from a torsor $Y \to X$ under a $k$-form $G$ of $L$.
\end{prop} 
\begin{proof}
    Let $\bar{G}^{\operatorname{u}}$ be the unipotent radical of $\bar{G}$, $\bar{G}^{\operatorname{red}}:=\bar{G}/\bar{G}^{\operatorname{u}}$ (which is reductive), $\bar{T}$ the abelianization of $\bar{G}^{\operatorname{red}}$ (it is the maximal toric quotient of $\bar{G}$), and $\pi: \bar{G} \to \bar{T}$ the canonical projection. Since $\Ker(\pi)$ is characteristic in $\bar{G}$, the outer action $\kappa$ on $\bar{G}$ induces an {\em action} of $\Gamma_k$ on $\bar{T}$, hence a $k$-form $T$ of $\bar{T}$. By Construction \ref{cons:PushforwardOfDescentType} and Lemma \ref{lemm:PushforwardOfDescentType1}, we have a map $\pi_\ast: \DType(X,L) \to \DType(X,T)$. If $\lambda = [\bar{Y},E] \in \DType(X,L)$, then $\pi_\ast \lambda = [\pi_\ast \bar{Y},\pi_\ast E]$, where $[\pi_\ast E] = \pi_\ast[E] \in \H^2(k,T)$. By Theorem \ref{thm:Abelian}, the natural map $\H^2(k,T) \to \H^2(X,T)$ takes $\pi_\ast[E]$ to $0$. Under the assumption $X(k_\Omega)^{\Be(X)} \neq \varnothing$, Harari and Skorobogatov have shown that this map is injective (combine Proposition 8.1 and Corollary 8.17 in \cite{HS2013Descent}). It follows that $\pi_\ast[E] = 0$. On the other hand, since $X(k_\Omega) \neq \varnothing$, the restriction of $[E]$ to $\H^2(k_v,L)$ is neutral for all $v \in \Omega$ by virtue of Proposition \ref{prop:DescentType} \ref{prop:DescentType2}. By \cite[Proposition 6.5]{Borovoi1993H2}, we conclude that the class $[E] \in \H^2(k,L)$ is neutral. Now, the existence of an $X$-torsor of descent type $\lambda$ follows from Proposition \ref{prop:DescentType} \ref{prop:DescentType1}.
\end{proof}

\begin{remk}
    Proposition \ref{prop:Connected} has been shown in the case where $L = \lien(G)$ for some finite commutative algebraic $k$-group $G$ in \cite[Lemma 3.7]{HW2022Supersolvable}. 
\end{remk}

\section{Comparing descent types and extended types} \label{sec:Comparison}

In this \S, we fix a perfect field $k$ and a smooth geometrically integral variety $p: X \to \Spec(k)$. We define the {\em units-Picard complex} $\UPic(\bar{X})$ of $X$ to be the cone of the canonical morphism
    \begin{equation*}
        \Gbb_{m,k}[1] \to (\tau_{\le 1} \Rbb p_\ast \Gbb_{m,X})[1]
    \end{equation*}
in the bounded-below derived category $\Dcal^+(k)$ of $\Gamma_k$-modules. By \cite[Lemma 2.3, Corollary 2.5]{BvH2009Picard}, the truncated complex $\tau_{\le 1} \Rbb p_\ast \Gbb_{m,X}$ is quasi-isomorphic to $[\bar{k}(X)^\times \xrightarrow{\div} \Div(\bar{X})]$, with $\bar{k}(X)^\times$ in degree $0$. Accordingly, $\UPic(\bar{X})$ is represented by the complex $[\bar{k}(X)^\times/\bar{k}^\times \xrightarrow{\div} \Div(\bar{X})]$ concentrated in degrees $-1$ and $0$, and its cohomology is given by
    \begin{equation*}
        \Hscr^{-1}(\UPic(\bar{X})) = \Unit(\bar{X}) := \bar{k}[X]^\times / \bar{k}^\times
    \end{equation*}
and 
    \begin{equation*}
        \Hscr^0(\UPic(\bar{X})) = \Pic(\bar{X}).
    \end{equation*}
We note that the abelian group $\Unit(\bar{X})$ is finitely generated and free. This can be easily seen if $X$ has a smooth compactification $X^c$, (that is, a smooth proper variety containing $X$ as a dense open subset\footnote{such a compactification always exists when $k$ has characteristic $0$, by resolution of singularities.}), since in this case we have an inclusion $\Unit(\bar{X}) \hookrightarrow \Div_{\infty} (\bar{X}^c)$, where $\Div_{\infty}$ means the group of Weil divisors supported outside $\bar{X}$. For a characteristic-free proof, see \cite[Proposition 4.1.2]{Brion2018Linearization}.

Let $G$ be a smooth\footnote{{\em i.e.}, the characteristic of $k$ does not divide the cardinality of the torsion subgroup of $\wh{G}$.} $k$-group of multiplicative type. When $Y \to X$ is a torsor under $G$, its {\em type} (a notion due to Colliot-Th\'el\`ene and Sansuc \cite{CTS1987Descente}) is the Galois equivariant homomorphism
    \begin{equation*}
        \type(Y): \wh{G} \to \Pic(\bar{X}), \quad \chi \mapsto \chi_\ast[\bar{Y}],
    \end{equation*}
(we recall that $\wh{G} = \Hom_{\bar{k}}(\bar{G},\Gbb_m)$, $[\bar{Y}] \in \H^1(\bar{X},\bar{G})$, and $\H^1(\bar{X},\Gbb_m) = \Pic(\bar{X})$). Harari and Skorobogatov generalized the classical type into the {\em extended type} $\xtype(Y)$, which is a morphism $\wh{G} \to \UPic(\bar{X})$ in the derived category $\Dcal^+(k)$. More precisely, by \cite[Proposition 8.1]{HS2013Descent}, we have the top row of the following commutative diagram

\begin{equation} \label{eq:ExtendTypeExactSequence}
    \xymatrix@C-1pc{
	& \H^1(k,G) \ar[r]^-{p^\ast} \ar[d] & \H^1(X,G) \ar[rr]^-{\xtype} \ar@{=}[d] && \Hom_k(\wh{G},\UPic(\bar{X})) \ar[rr]^-{\partial} \ar[d] && \H^2(k,G) \ar[r]^-{p^\ast} \ar[d] & \H^2(X,G) \ar@{=}[d] \\
	0 \ar[r] & \H^1(k,G(\bar{X})) \ar[r] & \H^1(X,G) \ar[rr] && \H^0(k,\H^1(\bar{X},\bar{G})) \ar[rr]^-{\trans} && \H^2(k,G(\bar{X})) \ar[r] & \H^2(X,G),
    }
\end{equation}
the bottom row being issued from the Hochschild--Serre spectral sequence
    \begin{equation*}
        \H^p(k,\H^q(\bar{X},\bar{G})) \Rightarrow \H^{p+q}(X,G).
    \end{equation*}
Over $\bar{k}$, this tells us that $\H^1(\bar{X},\bar{G}) \cong \Hom_k(\wh{G},\UPic(\bar{X}))$, that is, an $\bar{X}$-torsor under $\bar{G}$ is perfectly determined by its extended type. Over $k$, the extended type classifies torsors up to twisting by a Galois cocycle $\sigma \in \Z^1(k,G)$. These features are similar to the those of the descent type. The aim of this \S{} is to show that these two notions of types are in fact equivalent. We shall compare the top rows of \eqref{eq:DescentAbelianExactSequence0} and \eqref{eq:ExtendTypeExactSequence}.

If $\Unit(\bar{X}) = 0$ (that is, if $\bar{k}[X]^\times = \bar{k}^\times$), then $G(\bar{X}) = G(\bar{k})$ and hence
    \begin{equation*}
        \Hom_k(\wh{G},\UPic(\bar{X})) = \Hom_k(\wh{G},\Pic(\bar{X})).
    \end{equation*}
In this case, the extended type coincides with the classical type, which gives an isomorphism
    \begin{equation*}
        \type: \Hom_k(\wh{G},\Pic(\bar{X})) \xrightarrow{\cong} \H^0(k,\H^1(\bar{X},\bar{G})).
    \end{equation*}
In other words, one may simply think of the type $\type(Y)$ as the class $[\bar{Y}] \in \H^0(k,\H^1(\bar{X},\bar{G}))$, just like in the classical descent theory of Colliot-Th\'el\`ene--Sansuc. The other extreme case is when $\Pic(\bar{X}) = 0$ (for example, when $X$ is a $k$-torsor under a torus). In this case, the classical type gives absolutely no information. In fact, the extended type is now an element of $\Ext^1_k(\wh{G},\Unit(\bar{X}))$.

Here is the main result of this \S.

\begin{thm} [Theorem \ref{customthm:Comparison}] \label{thm:Comparison}
    There exists a homomorphism 
        \begin{equation*}
            \Phi: \DType(X,G) \to \Hom_k(\wh{G},\UPic(\bar{X}))
        \end{equation*}
    of abelian groups, which is functorial in $X$ and $G$. It fits in a diagram
	\begin{equation} \label{eq:DiagramComparison}
		\xymatrix{
			\H^1(k,G) \ar[r]^{p^\ast} \ar@{=}[d] & \H^1(X,G) \ar[r]^-{\dtype} \ar@{=}[d] & \DType(X,G) \ar[rr]^{[\bar{Y},E] \mapsto [E]} \ar[d]^{\Phi} && \H^2(k,G) \ar@{=}[d] \ar[r]^{p^\ast} & \H^2(X,G) \ar@{=}[d] \\
			\H^1(k,G) \ar[r]^{p^\ast} & \H^1(X,G) \ar[r]^-{\xtype} & \Hom_k(\wh{G},\UPic(\bar{X})) \ar[rr]^-{\partial} && \H^2(k,G) \ar[r]^{p^\ast} & \H^2(X,G),
		}
	\end{equation}
	which commutes up to a sign, where the bottom row is exact. If $X$ is quasi-projective, then $\Phi$ is an isomorphism.
\end{thm}

\subsection{Extended types revisited} \label{subsec:ExtendedType}

The first step in the proof of Theorem \ref{thm:Comparison} is to describe the map $\xtype$ explicitly. Let $f: Y \to X$ be a torsor under $G$. As in \cite[\S 2]{HS2013Descent}, we define $\bar{k}[Y]^\times_G \subseteq \bar{k}[Y]^\times$ to be the $\Gamma_k$-submodule consisting of invertible functions $a$ for which there exists a character $\chi \in \wh{G}$ such that $a(y \cdot g) = a(y) \chi(g)$ for all $y \in Y(\bar{k})$ and $g \in G(\bar{k})$, that is, $a \circ \rho_g = a \chi(g)$ (where $\rho_g: \bar{Y} \to \bar{Y}$ denotes the right action of $g$). This character is necessarily unique and shall be denoted by $\chi_a$. The homomorphism $\bar{k}[Y]^\times_G \xrightarrow{a \mapsto \chi_a} \wh{G}$ is $\Gamma_k$-equivariant. Indeed, for all $a \in \bar{k}[Y]^\times_G$, $s \in \Gamma_k$, and $g \in G(\bar{k})$, we have
\begin{equation*}
	\chi_{\tensor[^s]{a}{}}(g) = \frac{\tensor[^s]{a}{} \circ \rho_g}{\tensor[^s]{a}{}} = \frac{s_{\#} \circ a \circ s_{\#}^{-1} \circ \rho_g}{s_{\#} \circ a \circ s_{\#}^{-1}} =  \frac{s_{\#} \circ a \circ \rho_{\tensor[^{s^{-1}}]{g}{}} \circ s_{\#}^{-1}}{s_{\#} \circ a \circ s_{\#}^{-1}} = s_{\#} \circ \chi_a(\tensor[^{s^{-1}}]{g}{}) = \tensor[^s]{\chi}{_a}(\tensor[^{s^{-1}}]{g}{}) = (\tensor[^s]{\chi}{_a})(g),
\end{equation*}
where $s_{\#}$ denotes the Galois action of $s$ on $\bar{Y}$ and $\Gbb_{m,\bar{k}}$. It follows that $\chi_{\tensor[^s]{a}{}} = \tensor[^s]{\chi}{_a}$ as claimed.

We also define the subsheaf $(f_\ast \Gbb_{m,Y})_G$ of the \'etale sheaf $f_\ast \Gbb_{m,Y}$ on $X$ as follows. For any \'etale morphism $U \to X$, the section group $(f_\ast \Gbb_{m,Y})_G(U) \subseteq (f_\ast \Gbb_{m,Y})(U) = \Mor_U(Y_U,\Gbb_{m,U})$ consists of the $U$-morphisms $a: Y_U \to \Gbb_{m,U}$ for which there exists a morphism $\chi: G_U \to \Gbb_{m,U}$ of $U$-group schemes such that $a(y \cdot g) = a(y)\chi(g)$ for all $y \in Y_U(\bar{k})$ and $g \in G_U(\bar{k})$. We have an exact sequence 
\begin{equation} \label{eq:Comparison1}
	0 \to \Gbb_{m,X} \to (f_\ast \Gbb_{m,Y})_G \to  \wh{G}_X \to 0
\end{equation}
of \'etale sheaves on $X$ (see  \cite[Proposition 8.9(i)]{HS2013Descent}), generalizing the relative version of Rosenlicht's lemma \cite[Proposition 1.4.2]{CTS1987Descente}. Note that $\bar{k}[Y]^\times_G = \bar{k}[Y]^\times$ and $(f_\ast \Gbb_{m,Y})_G = f_\ast \Gbb_{m,Y}$ in the case where $G$ is a torus. 

Under the identification $\H^1(X,G) \cong \Ext^1_X(\wh{G},\Gbb_m)$ ({\em cf.} \cite[Proposition 1.4.1]{CTS1987Descente}), the class $-[Y]$ corresponds to the class of the extension \eqref{eq:Comparison1} (see \cite[Proposition 8.9(ii)]{HS2013Descent}). Now, applying $\Rbb p_\ast$ to \eqref{eq:Comparison1} yields a distinguished triangle 
\begin{equation} \label{eq:Comparison2}
	\Rbb p_\ast\Gbb_{m,X} \to \Rbb p_\ast (f_\ast \Gbb_{m,Y})_G \to \Rbb p_\ast\wh{G}_X \to \Rbb p_\ast\Gbb_{m,X}[1]
\end{equation}
in $\Dcal^+(k)$. Since $\Rbb\Hom_X(\wh{G},-) = \Rbb\Hom_k(\wh{G},-) \circ \Rbb p_\ast$ (see {\em e.g.} \cite[Corollary 10.8.3]{Weibel1994Homological}), we have a canonical isomorphism $\Ext^1_X(\wh{G},\Gbb_m) \cong \Ext^1_k(\wh{G},\Rbb p_\ast \Gbb_{m,X})$, which takes the class of \eqref{eq:Comparison1} to that of the composite
\begin{equation*}
	\wh{G} = p_\ast \wh{G}_X \to \Rbb p_\ast \wh{G}_X \to \Rbb p_\ast\Gbb_{m,X}[1],
\end{equation*}
the last arrow being the one from \eqref{eq:Comparison2}. Here, we have $\wh{G} = p_\ast \wh{G}_X$ because $\bar{X}$ is connected. As in the proof of \cite[Proposition 8.1]{HS2013Descent}, we have $\Ext^1_k(\wh{G},\Rbb p_\ast\Gbb_{m,X}) \cong \Ext^1_k(\wh{G},\tau_{\le 1}\Rbb p_\ast\Gbb_{m,X})$. The above morphism hence factors through a morphism $\wh{G} \to (\tau_{\le 1} \Rbb p_\ast \Gbb_{m,X})[1]$, which, after composing with the canonical map $(\tau_{\le 1} \Rbb p_\ast \Gbb_{m,X})[1] \to \UPic(\bar{X})$, yields a morphism in $\Dcal^+(k)$ representing $-\xtype([Y])$. Our aim now is to describe the cone of $\wh{G} \to (\tau_{\le 1} \Rbb p_\ast \Gbb_{m,X})[1]$ as an explicit $2$-term complex.

Recall that $\tau_{\le 1} \Rbb p_\ast \Gbb_{m,X}$ is quasi-isomorphic to the complex $[\bar{k}(X)^\times \xrightarrow{\div} \Div(\bar{X})]$ concentrated in degrees $0$ and $1$. We start by investigating $\bar{k}(X)^\times$. Let $K:=k(X)$ and $L:=\bar{k} \otimes_k L = \bar{k}(X)$. The exact sequence \eqref{eq:Comparison1} for the torsor $Y_L \to \Spec(L)$ reads
\begin{equation*}
	1 \to \bar{L}^\times \to \bar{L}[Y_L]^\times_{G_L} \to \wh{G} \to 0.
\end{equation*}
Taking Galois cohomology (noting that $\H^1(L,\Gbb_m) = 0$ by Hilbert's Theorem 90 and $\H^0(L,\wh{G}) = \wh{G}$ since $G$ is split over $\bar{k} \subseteq L$), one obtain an exact sequence
\begin{equation} \label{eq:Comparison3}
	1 \to \bar{k}(X)^\times \to L[Y_L]^\times_{G_L} \to \wh{G} \to 0.
\end{equation}
On the other hand, taking cohomology of \eqref{eq:Comparison2} yields an exact sequence
\begin{equation} \label{eq:Comparison4}
	1 \to \bar{k}[X]^\times \xrightarrow{f^\ast} \bar{k}[Y]^\times_{G} \xrightarrow{a \mapsto \chi_a} \wh{G} \to \Pic(\bar{X}).
\end{equation}
For each non-empty open subset $U \subseteq X$, it is clear that $\bar{k}[Y_U]^\times_{G} \subseteq L[Y_L]^\times_{G_L}$. We define the $\Gamma_k$-module
\begin{equation} \label{eq:Comparison5}
	\bar{k}(Y)^\times_{G}:=\varinjlim_{U \subseteq X} \bar{k}[Y_U]^\times_{G},
\end{equation}
which {\em a priori} is a subgroup of $L[Y_L]^\times_{G_L}$. Since $\wh{G}$ is finitely generated as an abelian group, the composite $\wh{G} \to \Pic(\bar{X}) \to \Pic(\bar{U})$ is $0$ if $U$ is a sufficiently small. As $\bar{k}(X)^\times = \varinjlim\limits_{U \subseteq X} \bar{k}[U]^\times$, we have an exact sequence
\begin{equation} \label{eq:Comparison6}
	1 \to \bar{k}(X)^\times \xrightarrow{f^\ast} \bar{k}(Y)^\times_{G} \xrightarrow{a \mapsto \chi_a} \wh{G} \to 0.
\end{equation}
Comparing \eqref{eq:Comparison3} to \eqref{eq:Comparison6}, we see that, in fact, $\bar{k}(Y)^\times_{G} = L[Y_L]^\times_{G_L}$.

Let us now inspect the term $\Div(\bar{X})$. Although the $k$-variety $Y$ is not necessarily integral, it is smooth and pure-dimensional and we can talk about Weil divisors on it. Call an integral divisor $Z \subseteq Y$ {\em horizontal} if it is the Zariski closure of a divisor on the generic fibre $Y_K \to \Spec(K)$ of $f$, and {\em vertical} otherwise ({\em i.e.}, if $f(Z)$ does not contain the generic point of $X$). Let $\Div^{\hor}(Y)$ (resp. $\Div^{\ver}(Y)$) denote the (free) abelian groups of horizontal (resp. vertical) divisors on $Y$. Then $\Div(Y) = \Div^{\hor}(Y) \oplus \Div^{\ver}(Y)$. We say that a divisor $D = \sum_{\eta \in Y^{(1)}} n_{\eta} [\eta] \in \Div(Y)$ is {\em $G$-invariant} if for every $\xi \in X^{(1)}$ and $\eta,\eta' \in f^{-1}(\xi)$, one has $n_{\eta} = n_{\eta'}$.  If $G$ is a torus, then every divisor is $G$-invariant because $f$ has geometrically integral fibres. We write $\Div^{\ver,G}(Y) \subseteq \Div^{\ver}(Y)$ for the subgroup of vertical, $G$-invariant divisors. Then the flat pullback $f^\ast: \Div(X) \to \Div(Y)$ is injective with image $\Div^{\ver,G}(Y)$.

Over $\bar{k}$, a divisor $D \in \Div(\bar{Y})$ is $\bar{G}$-invariant if and only if $\rho_g^{\ast} D = D$ for all $g \in G(\bar{k})$ (as always, we use $\rho_g$ to denote the automorphism $y \mapsto y \cdot g$ of $\bar{Y}$). Since $\bar{k}(Y)^\times_{G} = L[Y_L]^\times_{G_L}$, any function $a \in \bar{k}(Y)^\times_{G}$ is regular on the generic fibre of $\bar{Y} \to \bar{X}$, {\em i.e.}, $\div(a) \in \Div^{\ver}(\bar{Y})$. Furthermore, it follows from the definition of the groups $\bar{k}[Y_U]^\times_G$ that $\div(a) \in \Div^{\ver,\bar{G}}(\bar{Y})$. In view of \eqref{eq:Comparison6}, we have an exact sequence 
\begin{equation} \label{eq:Comparison7}
	0 \to [\bar{k}(X)^\times \xrightarrow{\div} \Div(\bar{X})] \xrightarrow{f^\ast} [\bar{k}(Y)^\times_G \xrightarrow{\div} \Div^{\ver,\bar{G}}(\bar{Y})] \to \wh{G} \to 0
\end{equation}
of complexes of $\Gamma_k$-modules, where the two complexes are concentrated in degrees $0$ and $1$.

\begin{prop} \label{prop:Comparison1}
	Let $f: Y \to X$ be a torsor under $G$. Under the isomorphism
	\begin{equation*}
		\H^1(X,G) \cong \Ext^1_X(\wh{G},\Gbb_m) \cong \Ext^1_k(\wh{G},\Rbb p_\ast \Gbb_m) \cong \Ext^1_k(\wh{G},\tau_{\le 1} \Rbb p_\ast \Gbb_m) = \Ext^1_k(\wh{G},[\bar{k}(X)^\times \xrightarrow{\div} \Div(\bar{X})]),
	\end{equation*}
	the element $[Y] \in \H^1(X,G)$ corresponds to the inverse class of the extension \eqref{eq:Comparison7}. Consequently, the extended type $\xtype([Y]) \in \Hom_k(\wh{G},\UPic(\bar{X})) = \Ext^1_k({\wh{G}},[\bar{k}(X)^\times/\bar{k}^\times \xrightarrow{\div} \Div(\bar{X})])$ is the inverse class of the extension
	\begin{equation} \label{eq:Comparison8}
		0 \to [\bar{k}(X)^\times/\bar{k}^\times \xrightarrow{\div} \Div(\bar{X})] \xrightarrow{f^\ast} [\bar{k}(Y)^\times_G/\bar{k}^\times \xrightarrow{\div} \Div^{\ver,\bar{G}}(\bar{Y})] \to \wh{G} \to 0.
	\end{equation}
\end{prop}
\begin{proof}
    The exact sequence \eqref{eq:Comparison1} for the generic fibre $Y_K \to \Spec(K)$ is 
        \begin{equation*}
            0 \to \Gbb_{m,K} \to (f_\ast \Gbb_{m,Y_K})_{G_K} \to \wh{G}_K \to 0.
        \end{equation*}
    Let $j: \Spec(K) \to X$ denote the inclusion of the generic point. Since $\Rbb^1 j_\ast \Gbb_{m,K} = 0$ by Hilbert's Theorem 90 and since $j_\ast \wh{G}_K = j_\ast (p \circ j)^\ast \wh{G} = j_\ast j^\ast p^\ast \wh{G} = j_\ast j^\ast (\wh{G}_X) = \wh{G}_X$, we have an exact sequence
	\begin{equation} \label{eq:Comparison9}
		0 \to j_\ast \Gbb_{m,K} \to j_\ast (f_\ast \Gbb_{m,Y_K})_{G_K} \to \wh{G}_X \to 0
	\end{equation}
	of \'etale sheaves on $X$. Let $\iDiv_X$ denote the sheaf of Weil divisors on $X$, and $\iDiv^{\ver,G}_Y$ the sheaf of vertical, $G$-invariant divisors on $Y$. Then, \eqref{eq:Comparison1} and \eqref{eq:Comparison9} fit in a commutative diagram
	\begin{equation*}
		\xymatrix{
			& 0 \ar[d] & 0 \ar[d]\\
			0 \ar[r] & \Gbb_{m,X} \ar[r] \ar[d] & (f_\ast \Gbb_{m,Y})_G \ar[r] \ar[d] & \wh{G}_X \ar[r] \ar@{=}[d] & 0\\
			0 \ar[r] & j_\ast \Gbb_{m,K} \ar[r] \ar[d] & j_\ast (f_\ast \Gbb_{m,Y})_G \ar[r] \ar[d] & \wh{G}_X \ar[r] & 0 \\
			0 \ar[r] & \iDiv_X \ar[r]^-{\cong} \ar[d] & \iDiv^{\ver,G}_Y \ar[d] & \\
			& 0 & 0.
		}
	\end{equation*}
    Since the left column is exact, it is also the case for the middle one. In particular, we have quasi-isomorphisms 
        \begin{equation*}
            \Rbb p_\ast \Gbb_{m,X} \simeq \Rbb p_\ast[j_\ast \Gbb_{m,K} \to \iDiv_X]
        \end{equation*}
    and
        \begin{equation*}
            \Rbb p_\ast(f_\ast \Gbb_{m,Y})_G \simeq \Rbb p_\ast[j_\ast (f_\ast \Gbb_{m,Y})_G \to \iDiv^{\ver,G}_Y].
        \end{equation*}
    Now, any complex of sheaves $\Fcal$ on $X_{\et}$ yields a canonical morphism $p_\ast \Fcal \to \Rbb p_\ast \Fcal$ in $\Dcal^+(k)$, which is functorial in $\Fcal$. Since $p_\ast j_\ast \Gbb_{m,K} = \bar{k}(X)^\times$, $p_\ast \iDiv_X = \Div(\bar{X})$, $p_\ast j_\ast (f_\ast \Gbb_{m,Y})_G = \bar{k}(Y)^\times_G$, and $p_\ast \iDiv^{\ver,G}_Y = \Div^{\ver,\bar{G}}(\bar{Y})$, we obtain a morphism of distinguished triangles from \eqref{eq:Comparison7} to \eqref{eq:Comparison2}, {\em i.e.}, a commutative diagram
	\begin{equation*}
		\xymatrix{
			[\bar{k}(X)^\times \xrightarrow{\div} \Div(\bar{X})] \ar[r]^-{f^\ast} \ar[d] & [\bar{k}(Y)^\times_G \xrightarrow{\div} \Div^{\ver,\bar{G}}(\bar{Y})] \ar[r] \ar[d] & \wh{G} \ar[r] \ar[d] & [\bar{k}(X)^\times \xrightarrow{\div} \Div(\bar{X})][1] \ar[d] \\
			\Rbb p_\ast \Gbb_{m,X} \ar[r] & \Rbb p_\ast (f_\ast \Gbb_{m,X})_G \ar[r] & \Rbb p_\ast \wh{G}_X \ar[r] & \Rbb p_\ast \Gbb_{m,X}[1].
		}
	\end{equation*}
	Since the complex $[\bar{k}(X)^\times \xrightarrow{\div} \Div(\bar{X})]$ represents the object $\tau_{\le 1} \Rbb p_\ast \Gbb_{m,X}$, it follows from the discussion at the beginning of this paragraph that $-\type([Y])$ is precisely the class of the morphism $\wh{G} \to [\bar{k}(X)^\times \xrightarrow{\div} \Div(\bar{X})][1]$ from the top row of the above diagram, {\em i.e.}, the class of \eqref{eq:Comparison7}.
\end{proof}

\subsection{The comparison map} \label{subsec:ComparisonMap}

Proposition \ref{prop:Comparison1} gives us a hint on how to construct an element of $\Hom_k(\wh{G},\UPic(\bar{X}))$ from a descent type $\lambda = [\bar{Y},E] \in \DType(X,G)$ (instead of a torsor $Y \to X$ under $G$). We have a complex $[\bar{k}(\bar{Y})^\times_{\bar{G}}/\bar{k}^\times \xrightarrow{\div} \Div^{\ver,\bar{G}}(\bar{Y})]$ of {\em abelian groups}, and we shall use the subgroup $E \subseteq \SAut_G(\bar{Y}/X)$ to define a suitable Galois action on it.

\begin{lemm} \label{lemm:ActionOfSAut}
    Let $f: \bar{Y} \to \bar{X}$ be a torsor under $\bar{G}$ such that $[\bar{Y}] \in \H^0(k,\H^1(\bar{X},\bar{G}))$. Recall that the sequence
	\begin{equation*}
		1 \to G(\bar{X}) \to \SAut_G(\bar{Y}/X) \xrightarrow{q} \Gamma_k \to 1
	\end{equation*}
    is exact ({\em cf.} \eqref{eq:H2Transgression}). We define an action of $\SAut_G(\bar{Y}/X)$ on $\bar{k}[\bar{Y}]^\times_{\bar{G}}$ by the formula
	\begin{equation} \label{eq:ActionOfSAut}
		\forall \alpha \in \SAut_G(\bar{Y}/X), \forall a \in \bar{k}[\bar{Y}]^\times_{\bar{G}}, \quad \alpha \cdot a := q(\alpha)_{\#} \circ a \circ \alpha^{-1},
	\end{equation}
    Then the following are true.
	\begin{enumerate}
		\item \label{lemm:ActionOfSAut1}  
		The action \eqref{eq:ActionOfSAut} is well defined, and the map $\bar{k}[\bar{Y}]^\times_{\bar{G}} \xrightarrow{a \mapsto \chi_a} \wh{G}$ is $q$-equivariant.
		
	   \item \label{lemm:ActionOfSAut2} The group $G(\bar{X}) = \Ker(q)$ acts trivially on the image of the inclusion $f^\ast: \bar{k}[X]^\times \hookrightarrow \bar{k}[\bar{Y}]^\times_{\bar{G}}$, and the induced action of $\Gamma_k = \Coker(q)$ on this image is the natural Galois action.

		\item \label{lemm:ActionOfSAut3}		Let $\phi: X' \to X$ be a morphism between smooth geometrically integral $k$-varieties, and let $\bar{Y}' := \phi^\ast \bar{Y}  = \bar{Y} \times_{\bar{X}} \bar{X}'$. Recall from Construction \ref{cons:PullbackOfDescentType} that we have a homomorphism $\phi^\ast: \SAut_G(\bar{Y}/X) \to \SAut_G(\bar{Y}'/X')$. The pullback homomorphism $\bar{k}[\bar{Y}]^\times_{\bar{G}} \to \bar{k}[\bar{Y}']^\times_{\bar{G}}$ is $\phi^\ast$-equivariant.
		
		\item \label{lemm:ActionOfSAut4} Let $\varphi: G \to H$ be a morphism of groups of $k$-groups of multiplicative type, $\bar{Z} \to \bar{X}$ a torsor under $H$ such that $[\bar{Z}] \in \H^0(k,\H^1(\bar{X},\bar{H}))$, and $\iota: \bar{Y} \to \bar{Z}$ a $\varphi$-equivariant $\bar{X}$-morphism. Let $\alpha \in \SAut_G(\bar{Y}/X)$, and let $\beta \in \SAut_H(\bar{Z}/X)$ be a $q(\alpha)$-semilinear automorphism fitting in a commutative diagram
		\begin{equation*}
			\xymatrix{
				\bar{Y} \ar[r]^{\alpha} \ar[d]^{\iota} & \bar{Y} \ar[d]^{\iota} \\
				\bar{Z} \ar[r]^{\beta} & \bar{Z}.
			}
		\end{equation*}
		Then $\alpha \cdot \iota^\ast a = \iota^\ast(\beta \cdot a)$ for all $a \in \bar{k}[\bar{Z}]_{\bar{H}}^\times$.
	\end{enumerate}
\end{lemm}
\begin{proof}
    \begin{enumerate}
	\item Let $a \in \bar{k}[\bar{Y}]_{\bar{G}}^\ast$, $\alpha \in \SAut_G(\bar{Y}/X)$, and $s:=q(\alpha)$. For any $g \in G(\bar{k})$, one has
		\begin{equation*}
			\frac{(\alpha \cdot a) \circ \rho_g}{\alpha \cdot a} = \frac{s_{\#} \circ a \circ \alpha^{-1} \circ \rho_g}{s_{\#} \circ a \circ \alpha^{-1}} = \frac{s_{\#} \circ a \circ \rho_{\tensor[^{s^{-1}}]{g}{}} \circ \alpha^{-1}}{s_{\#} \circ a \circ \alpha^{-1}} = s_{\#} \circ \chi_a(\tensor[^{s^{-1}}]{g}{}) = \tensor[^s]{\chi}{_a}(\tensor[^{s^{-1}}]{g}{}) = (\tensor[^s]{\chi}{_a})(g).
		\end{equation*}
		This means $\alpha \cdot a \in \bar{k}[\bar{Y}]_{\bar{G}}^\ast$ and $\chi_{\alpha \cdot a} = \tensor[^s]{\chi}{_a}$, which is what needed to be proved.
		
	\item Recall that the inclusion $G(\bar{X}) = \Mor_{\bar{k}}(\bar{X},\bar{G}) \hookrightarrow \SAut_G(\bar{Y}/X))$ is given by $\sigma \mapsto \rho_\sigma^{-1}$, where $\rho_\sigma: \bar{Y} \to \bar{Y}$ is the $\bar{X}$-automorphism $y \mapsto y \cdot \sigma(f(y))$. For $b \in \bar{k}[X]^\times$ and $\alpha \in \SAut_G(\bar{Y}/X)$, we have
		\begin{equation*}
			\alpha \cdot (f^\ast b) = q(\alpha)_{\#} \circ b \circ f \circ \alpha^{-1} = q(\alpha)_{\#} \circ b \circ q(\alpha)^{-1}_{\#} \circ f = \tensor[^{q(\alpha)}]{b}{} \circ f = f^\ast(\tensor[^{q(\alpha)}]{b}{}),
		\end{equation*}
		which is what claimed.
		
	\item Let $\pi: \bar{Y}' \to \bar{Y}$ be the canonical projection. For each $\alpha \in \SAut_G(\bar{Y}/X)$, we have $\alpha \circ \pi = \pi \circ \phi^\ast \alpha$ by the construction of $\phi^\ast \alpha$. This implies
		\begin{equation*}
			\phi^\ast \alpha \cdot \pi^\ast a =  q(\alpha)_{\#} \circ a \circ \pi \circ (\phi^\ast \alpha)^{-1} = q(\alpha)_{\#} \circ a \circ \alpha^{-1} \circ \pi = (\alpha \cdot a) \circ \pi = \pi^\ast (\alpha \cdot a)
		\end{equation*}
		for all $a \in \bar{k}[\bar{Y}]_{\bar{G}}^\ast$. The claim follows.
		
	\item We have $\alpha \cdot \iota^\ast a = q(\alpha)_{\#} \circ a \circ \iota \circ \alpha^{-1} = q(\alpha)_{\#} \circ a \circ \beta^{-1} \circ \iota = (\beta \cdot a) \circ \iota = \iota^{\ast}(\beta \cdot a)$.
    \end{enumerate}
\end{proof}

\begin{cons} \label{cons:Comparison}
    Let $f: \bar{Y} \to \bar{X}$ be a torsor under $\bar{G}$ such that $[\bar{Y}] \in \H^0(k,\H^1(\bar{X},\bar{G}))$, and let $E \subseteq \SAut_G(\bar{Y}/X)$ be a subgroup extension of $\Gamma_k$ by $G(\bar{k})$. Then we have an exact sequence 
	\begin{equation*}
		1 \to \bar{k}(X)^\times \xrightarrow{f^\ast} \bar{k}(\bar{Y})^\times_{\bar{G}} \xrightarrow{a \mapsto \chi_a} \wh{G} \to 0
	\end{equation*}
    of abelian groups ({\em cf.} \eqref{eq:Comparison6}). Equipped with the action \eqref{eq:ActionOfSAut}, the above sequence is an exact sequence of $E$-module (Lemma \ref{lemm:ActionOfSAut} \ref{lemm:ActionOfSAut1},\ref{lemm:ActionOfSAut2}, and \ref{lemm:ActionOfSAut3}). For $a \in \bar{k}(\bar{Y})^\times_{\bar{G}}$ and $g \in G(\bar{k})$, we have $\rho^{-1}_g \cdot a = a \circ \rho_g = a \chi_a(g)$, hence $[\rho_g^{-1} \cdot a] = [a]$ in $\bar{k}(\bar{Y})^\times_{\bar{G}}/\bar{k}^\times$. In otherwords, $G(\bar{k})$ acts trivially on the quotient $\bar{k}(\bar{Y})^\times_{\bar{G}}/\bar{k}^\times$. It follows that we have an exact sequence
	\begin{equation*}
		1 \to \bar{k}(X)^\times/\bar{k}^\times \xrightarrow{f^\ast} \bar{k}(\bar{Y})^\times_{\bar{G}}/\bar{k}^\times \xrightarrow{a \mapsto \chi_a} \wh{G} \to 0
	\end{equation*}
	of $\Gamma_k$-modules. By Lemma \ref{lemm:ActionOfSAut} \ref{lemm:ActionOfSAut4}, the isomorphism class of the $\Gamma_k$-module $\bar{k}(\bar{Y})^\times_{\bar{G}}/\bar{k}^\times$ only depends on the descent type $[\bar{Y},E] \in \DType(X,G)$, and not the pair $(\bar{Y},E)$ representing it.
	
	\begin{lemm} \label{lemm:ActionOfSAutOpenStabilizer}
		With the above notations, the following are true.
		\begin{enumerate}
			\item \label{lemm:ActionOfSAutOpenStabilizer1} If $[\bar{Y},E] = \dtype(Y)$ for some torsor $Y \to X$ under $G$, then the above action of $\Gamma_k$ on $\bar{k}(Y)^\times_{G}/\bar{k}^\times$ is the natural Galois action.
			
			\item \label{lemm:ActionOfSAutOpenStabilizer2} In general, $\bar{k}(\bar{Y})^\times_{\bar{G}}/\bar{k}^\times$ is a discrete $\Gamma_k$-module.
		\end{enumerate}
	\end{lemm}
	\begin{proof}
		\begin{enumerate}
			\item Let $s \in \Gamma_k$ and lift it to the semilinear automorphism $s_{\#}: \bar{Y} \to \bar{Y}$. Then $s_{\#} \cdot a = s_{\#} \circ a \circ s_{\#}^{-1} = \tensor[^s]{a}{}$ for all $a \in \bar{k}(Y)^\times_{G}$. The claim follows.
			
			\item Take any non-empty affine open subset $U \subseteq X$. Then $\bar{k}(\bar{Y})^\times_{\bar{G}} = \bar{k}(\bar{Y}_{\bar{U}})^\times_{\bar{G}}$, and thus we may assume that $X$ is affine. Let $K/k$ be a finite extension trivializing the class $[E] \in \H^2(k,G)$. By Proposition \ref{prop:DescentType} \ref{prop:DescentType1}, we have the restriction of $[\bar{Y},E]$ to $K$ is $\dtype(Y)$ for some torsor $Y \to X_K$ under $G_K$. It follows from \ref{lemm:ActionOfSAutOpenStabilizer1} that the stabilizer of any function $a \in \bar{k}(\bar{Y})^\times_{\bar{G}}$ is an open subgroup of $\Gamma_K$, hence also an open subgroup of $\Gamma_k$.
		\end{enumerate}
	\end{proof}
    Now, we define a Galois action on $\Div^{\ver,\bar{G}}(\bar{Y})$ using the isomorphism $f^\ast: \Div(\bar{X}) \xrightarrow{\cong} \Div^{\ver,\bar{G}}(\bar{Y})$. Then the map $\div: \bar{k}(\bar{Y})^\times_{\bar{G}}/\bar{k}^\times \to \Div^{\ver,\bar{G}}(\bar{Y})$ is $\Gamma_k$-equivariant. Indeed, let $s \in \Gamma_k$ and lift it to an $s$-semilinear automorphism $\alpha \in E$. Let $a \in \bar{k}(\bar{Y})^\times_{\bar{G}}$ and write $\div(a) = f^\ast D$ for some (unique) $D \in \Div(\bar{X})$. Since $f \circ \alpha^{-1} = s_{\#}^{-1} \circ f$, we have
	\begin{equation*}
		\div(\alpha \cdot a) = (\alpha^{-1})^\ast \div(a) = (\alpha^{-1})^\ast f^\ast D = f^\ast (s_{\#}^{-1})^\ast D = f^\ast (\tensor[^s]{D}{}).
	\end{equation*}
	To sum up, we have an exact sequence 
	\begin{equation} \label{eq:Comparison10}
		0 \to [\bar{k}(X)^\times/\bar{k}^\times \xrightarrow{\div} \Div(\bar{X})] \xrightarrow{f^\ast} [\bar{k}(\bar{Y})^\times_{\bar{G}}/\bar{k}^\times \xrightarrow{\div} \Div^{\ver,\bar{G}}(\bar{Y})] \to \wh{G} \to 0.
	\end{equation}
	of complexes of $\Gamma_k$-modules, similar to \eqref{eq:Comparison8}. Let $\lambda:=[\bar{Y},E] \in \DType(X,G)$. We define 
	\begin{equation*}
		\Phi(\lambda) \in \Ext^1_k(\wh{G}, [\bar{k}(X)^\times/\bar{k}^\times \xrightarrow{\div} \Div(\bar{X})]) = \Hom_k(\wh{G},\UPic(\bar{X}))
	\end{equation*}
	to be the inverse class of the above extension. We obtain a map
        \begin{equation*}
            \Phi: \DType(X,G) \to \Hom_k(\wh{G},\UPic(\bar{X})).
        \end{equation*}
    By virtue of Lemma \ref{lemm:ActionOfSAut} \ref{lemm:ActionOfSAut3} and \ref{lemm:ActionOfSAut4}, the map $-\Phi$ functorial in $X$ and $G$. 
\end{cons}

\begin{lemm} \label{lemm:ComparisonMapIsHomomorphism}
    The map $\Phi: \DType(X,G) \to \Hom_k(\wh{G},\UPic(\bar{X}))$ from Construction \ref{cons:Comparison} is a homomorphism.
\end{lemm}
\begin{proof}
    Let $\lambda_1 = [\bar{Y}_1,E_1]$ and  $\lambda_2 = [\bar{Y}_2,E_2]$ be descent types on $X$ bound by $\lien(G)$. We consider their external product $\lambda_1 \times \lambda_2 = [\bar{Y}, E_1 \boxtimes E_2] \in \DType(X \times_k X, G \times_k G)$, where $\bar{Y}:=\bar{Y}_1 \times_{\bar{k}} \bar{Y}_2$ ({\em cf.} Definition \ref{defn:ExternalProduct}). Recall that the sum $\lambda_1 + \lambda_2 \in \DType(X,G)$ is defined to be $\nabla_\ast \Delta^\ast(\lambda_1 \times \lambda_2)$, where $\nabla: G \times_k G \to G$ and $\Delta: X \to X \times_k X$ are the multiplication and diagonal morphisms, respectively (see Definition \ref{defn:SumOfDescentType}).
	
    For $i=1,2$, denote by $f_i: \bar{Y}_i \to \bar{X}$ the structure morphism, and by $\pi_i: \bar{Y} \to \bar{Y}_i$ and $p_i: \bar{X} \times_{\bar{k}} \bar{X} \to \bar{X}$ the projections. Let $\varsigma = p_1^\ast + p_2^\ast: \UPic(\bar{X}) \oplus \UPic(\bar{X}) \to \UPic(\bar{X}  \times_{\bar{k}} \bar{X})$ and $\delta = \Delta^\ast: \UPic(\bar{X} \times_{\bar{k}} \bar{X}) \to \UPic(\bar{X})$, then $\delta \circ \varsigma$ is the sum map on $\UPic(\bar{X})$. On the other hand, the Cartier dual $\wh{\nabla}: \wh{G} \to \wh{G} \oplus \wh{G}$ of $\nabla$ is the diagonal map. It follows from the definition of the addition on $\Hom_k(\wh{G},\UPic(\bar{X}))$ (the ``Baer sum'') that 
	\begin{equation*}
		-\Phi(\lambda_1) - \Phi(\lambda_2) = \wh{\nabla}^\ast \delta_\ast \varsigma_\ast((-\Phi(\lambda_1)) \oplus (-\Phi(\lambda_2))).
	\end{equation*}
    Since the map $\Phi$ is functorial in $X$ and $G$, we have
	\begin{equation*}
		-\Phi(\lambda_1+\lambda_2) = -\Phi(\nabla_\ast \Delta^\ast(\lambda_1 \times \lambda_2)) = \wh{\nabla}^\ast(-\Phi(\Delta^\ast(\lambda_1 \times \lambda_2))) = \wh{\nabla}^\ast \delta_\ast (-\Phi(\lambda_1 \times \lambda_2)).
	\end{equation*}
    Thus, it would be enough to prove that 
        \begin{equation*}
            \varsigma_\ast((-\Phi(\lambda_1)) \oplus (-\Phi(\lambda_2))) = -\Phi(\lambda_1 \times \lambda_2) \in \Hom_k(\wh{G},\UPic(\bar{X} \times_{\bar{k}} \bar{X})).
        \end{equation*}
    To this end, we are to establish a commutative diagram 
	\begin{equation*}
		\xymatrix@C-5pt{
			[\bar{k}(X)^\times / \bar{k}^\times \xrightarrow{\div} \Div(\bar{X})]^{\oplus 2} \ar[rr]^-{f_1^\ast \oplus f_2^\ast} \ar[d]^{\varsigma} && \bigoplus\limits_{i=1}^2 [\bar{k}(\bar{Y}_i)^\times_{\bar{G}} / \bar{k}^\times \xrightarrow{\div} \Div^{\ver,\bar{G}}(\bar{Y}_i)]^{\oplus 2} \ar[r] \ar[d]^{\pi_1^\ast + \pi_2^\ast} & \wh{G} \oplus \wh{G} \ar@{=}[d] \\
			[\bar{k}(X \times_k X)^\times / \bar{k}^\times \xrightarrow{\div} \Div(\bar{X} \times_{\bar{k}} \bar{X})] \ar[rr]^{(f_1 \times_{\bar{k}} f_2)^\ast} && [\bar{k}(\bar{Y})^\times_{\bar{G} \times_{\bar{k}} \bar{G}} / \bar{k}^\times \xrightarrow{\div} \Div^{\ver,\bar{G} \times_{\bar{k}} \bar{G}}(\bar{Y})] \ar[r] & \wh{G \times_k G}
		}
	\end{equation*}
    of complexes of $\Gamma_k$-modules. The left square commutes because we have a commutative diagram
	\begin{equation*}
		\xymatrix{
			\bar{Y}_1 \ar[d]^{f_1} & \bar{Y} \ar[d]^{f_1 \times_{\bar{k}} f_2} \ar[r]^-{\pi_2} \ar[l]_-{\pi_1} & \bar{Y}_2 \ar[d]^{f_2}\\
			\bar{X} & \bar{X} \times_{\bar{k}} \bar{X} \ar[r]^-{p_2} \ar[l]_-{p_1} & \bar{X}.
		}
	\end{equation*}
    To see that the right square also commutes, take $a_i \in \bar{k}(\bar{Y}_i)^\times_{\bar{G}}$ for $i=1,2$, and let $\chi_i := \chi_{a_i} \in \wh{G}$. We have
	\begin{align*}
		\frac{((\pi_1^\ast a_1)(\pi_2^\ast a_2)) \circ \rho_{(g_1,g_2)}}{(\pi_1^\ast a_1)(\pi_2^\ast a_2)} & = \left(\frac{a_1 \circ \pi_1 \circ \rho_{(g_1,g_2)}}{a_1 \circ \pi_1}\right)\left(\frac{a_2 \circ \pi_2 \circ \rho_{(g_1,g_2)}}{a_2 \circ \pi_2}\right) \\
        & = \left(\frac{a_1 \circ \rho_{g_1} \circ \pi_1}{a_1 \circ \pi_1}\right) \left(\frac{a_2 \circ \rho_{g_2} \circ \pi_2}{a_2 \circ \pi_2}\right) \\
        & = \chi_1(g_1)\chi_2(g_2)
	\end{align*}
    for all $g_1,g_2 \in G(\bar{k})$. This means $\chi_{(\pi_1^\ast a_1)(\pi_2^\ast a_2)} = (\chi_1,\chi_2)$ under the identification $\wh{G\times_k G} = \wh{G} \oplus \wh{G}$.
	
    It remains to check that the map $\pi_1^\ast + \pi_2^\ast$ is a morphism of complexes of {\em $\Gamma_k$-modules}. On the level of vertical invariant divisors, this is obvious thanks to the definition of the Galois action on the divisor groups. On the level of rational functions, take $s \in \Gamma_k$ and $a_i \in \bar{k}(\bar{Y}_i)^\times_{\bar{G}}$, and let $\alpha_i \in E_i$ be a lifting of $s$ for $i=1,2$. The element $\alpha_1 \times_{\bar{k}} \alpha_2 \in E_1 \boxtimes E_2$ is a lifting of $s$, and we have
	\begin{align*}
		(\alpha_1 \times_{\bar{k}} \alpha_2) \cdot ((\pi_1^\ast a_1)(\pi_2^\ast a_2)) & = s_{\#} \circ ((\pi_1^\ast a_1)(\pi_2^\ast a_2)) \circ \alpha_1^{-1} \times_{\bar{k}} \alpha_2^{-1} \\
		& = (s_{\#} \circ (\pi_1^\ast a_1) \circ (\alpha_1^{-1} \times_{\bar{k}} \alpha_2^{-1}))(s_{\#} \circ (\pi_2^\ast a_2) \circ (\alpha_1^{-1} \times_{\bar{k}} \alpha_2^{-1})) \\
		& = (s_{\#} \circ a_1 \circ \pi_1 \circ (\alpha_1^{-1} \times_{\bar{k}} \alpha_2^{-1}))(s_{\#} \circ a_2 \circ \pi_2 \circ (\alpha_1^{-1} \times_{\bar{k}} \alpha_2^{-1})) \\
		& = (s_{\#} \circ a_1 \circ \alpha_1^{-1} \circ \pi_1)(s_{\#} \circ a_2 \circ \alpha_2^{-1} \circ \pi_2) \\
		& = ((\alpha_1 \cdot a_1) \circ \pi_1)((\alpha_2 \cdot a_2) \circ \pi_2) \\
		& = (\pi_1^\ast (\alpha_1 \cdot a_1))(\pi_2^\ast (\alpha_2 \cdot a_2))
	\end{align*}
	in $\bar{k}(\bar{Y})^\times_{\bar{G}}$, or $\tensor[^s]{}{}[(\pi_1^\ast a_1)(\pi_2^\ast a_2)] = (\pi^\ast_1(\tensor[^s]{[a]}{_1}))(\pi^\ast_2(\tensor[^s]{[a]}{_2}))$ in $\bar{k}(\bar{Y})^\times_{\bar{G}}/\bar{k}^\times$, which is what needed to be proved.
\end{proof}

\subsection{Comparing two fundamental exact sequences} \label{subsec:FundamentalComparison}

We prove Theorem \ref{thm:Comparison} in this paragraph. By Proposition \ref{prop:Comparison1} and Lemma \ref{lemm:ActionOfSAutOpenStabilizer} \ref{lemm:ActionOfSAutOpenStabilizer1}, the diagram
\begin{equation*}
	\xymatrix{
		\H^1(X,G) \ar[r]^-{\dtype} \ar@{=}[d] & \DType(X,G) \ar[d]^{\Phi} \\
		\H^1(X,G) \ar[r]^-{\xtype} & \Hom_k(\wh{G},\UPic(\bar{X}))
	}
\end{equation*}
commutes. Once the commutative diagram \eqref{eq:DiagramComparison} is established, its top row is exact if $X$ is quasi-projective, by Theorem \ref{thm:Abelian}. In this case, it follows that $\Phi$ is an isomorphism by the five lemma.

Thus, it remains to establish commutativity up to a sign of the diagram
\begin{equation*}
	\xymatrix{
		\DType(X,G) \ar[r] \ar[d]^{\Phi} & \H^2(k,G) \ar@{=}[d] \\
		\Hom_k(\wh{G},\UPic(\bar{X})) \ar[r]^-{\partial} & \H^2(k,G).
	}
\end{equation*}
Let $\lambda:=[\bar{Y},E]$ be a descent type on $X$ bound by $\lien(G)$. We proceed to show that $\partial(\Phi(\lambda)) = -[E]$. The first reduction step is a localization. By functoriality, we may replace $X$ by any non-empty open subset (this does not affect the class $[E] \in \H^2(k,G)$, {\em cf.} Construction \ref{cons:PullbackOfDescentType}). Since $\Hscr^{-1}(\UPic(\bar{X})) = \Unit(\bar{X})$ and $\Hscr^0(\UPic(\bar{X})) = \Pic(\bar{X})$, we have a distinguished triangle
\begin{equation*}
	\Unit(\bar{X})[1] \to \UPic(\bar{X}) \to \Pic(\bar{X}) \to \Unit(\bar{X})[2]
\end{equation*}
in the category $\Dcal^+(X)$. Applying $\Hom_k(\wh{G},-)$ to it yields an exact sequence
\begin{equation} \label{eq:Comparison11}
	0 \to \Ext_k^1(\wh{G},\Unit(\bar{X})) \to \Hom_k(\wh{G},\UPic(\bar{X})) \to \Hom_k(\wh{G},\Pic(\bar{X}))
\end{equation}
Let us consider the image $\lambda'$ of $\Phi(\lambda)$ in $\Hom_k(\wh{G},\Pic(\bar{X}))$. Since $\wh{G}$ is finitely generated, for some sufficiently small non-empty open subset $U \subseteq X$, the composite $\wh{G} \xrightarrow{\lambda'} \Pic(\bar{X}) \to \Pic(\bar{U})$ is $0$. Replacing $X$ by $U$, we may and shall from now on assume that $\lambda' = 0$. This means we have a short exact sequence
\begin{equation} \label{eq:Comparison12}
	0 \to \Unit(\bar{X}) \to \bar{k}[\bar{Y}]^\times_{\bar{G}}/\bar{k}^\times \to \wh{G} \to 0
\end{equation}
by applying $\Hscr^0$ to \eqref{eq:Comparison10}. Let $\varepsilon \in \Ext^1_k(\wh{G},\Unit(\bar{X}))$ denote its class. It is clear that we have a morphism from \eqref{eq:Comparison12} to \eqref{eq:Comparison10}, hence the map $\Ext_k^1(\wh{G},\Unit(\bar{X})) \to \Hom_k(\wh{G},\UPic(\bar{X}))$ from \eqref{eq:Comparison11} takes $\varepsilon$ to $-\Phi(\lambda)$.

Since $\Unit(\bar{X})$ is a finitely generated free abelian group (see the beginning of this \S), the group $\bar{k}[\bar{Y}]^\times_{\bar{G}}/\bar{k}^\times$ is also finitely generated. Let $T$ (resp. $M$) be the $k$-torus (resp. $k$-group of multiplicative type) with character module $\wh{T} = \Unit(\bar{X})$ (resp. $\wh{M} = \bar{k}[\bar{Y}]^\times_{\bar{G}}/\bar{k}^\times$). We have a commutative diagram with solid arrows
\begin{equation} \label{eq:TMG}
	\xymatrix{
		&& 1 \ar[d] & 0 \ar[d] \\
		1 \ar[r] & \bar{k}^\times \ar@{=}[d] \ar[r] & \bar{k}[X]^{\times} \ar[r] \ar[d]^{f^\ast} & \wh{T} \ar[r] \ar[d] \ar@{.>}@/_1pc/[l]_-{\varsigma_2} & 0 \\
		1 \ar[r] & \bar{k}^\times \ar[r] & \bar{k}[\bar{Y}]_{\bar{G}}^{\times} \ar[r] \ar[d] & \wh{M} \ar[r] \ar[d] \ar@{.>}@/_1pc/[l]_-{\varsigma_1} & 0 \\
		&& \wh{G} \ar[d] \ar@{=}[r] & \wh{G} \ar[d]\\
		&& 0 & 0
	}
\end{equation}
with exact rows and columns, in the category of {\em abelian groups}. Its top row and right column are exact sequences of $\Gamma_k$-modules. Let $\tau \in \Z^1(k,T)$ be a cocycle such that $-[\tau] \in \H^1(k,T)$ corresponds to the class of the top row, under the identification $\H^1(k,T) \cong \Ext^1_k(\wh{T}, \Gbb_m)$ \cite[Proposition 1.4.1]{CTS1987Descente}.

To describe $\partial(\Phi(\lambda))$, we recall how the map $\partial: \Hom_k(\wh{G},\UPic(\bar{X})) \to \H^2(k,G)$ was defined. According to the proof of \cite[Proposition 8.1]{HS2013Descent}, the bottom row of \eqref{eq:DiagramComparison} is obtained by applying $\Hom_k(\wh{G},-)$ to the triangle
\begin{equation*}
	\Gbb_m \to \tau_{\le 1} \Rbb p_\ast \Gbb_{m,X} \to \UPic(\bar{X})[-1] \to \Gbb_m[1],
\end{equation*}
or equivalently, the extension of complexes
\begin{equation} \label{eq:Comparison13}
	0 \to \bar{k}^\times \to [\bar{k}(X)^\times \xrightarrow{\div} \Div(\bar{X})] \to [\bar{k}(X)^\times/\bar{k}^\times \xrightarrow{\div} \Div(\bar{X})] \to 0.
\end{equation}
Thus, under identification $\H^2(k,G) \cong \Ext^2_k(\wh{G}, \Gbb_m)$ \cite[Proposition 1.4.1]{CTS1987Descente}, the map $\partial$ is given by the Yoneda product with the class of \eqref{eq:Comparison13}. There is a commutative diagram of Yoneda pairings
\begin{equation*}
	\xymatrix@C-2pc{
		\Ext^1_k(\wh{G},\wh{T}) \ar[d] & \times & \Ext^1_k(\wh{T},\Gbb_m) \ar[rrrrrrrr]^-{\sqdot} &&&&&&&&  \Ext^2_k(\wh{G},\Gbb_m) \ar@{=}[d] \\
		\Hom_k(\wh{G},\UPic(\bar{X})) & \times & \Ext^2_k(\UPic(\bar{X}),\Gbb_m) \ar[rrrrrrrr]^-{\sqdot} \ar[u] &&&&&&&& \Ext^2_k(\wh{G},\Gbb_m)  
	}
\end{equation*}  
Viewing that there is an obvious inclusion from the top row of \eqref{eq:TMG} to \eqref{eq:Comparison13}, the above diagram gives us $\partial(\Phi(\lambda)) =  \varepsilon \sqdot [\tau]$. Alternatively, we have $\partial(\Phi(\lambda)) = \delta([\tau])$, where $\delta: \H^1(k,T) \to \H^2(k,G)$ is the connecting homomorphism issued from the exact sequence of $k$-groups of multiplicative type
\begin{equation} \label{eq:GMT}
	1 \to G \to M \to T	\to 1
\end{equation}
dual to the right column of \eqref{eq:TMG}. The task now is to show that $\delta([\tau]) = -[E]$.

Let us now take the dotted arrows in diagram \eqref{eq:TMG} into consideration. Since $\bar{k}^\times$ is a divisible abelian group, the bottom row has a section $\varsigma_1: \wh{M} \to \bar{k}[\bar{Y}]^\times_{\bar{G}}$. By exactness of the middle column, the composite $\wh{T} \to \wh{M} \xrightarrow{\varsigma_1} \bar{k}[\bar{Y}]^\times_{\bar{G}}$ factors through a section $\varsigma_2: \wh{T} \to \bar{k}[X]^\times$ (in the category of abelian groups). Now, the section $\varsigma_1$ (resp. $\varsigma_2$) yields a homomorphism $\pi^\ast: \bar{k}[M] \to \bar{k}[\bar{Y}]$ (resp. $\varpi^\ast: \bar{k}[T] \to \bar{k}[X]$) of $\bar{k}$-algebras (see {\em cf.} paragraph \ref{subsec:Notation}). Since $\bar{M}$ and $\bar{T}$ are affine, we obtain $\bar{k}$-morphisms $\pi: \bar{Y} \to \bar{M}$ and $\varpi: \bar{X} \to \bar{T}$ fitting in a commutative diagram
\begin{equation} \label{eq:YMXT}
	\xymatrix{
		\bar{Y} \ar[r]^{\pi} \ar[d]^{f} & \bar{M} \ar[d] \\
		\bar{X} \ar[r]^{\varpi} & \bar{T},}
\end{equation}
where both vertical arrows are torsors under $\bar{G}$, and $\pi$ is $\bar{G}$-equivariant. 

One should keep in mind that $\pi$ induces a homomorphism (also denoted by) 
\begin{equation*}
	\pi^\ast: \bar{k}[M]^\times_{M} \to \bar{k}[\bar{Y}]^\times_{\bar{G}}
\end{equation*}
of abelian groups, which fits in a commutative diagram with exact rows
\begin{equation} \label{eq:DefinitionOfM}
	\xymatrix{
		1 \ar[r] & \bar{k}^\times \ar[r] \ar@{=}[d] & \bar{k}[M]^\times_M \ar[r]^-{b \mapsto \chi_b} \ar[d]^{\pi^\ast} & \wh{M} \ar@{=}[d] \ar[r] & 0 \\
		1 \ar[r] & \bar{k}^\times \ar[r] & \bar{k}[\bar{Y}]^\times_{\bar{G}} \ar[r]^-{a \mapsto [a]} & \bar{k}[\bar{Y}]^\times_{\bar{G}}/\bar{k}^\times \ar[r] & 0
	}
\end{equation}
(the top row is exact because it is \eqref{eq:Comparison4} applied to the trivial torsor $M \to \Spec(k)$). By the five lemma, $\pi^\ast$ is an isomorphism, which induces identity on the Galois module $\wh{M}$. Considering the action $\cdot$ of $\SAut_G(\bar{Y}/X)$ introduced in \eqref{eq:ActionOfSAut}, we have $\tensor[^s]{[a]}{} = [\alpha \cdot a] \in \bar{k}[\bar{Y}]^\times_{\bar{G}}/\bar{k}^\times = \wh{M}$ for any $s \in \Gamma_k$, any lifting $\alpha \in E$ of $s$, and any $a \in \bar{k}[\bar{Y}]^\times_{\bar{G}}$. Conversely, we shall need the following Lemma, whose proof is postponed to the end of this \S.

\begin{lemm} \label{lemm:CharacterizationofE}
	Let $s \in \Gamma_k$ and let $\alpha \in \SAut_G(\bar{Y}/X)$ be a lifting of $s$. Then $\alpha \in E$ if and only if $\tensor[^s]{[a]}{} = [\alpha \cdot a]$ in $\bar{k}[\bar{Y}]^\times_{\bar{G}}/\bar{k}^\times = \wh{M}$ for all $a \in \bar{k}[\bar{Y}]^\times_{\bar{G}}$.
\end{lemm}

\begin{proof} [End of proof of Theorem \ref{thm:Comparison}, modulo Lemma \ref{lemm:CharacterizationofE}]
	Let $Z$ be the $k$-torsor under $T$ defined by the cocycle $\tau$. This means $\bar{Z} = \bar{T}$ as a $\bar{k}$-variety, and the action $\cdot_{\tau}$ of $\Gamma_k$ on $\bar{Z}$ is given by $s \mapsto \rho_{\tau_s}^{-1} \circ s_{\#}$ (where $s_{\#}: T \to T$ is the usual Galois action). By \cite[Propositions 1.4.1, 1.4.2 and 1.4.3]{CTS1987Descente}, the class in $\Ext^1_k(\wh{T},\Gbb_m) \cong \H^1(k,T)$ of the extension 
	\begin{equation*}
		1 \to \bar{k}^\times \to \bar{k}[Z]^\times \to \wh{T} \to 0,
	\end{equation*}
	is $-[\tau]$, that is, the class of the top row of \eqref{eq:TMG}. This means  $\varpi^\ast: \bar{k}[Z] \to \bar{k}[X]$ is in fact a homomorphism of $k$-algebra (because it induces an isomorphism $\bar{k}[Z]^\times \to \bar{k}[X]^{\times}$ of $\Gamma_k$-modules thanks to the five lemma, and $\bar{k}[Z]^\times$ generates the $\bar{k}$-vector space $\bar{k}[Z]^\times$). In other words, $\varpi$ descends into a $k$-morphism $X \to Z$.
	
	Recall that our objective is to show the identity $\delta([\tau]) = -[E] \in \H^2(k,G)$. If $\mu: \Gamma_k \to M(\bar{k})$ is a cochain lifting $\tau$, then $\delta([\tau])$ is the class of the $2$-cocycle 
	\begin{equation} \label{eq:2CocycleMLiftingT}
		g: \Gamma_k \times \Gamma_k \to G(\bar{k}), \quad g_{s,t}:=\mu_s \tensor[^s]{\mu}{_t}\mu_{st}^{-1}.
	\end{equation}
	For each $s \in \Gamma_k$, let $\beta_s:=\rho_{\mu_s}^{-1} \circ s_{\#}$. Since $\mu_s$ is a lifting of $\tau_s$, one sees that $\beta_s \in \SAut_G(\bar{M}/Z)$ is $s$-semilinear. From diagram \eqref{eq:YMXT}, we obtain a morphism $\psi = (f,\pi): \bar{Y} \to \bar{X} \times_{\bar{Z}} \bar{M}$ of $\bar{X}$-torsors under $\bar{G}$, which, like any morphism of torsors, is an isomorphism. Let $\pr: \bar{X} \times_{\bar{Z}} \bar{M} \to \bar{M}$ denote the second projection, then $\pr \circ \psi = \pi$. Let
	\begin{equation*}
		\alpha_s := \psi^{-1} \circ (s_{\#} \times_{\bar{Z}} \beta_s) \circ \psi \in \SAut_G(\bar{Y}/X),
	\end{equation*}
	which is $s$-semilinear. For each $a \in \bar{k}[\bar{Y}]^\times_{\bar{G}}$, write it as $a = \pi^\ast b = b \circ \pi$ for a unique $b \in \bar{k}[M]^\times_{M}$ and let $\chi_b = [a] \in \wh{M} = \bar{k}[\bar{Y}]^\times_{\bar{G}}/\bar{k}^\times$ be the character of $M$ associated with $b$ (using the notations from diagram \eqref{eq:DefinitionOfM}). We have
	\begin{align*}
		\alpha_s \cdot a & = s_{\#} \circ b \circ \pi \circ \psi^{-1} \circ (s_{\#} \times_{\bar{Z}} \beta_s)^{-1} \circ \psi, & \text{by \eqref{eq:ActionOfSAut}}, \\
		& = s_{\#} \circ b \circ \pr \circ (s_{\#} \times_{\bar{Z}} \beta_s)^{-1} \circ \psi, & \text{since } \pr \circ \psi = \pi, \\
		& = s_{\#} \circ b \circ \beta_s^{-1} \circ \pr \circ \psi, & \text{since } \beta_s \circ \pr = \pr \circ (s_{\#} \times_{\bar{Z}} \beta_s), \\
		& = s_{\#} \circ b \circ s_{\#}^{-1} \circ \rho_{\mu_s} \circ \pi, & \text{since } \pr \circ \psi = \pi \text{ and } \beta_s = \rho_{\mu_s}^{-1} \circ s_{\#},\\
		& = \tensor[^s]{b}{} \circ \rho_{\mu_s} \circ \pi \\
		& = (\tensor[^s]{b}{} \chi_{\tensor[^s]{b}{}}(\mu_s)) \circ \pi, & \text{since } \tensor[^s]{b}{} \in \bar{k}[M]^\times_M, \\
		& = (\pi^\ast(\tensor[^s]{b}{}))\chi_{\tensor[^s]{b}{}}(\mu_s).
	\end{align*}
	It follows that $[\alpha_s \cdot a] = [\pi^\ast(\tensor[^s]{b}{})] = \chi_{\tensor[^s]{b}{}} = \tensor[^s]{\chi}{_b} = \tensor[^s]{[a]}{}$ in $\bar{k}[\bar{Y}]^\times_{\bar{G}}/\bar{k}^\times = \wh{M}$. By Lemma \ref{lemm:CharacterizationofE}, we have $\alpha_s \in E$. To conclude, we have constructed a continuous set-theoretic section $\alpha: \Gamma_k \to E$. Finally, for $s,t \in \Gamma_k$, one has
	\begin{equation*}
		\beta_{st} \circ \beta_t^{-1} \circ \beta_s^{-1} = \rho_{\mu_{st}}^{-1} \circ (st)_{\#} \circ t_{\#}^{-1} \circ \rho_{\mu_t} \circ s_{\#}^{-1} \circ \rho_{\mu_s} = \rho_{\mu_{st}}^{-1} \circ s_{\#} \circ t_{\#} \circ t_{\#}^{-1} \circ s_{\#}^{-1} \circ \rho_{\tensor[^s]{\mu}{_t}} \circ \rho_{\mu_s} = \rho_{g_{s,t}}
	\end{equation*}
	using \eqref{eq:2CocycleMLiftingT}. It follows that $\alpha_{st} \circ \alpha_t^{-1} \circ \alpha_s^{-1} = \rho_{g_{s,t}} = \rho_{g_{s,t}^{-1}}^{-1}$ since $\psi$ is a $\bar{G}$-equivariant $\bar{X}$-isomorphism. From this, we conclude that $[E] = -[g] = -\delta([\tau])$, which achieves the proof.
\end{proof}

To complete the proof of Theorem \ref{thm:Comparison}, it remains to prove Lemma \ref{lemm:CharacterizationofE}.

\begin{proof}[Proof of Lemma \ref{lemm:CharacterizationofE}]
    Let $E' \subseteq \SAut_G(\bar{Y}/X)$ be the subgroup consisting of $G$-semiequivariant automorphisms $\alpha$ for which $\tensor[^s]{[a]}{} = [\alpha \cdot a]$ in $\bar{k}[\bar{Y}]^\times_{\bar{G}}/\bar{k}^\times$ for all $a \in \bar{k}[\bar{Y}]^\times_{\bar{G}}$, where $s \in \Gamma_k$ denotes the image of $\alpha$. We have seen that $E \subseteq E'$. In order to show that $E = E'$, it suffices to check that the sequence $1 \to G(\bar{k}) \to E' \to \Gamma_k \to 1$	is exact. Of course, only exactness at the middle term $E'$ needs verification. Let $\alpha \in E'$ whose image in $\Gamma_k$ is $1$. Since the sequence $1 \to G(\bar{X}) \to \SAut_G(\bar{Y}/X) \to \Gamma_k \to 1$ is exact, $\alpha$ has the form $\alpha(y) = y \cdot \sigma(f(y))^{-1}$ for some $\bar{k}$-morphism $\sigma: \bar{X} \to \bar{G}$. By definition, $[\alpha \cdot a] = [a]$ in $\bar{k}[\bar{Y}]^\times_{\bar{G}}/\bar{k}^\times$ for all $a \in \bar{k}[\bar{Y}]^\times_{\bar{G}}$. Hence, we have a homomorphism
	\begin{equation*}
		c:\bar{k}[\bar{Y}]^\times_{\bar{G}} \to \bar{k}^\times, \quad c(a):=\frac{\alpha \cdot a}{a}.
	\end{equation*}
	Let $x \in X(\bar{k})$ and lift it to a point $y \in \bar{Y}(\bar{k})$. For any $a \in \bar{k}[\bar{Y}]^\times_{\bar{G}}$, one has 
	\begin{equation*}
		c(a) = \frac{(\alpha \cdot a)(y)}{a(y)} = \frac{a(\alpha^{-1}(y))}{a(y)} = \frac{a(y \cdot \sigma(x))}{a(y)} = \chi_a(\sigma(x)).
	\end{equation*}
	We have assumed that the map last arrow in \eqref{eq:Comparison11} takes $\Phi(\lambda)$ to $0$, hence we have an exact sequence 
	\begin{equation*}
		1 \to \bar{k}[X]^\times \to \bar{k}[\bar{Y}]^\times_{\bar{G}} \xrightarrow{a \mapsto \chi_a} \wh{G} \to 0
	\end{equation*}
	of abelian groups. By the above calculation, $c$ factors through a homomorphism $\bar{c}: \wh{G} \to \bar{k}^\times$ defined by $\bar{c}(\chi) = \chi(\sigma(x))$ for {\em any} $x \in X(\bar{k})$. Since the pairing $G(\bar{k}) \times \wh{G} \to \bar{k}^\times$ is perfect, there exists a unique $g \in G(\bar{k})$ such that $\chi(g) = \chi(\sigma(x))$ for all $x \in X(\bar{k})$ and $\chi \in \wh{G}$. Also by perfectness, this implies that $\sigma$ is the constant $\bar{k}$-morphism $x \mapsto g$. In other words, $\alpha = \rho_g^{-1}$, the element of $\Aut_{\bar{G}}(\bar{Y}/\bar{X})$ corresponding to $g \in G(\bar{k})$. The lemma is proved.
\end{proof}

\bibliographystyle{alpha}
\bibliography{ref}
\end{document}